\documentclass[reqno]{amsart}
\usepackage[title]{appendix}
\usepackage{hyperref}
\usepackage{amsmath}
\usepackage{amssymb}
\usepackage{graphicx}
\usepackage{tikz}
\usetikzlibrary{calc}
\usepackage{xcolor}
\usetikzlibrary{snakes}
\usepackage{amsthm}
\usepackage{tikz-3dplot}
\usepackage{marvosym}
\usepackage{xcolor}

\definecolor{armygreen}{rgb}{.0, 0.5, 0.0}

\newcommand{\BBB}{\color{black}} 
 
\newcommand{\EEE}{\color{black}}

\newcommand{\GGG}{} 

\theoremstyle{plain}
\begingroup
\newtheorem{theorem}{Theorem}[section]
\newtheorem{lemma}[theorem]{Lemma}
\newtheorem{proposition}[theorem]{Proposition}
\newtheorem{corollary}[theorem]{Corollary}
\endgroup

\theoremstyle{definition}
\begingroup
\newtheorem{definition}[theorem]{Definition}
\newtheorem{remark}[theorem]{Remark}

\endgroup

\newenvironment{step}[1]{\underline{Step #1}.}{}

\usepackage[english]{babel}
\begin{document}

\author[M.~Cicalese]{Marco Cicalese}
\address[Marco Cicalese]{Zentrum Mathematik - M7, Technische Universit\"at M\"unchen, Garching, Germany}
\email{cicalese@ma.tum.de}

\author{Leonard Kreutz}
\address[Leonard Kreutz]{Applied Mathematics M\"unster,
WWU M\"unster, 48149 M\"unster, Germany}
\email{lkreutz@uni-muenster.de}

\author[G.P.~Leonardi]{Gian Paolo Leonardi}
\address[G.P.~Leonardi]{Department of Mathematics, University of Trento, 38050 Povo (TN), Italy}
\email{gianpaolo.leonardi@unitn.it}

\title{Emergence of Wulff-Crystals from atomistic systems on the FCC and HCP lattices}
\begin{abstract} We consider a system of $N$ hard spheres sitting on the nodes of either the $\mathrm{FCC}$ or $\mathrm{HCP}$ lattice and interacting via a sticky-disk potential. As $N$ tends to infinity (continuum limit), assuming the interaction energy does not exceed that of the ground-state by more than $N^{2/3}$ (surface scaling), we obtain the variational coarse grained model by $\Gamma$-convergence. More precisely, we prove that the continuum limit energies are of perimeter type and we compute explicitly their Wulff shapes. Our analysis shows that crystallization on $\mathrm{FCC}$ is preferred to that on $\mathrm{HCP}$ for $N$ large enough. The method is based on integral representation and concentration-compactness results that we prove for general periodic lattices in any dimension.
\end{abstract}

\keywords{Wulff shape, isoperimetric inequality, crystallization, $\Gamma$-convergence, lattice systems, concentration-compactness}
\maketitle
\section{Introduction}

A fundamental problem in crystallography is to understand why ensembles of large number of atoms arrange themselves into crystals at low temperatures. From the mathematical point of view, proving that equilibrium configurations of certain phenomenological interaction energies exhibit these structures is referred to as the crystallization problem \cite{Blanc}.

At zero temperature the internal energy of a configuration of atoms is expected to be solely governed by its geometric arrangement. \BBB Within \EEE the framework of molecular mechanics \cite{Molecular,Friesecke-Theil15,Lewars}, \BBB one identifies \EEE each ensemble of atoms with its atomic positions $X=\{x_1,\ldots,x_N\} \subset \mathbb{R}^3$ and \BBB associates \EEE to it a configurational energy of the form
\begin{align*}
\mathcal{E}(X) := \frac{1}{2}\sum_{i\neq j} V(|x_i-x_j|)\,,
\end{align*}
where $V \colon \mathbb{R} \to \mathbb{R}\cup \{+\infty\}$ is an empirical pair interaction potential (the factor $\frac{1}{2}$ accounts for double counting). Such potentials are typically repulsive at short distances and attractive at large distances. While clustering is favored by long range attraction, the density of a cluster cannot get too large due to short-range repulsion.

Notably, even under simplifying assumptions on the interaction potentials, the mathematical literature on rigorous crystallization results is scarce. In fact, for finite $N$, only results in one and two space dimensions are available. For example, if $V$ is of Lennard--Jones type, crystallization has been proved  only in one space dimension \cite{GardnerRadin:79}. In higher space dimensions only partial results are available. Most notably, in \cite{ELi:09,FlatleyTheil:15,Theil:06} it has been proven that crystalline structures have optimal bulk energy scaling.  In two dimensions, only results for (some variants of) the sticky disc potential (see Fig.~\ref{fig:potential})
\begin{align}\label{def:PHR}
V(r):= \begin{cases} 
+\infty &\text{if } r <1\,,\\
-1 &\text{if } r=1\,,\\
0 &\text{otherwise}
\end{cases}
\end{align}
 are available \cite{Lucia,Harborth:74,HeitmannRadin,Radin:81}. 
\begin{figure}
\begin{tikzpicture}
\draw[->](0,-1.5)--++(0,4.5) node[anchor= east] {$V(r)$};
\draw[->](-1,0)--++(5,0) node[anchor =north] {$r$};
\draw[decorate, decoration={snake,amplitude=.4mm,segment length=1.5mm}] (0,2) node[anchor =east]{$+\infty$}--++(1,0);
\draw[dashed,thin](1,-1) --++(0,1) node[anchor =north east] {$ 1 $}--++(0,2);
\draw[fill=black](0,-1) node[anchor = east] {$-1$}++(1,0) circle(.025);
\draw[thick](1,0)--++(2.4,0);
\draw[thick](1,-1)circle(.025);
\end{tikzpicture}
\caption{The sticky disc interaction potential $V$.}
\label{fig:potential}
\end{figure}
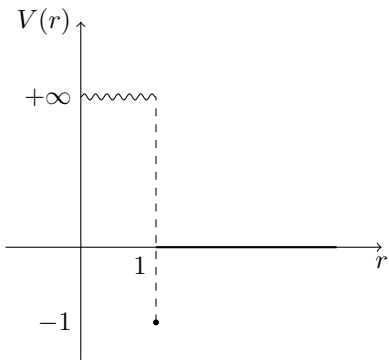
More recently, crystallization results have been proved for ionic compounds \cite{FriedrichKreutz,FriedrichKreutz2} and carbon structures in \cite{Mainini}. The potential given in \eqref{def:PHR} models the atoms as hard spheres that interact exactly when two of them are tangent.  In $\mathbb{R}^n$ the kissing number $k(n)$ is the highest number of $n$-dimensional spheres of radius $\frac{1}{2}$ which are tangent to a given sphere of the same size. It is well known that $k(2)=6$ and $k(3)=12$, see \cite{SchVan}. For a given configuration of non-overlapping equal balls centered at $X=\{x_1,\ldots,x_N\} \subset \mathbb{R}^n$, $N \in \mathbb{N}\cup\{+\infty\}$ the coordination number of $x \in X$ is the number of spheres centered at $y\in X\setminus \{x\}$ and tangent to the one centered at $x$. In two dimensions there is a unique (up to a rigid motion) configuration made of infinitely many particles such that all atoms have as coordination number the kissing number. Such a set $X$ is the triangular lattice with lattice spacing one. In three dimensions the problem is much more intricate. In fact, there exist infinitely many configurations with constant coordination number equal to $k(3)$. An infinite class of configurations can be obtained by stacking in an appropriate way layers of triangular lattices. A remarkable result by Hales \cite{Hales} shows that all such structures solve Kepler's conjecture which is to say that they have the maximal packing density in $\mathbb{R}^3$. Two notable cases of the aforementioned structures are the face-centered cubic lattice $\mathcal{L}_{\mathrm{FCC}}$ and the hexagonal closed-packed lattice $\mathcal{L}_{\mathrm{HCP}}$ (see \eqref{def:vectors}--\eqref{def: HCP} for their precise definition) which are the most prevalent among the crystalline arrangements in the periodic table of elements.

 In this paper we want to investigate already crystallized configurations, i.e.~configurations $X \subset \mathcal{L}$ where $\mathcal{L}=\mathcal{L}_{\mathrm{FCC}}$ or $ \mathcal{L}=\mathcal{L}_{\mathrm{HCP}}$. For such $X = \{x_1,\ldots,x_N\} \subset \mathcal{L}$,  fixing the lattice spacing to be $1$, we have 
\begin{align*}
\mathcal{E}(X)=\frac{1}{2}\sum_{i\neq j} V(|x_i-x_j|) = -\sum_{i=1}^N\#(\mathcal{N}(x_i)\cap X)\,, \text{ where } \mathcal{N}(x)= \{y \in \mathcal{L} \colon |x-y|=1\}\,.
\end{align*}
As described above the minimal energy per atom is $-k(3)=-12$. Further information on $\mathcal{E}$ as $N$ grows can be obtained by referring it to the minimal energy per atom and calculating the  {\it excess energy} \BBB $E_N(X)$ defined below\EEE. More precisely, in Theorem \ref{theorem:Gammaconvergence}, we carry out a rigorous variational asymptotic expansion (see \cite{BraTru}) of $\mathcal{E}(X)$, by considering
\begin{align}\label{introeq:rescaled energy}
E_N(X) =  N^{-2/3}\left( \mathcal{E}(X)+12N\right) = N^{-2/3}\sum_{i=1}^N(12-\#(\mathcal{N}(x)\cap X))\,
\end{align}
and calculating its $\Gamma$-limit \cite{Braides:02,DalMaso} as $N$ tends to infinity. This analysis has been done in two dimensions for configurations confined to the triangular lattice \cite{Yuen} as well as without any confinement assumption \cite{FriedrichKreutzSchmidt}. Note that, the scaling factor $N^{-2/3}$ is used in order to keep the energy bounded as the number of atoms grows. In fact, given a low energy configuration of $N$ atoms, the number of those contributing to the energy scales like $N^{2/3}$ for $N$ large. By associating to each configuration its rescaled emprical measure 
\begin{align*}
\mu_N(X) :=\BBB \frac{1}{N}\sum_{i=1}^N \delta_{N^{-1/3}x_i}\EEE\,,
\end{align*}
we show in Theorem \ref{theorem:Gammaconvergence} i) that the sequence of rescaled energies \eqref{introeq:rescaled energy} is equi-coercive with respect to the weak*-convergence of the associated empirical measures. In Theorem \ref{theorem:Gammaconvergence} ii), iii) we exploit integral representation theorems \cite{ABCS,AliCicRuf,AlicandroGelli} to show that the limit energy is finite on the set of measures $\mu = \sqrt{2}\mathcal{L}^3\llcorner V$, where $V \subset \mathbb{R}^3$ is a set of finite perimeter, on which the energy takes the form
\begin{align}\label{def:limitenergy}
E_\mathcal{L}(\mu) := \int_{\partial^* V} \varphi_\mathcal{L}(\nu)\,\mathrm{d}\mathcal{H}^2 \,.
\end{align}
 Here, $\partial^* V$ denotes the reduced boundary of the set $V$, $\nu(x)$ denotes its unit outer normal at the point $x\in \partial^*V$ and $\varphi_{\mathcal{L}}$ is an anisotropic surface energy density depending on the underlying lattice $\mathcal{L}$.  \GGG In the case of multi-lattices, like the HCP-lattice, this integral representation result has not yet been proven in the literature. We defer to Section~\ref{appendix:general} for a proof of this result  whose main ingredient is the integral representation theorem in \cite{AliCicRuf}. Furthermore, in the same section we prove general compactness and concentration lemmata that ensure the convergence of the rescaled empircal measures of minimizers of the discrete problem \eqref{introeq:rescaled energy} to the Wulff shape (up to a constant density factor) of the associated limiting anisotropic perimeter energy \eqref{def:limitenergy}. Such kind of result was previously known only in two dimensions \cite{Yuen}. Its extension to higher dimensions, see Lemma \ref{lemma:concentration}, requires  more refined tools from geometric measure theory that, to the best of our knowledge, are exploited in this setting here for the first time. \EEE
The main body of this work lies in the calculation of the surface energy density $\varphi_{\mathcal{L}} \colon \mathbb{R}^3 \to [0,+\infty)$ both for the $\mathrm{FCC}$ and the $\mathrm{HCP}$ lattices. Here, we take advantage of a recently proved finite cell formula \cite{ChaKre}. Finally, for both lattices, we solve the associated isoperimetric problem \BBB\cite{FonsecaMueller} \EEE
\begin{align}\label{introeq:isoperimetric}
m_{\mathcal{L}}:=\min\left\{\int_{\partial^* V} \varphi_{\mathcal{L}}(\nu)\,\mathrm{d}\mathcal{H}^2 \colon |V|=1\right\}\,
\end{align}
by calculating the (up to translation unique) set realizing the minimum in \eqref{introeq:isoperimetric}, also known as the Wulff shape \BBB\cite{Wulff}\EEE. We show that $m_{\mathrm{FCC}} < m_{\mathrm{HCP}}$ which also implies (since $\Gamma$-convergence and coercivity implies the convergence of minimum values) that, for large number of atoms, crystallization on the face-centered cubic lattice  is preferred to that on  the hexagonal-closed packed lattice. \GGG We finally mention \cite{Yuen-thesis} for some preliminary computations on the Wulff shape of the FCC and HCP. \EEE

\GGG In contrast to the uniqueness of the Wulff crystal in the continuum setting, minimizers to the discrete isoperimetric problem \cite{Har} are non-unique \cite{Lucia2}. Over the last years there has been a remarkable interest in establishing fluctuation estimates between different minimizers, i.e., estimating (several notions of) distances between different minimizers. Maximal fluctuation estimates between two minimizers have been
first conjectured in \cite{Yuen} in the case of the crystallization on the triangular lattice and have been later proved
in \cite{Schmidt} and \cite{DPS2}. The same estimates have been proved in \cite{Mainini-Piovano,FriedrichKreutz} and \cite{DPS,FriedrichKreutz2} for the square and the honeycomb lattices, respectively. A general approach linking the quantitative anistropic isoperimetric inequality to such fluctuation estimates has been set up in \cite{cicaleseleonardi} by two of the authors. In dimensions larger than two these fluctuation estimates have been only established for the cubic lattice in \cite{MaiPioSteSch} and for $\mathbb{Z}^d$ in \cite{MaininiSchmidt:20}. In order to establish the aforementioned fluctuation estimates, however, an understanding of the limiting macroscopic Wulff shape is essential. Since the present work yields these shapes for the FCC and HCP lattices, it is our opinion that it may be considered an indispensable first step to prove fluctuation estimates also for such lattices.   \EEE

The article is structured as follows. In Section \ref{sec:Notation}, we introduce the necessary mathematical preliminaries, the model, and the main results. In Section \ref{sec:proofs} we prove Proposition \ref{prop:main1} and \ref{prop:main2}, by calculating the surface energy density as well as the Wulff crystal associated to both the $\mathrm{FCC}$ and the $\mathrm{HCP}$ lattices. \GGG In Section \ref{sec:Gammaconvergence} we prove the main $\Gamma$-convergence Theorem \ref{theorem:Gammaconvergence}. The latter is a consequence of a more general theory for discrete perimeter energies on general periodic lattices developed in Section~\ref{appendix:general}. \EEE

\section{Setting and Notation}\label{sec:Notation}

\GGG Given a set of vectors $V\subset \mathbb{R}^n$ we denote by $\mathrm{span}_\mathbb{Z}V$ the set of finite linear combinations of elements of $V$ with coefficients in $\mathbb{Z}$. We denote by $\mathfrak{M}$ the collection of all Lebesgue measurable subsets of $\mathbb{R}^n$. Given $A\in \mathfrak{M}$ we denote by $|V|$ its $n$-dimensional Lebesgue measure, i.e., $|V|=\mathcal{L}^n(V)$, and $\mathcal{H}^n$ its $n$-dimensional Hausdorff measure. Given a countable set $X$, we denote by $\#X$
the cardinality of $X$. Given $a,b \in \mathbb{R}^n$ we denote by $\langle a,b\rangle$ their scalar product. We denote by $\mathbb{S}^{n-1}$ the set of unitary vectors in $\mathbb{R}^n$. For any $\nu \in \mathbb{S}^{n-1}$ let $\{\nu_1,\ldots,\nu_n=\nu\}$ be an orthonormal basis of $\mathbb{R}^n$ , and let $Q^\nu := \{x \in \mathbb{R}^n \colon
|\langle x, \nu_i\rangle| < 1/2,\, i = 1,\ldots,n\}$ be a unit cube centered at the origin with faces parallel and orthogonal to $\nu$. For $T>0$ and $x\in \mathbb{R}^n$ we set  $Q_T^\nu(x)= x+TQ^\nu$ and we write $Q_T^\nu=Q_T^\nu(0)$. For $r>0$ and $x\in \mathbb{R}^n$ we denote by $B_r(x)$ the $n$-dimensional Euclidean ball of radius $r$ centered at $x$ (for $x=0$ we write $B_r$ in place of $B_r(0)$) and we set $\omega_n=|B_1(x)|$. For $r>0$ and $A \subset \mathbb{R}^n$ we set $(A)_r=B_r +A$. Given $A \subset \mathbb{R}^n$ open,  we define the set of positive Radon measures by $\mathcal{M}_+(A)$. We say that $\{\mu_k\}_k\subset \mathcal{M}_+(A)$ converges to $\mu \in \mathcal{M}_+(A)$ with respect to the weak star topology and we write $\mu_k \overset{*}{\rightharpoonup} \mu$ if
\begin{align*}
\lim_{k\to \infty}\int_{A} \varphi \,\mathrm{d}\mu_k = \int_{A} \varphi \,\mathrm{d}\mu \text{ for all } \varphi \in C_c(A)\,.
\end{align*}
we denote by $BV(A)$ the space of functions of bounded variation in $A$ and we denote by $BV_{\mathrm{loc}}(A)=\{u \colon u \in BV(K) \text{ for all } K \subset\subset A, K \text{ open}\}$. Given a  function $u \in BV(A)$ we use the notation of \cite{AFP} for the jump set $J(u)$  and the measure theoretic normal $\nu_u \colon J(u) \to \mathbb{S}^{n-1}$. For $V\subset A$, $V\in \mathfrak{M}$ we denote the relative perimeter of $V$ in $A$ by
\begin{align*}
\mathrm{Per}(V,A) = \sup \left\{\int_V \mathrm{div}\, v\,\mathrm{d}x \colon v \in C_c^\infty(A;\mathbb{R}^n), \|v\|_\infty\leq 1\right\}\,.
\end{align*}
For $\nu \in \mathbb{S}^{n-1}$ we set
\begin{align}\label{def:unu}
u_\nu(i):=\begin{cases} 1 &\text{if } \langle x,\nu\rangle \geq 0\,;\\
0 &\text{otherwise.}
\end{cases}
\end{align}
In Section~\ref{sec:Notation}--Section~\ref{sec:Gammaconvergence} we set $n=3$.
 \EEE

\subsection*{Definition of $\mathrm{HCP}$ and $\mathrm{FCC}$ lattices} \label{sec:HCPFCC} In the following we define the \emph{face-centered cubic lattice} (short $\mathrm{FCC}$-lattice) and the \emph{hexagonal closed-packed lattice} (short $\mathrm{HCP}$-lattice). To this end, we introduce the verctors
\begin{align}\label{def:vectors2}
b_1 :=\frac{1}{\sqrt{2}} \begin{pmatrix}
1\\
1\\
0
\end{pmatrix}\,, \quad b_2 :=\frac{1}{\sqrt{2}} \begin{pmatrix}
1\\
0\\
1
\end{pmatrix}\,, \quad b_3 :=\frac{1}{\sqrt{2}} \begin{pmatrix}
0\\
1\\
1
\end{pmatrix}\,
\end{align}
and 
\begin{align}\label{def:vectors}
e_1 := \begin{pmatrix}
1\\
0\\
0
\end{pmatrix}\,, \quad e_2 :=\frac{1}{2}\begin{pmatrix}
1\\
\sqrt{3}\\
0
\end{pmatrix}\,, \quad e_3 :=\frac{2}{3}\sqrt{6}\begin{pmatrix}
0\\
0\\
1
\end{pmatrix}\,, \quad v_1 :=\frac{1}{3}\left(e_1+e_2\right) +\frac{1}{2}e_3\,.
\end{align}
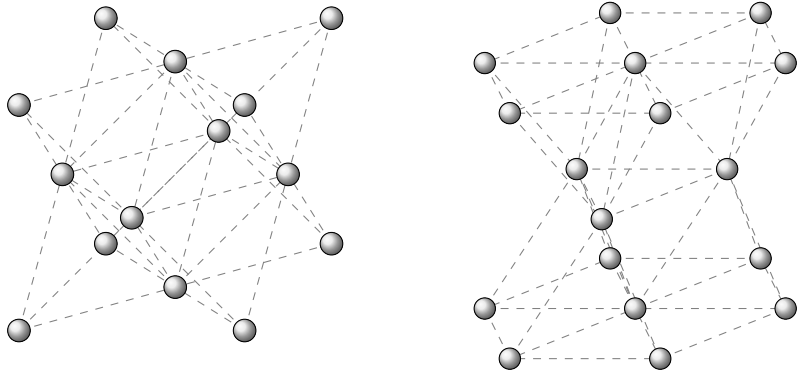
\begin{figure}
\begin{tikzpicture}

 \begin{scope}[shift={(0,0,.5)},scale=3]
\draw[gray,dashed](0,0,0)--++(1,1,0);

\draw[gray,dashed](1,0,0)--(0,1,0);

\draw[gray,dashed](0,0,0)--(1,0,-1);
\draw[gray,dashed](1,0,0)--(0,0,-1);
\draw[gray,dashed](1,0,0)--(1,1,-1);
\draw[gray,dashed](1,0,-1)--(1,1,0);
\draw[gray,dashed](0,0,0)--(0,1,-1);
\draw[gray,dashed](0,1,0)--(0,0,-1);
\draw[gray,dashed](0,1,0)--(1,1,-1);
\draw[gray,dashed](1,1,0)--(0,1,-1);
\draw[gray,dashed](1,1,-1)--(0,0,-1);
\draw[gray,dashed](1,0,-1)--(0,1,-1);


\draw[gray,dashed](.5,0,-.5)--(.5,.5,0);

\draw[gray,dashed](.5,0,-.5)--(1,.5,-.5);
\draw[gray,dashed](.5,0,-.5)--(0,.5,-.5);
\draw[gray,dashed](.5,0,-.5)--(.5,.5,-1);

\draw[gray,dashed](.5,1,-.5)--(.5,.5,0);

\draw[gray,dashed](.5,1,-.5)--(1,.5,-.5);
\draw[gray,dashed](.5,1,-.5)--(0,.5,-.5);
\draw[gray,dashed](.5,1,-.5)--(.5,.5,-1);

\draw[gray,dashed](.5,.5,0)--(1,.5,-.5);
\draw[gray,dashed](.5,.5,0)--(0,.5,-.5);

\draw[gray,dashed](.5,.5,-1)--(1,.5,-.5);
\draw[gray,dashed](.5,.5,-1)--(0,.5,-.5);

\draw[ball color = gray!20](0,0,0)circle(.05);
\draw[ball color = gray!20](1,0,0)circle(.05);
\draw[ball color = gray!20](0,1,0)circle(.05);
\draw[ball color = gray!20](0,0,-1)circle(.05);
\draw[ball color = gray!20](1,1,0)circle(.05);
\draw[ball color = gray!20](0,1,-1)circle(.05);
\draw[ball color = gray!20](1,1,-1)circle(.05);
\draw[ball color = gray!20](1,0,-1)circle(.05);
\draw[ball color = gray!20](.5,0,-.5)circle(.05);
\draw[ball color = gray!20](.5,.5,0)circle(.05);
\draw[ball color = gray!20](0,.5,-.5)circle(.05);
\draw[ball color = gray!20](1,.5,-.5)circle(.05);
\draw[ball color = gray!20](.5,1,-.5)circle(.05);
\draw[ball color = gray!20](.5,.5,-1)circle(.05);

\end{scope}

 \begin{scope}[shift={(8,.1,0)},scale=2]


\draw[gray,dashed](.5,0,{-sqrt(3)/2})--(-.5,0,{-sqrt(3)/2})--(-1,0,0)--(-.5,0,{sqrt(3)/2})--(.5,0,{sqrt(3)/2})--(1,0,0)--(.5,0,{-sqrt(3)/2});

\draw[gray,dashed](.5,{2*sqrt(6)/3},{-sqrt(3)/2})--(-.5,{2*sqrt(6)/3},{-sqrt(3)/2})--(-1,{2*sqrt(6)/3},0)--(-.5,{2*sqrt(6)/3},{sqrt(3)/2})--(.5,{2*sqrt(6)/3},{sqrt(3)/2})--(1,{2*sqrt(6)/3},0)--(.5,{2*sqrt(6)/3},{-sqrt(3)/2});

\draw[gray,dashed](.5,0,{-sqrt(3)/2})--(-.5,0,{sqrt(3)/2});
\draw[gray,dashed](-.5,0,{-sqrt(3)/2})--(.5,0,{sqrt(3)/2});
\draw[gray,dashed](-1,0,0)--(1,0,0);

\draw[gray,dashed](.5,{2*sqrt(6)/3},{-sqrt(3)/2})--(-.5,{2*sqrt(6)/3},{sqrt(3)/2});
\draw[gray,dashed](-.5,{2*sqrt(6)/3},{-sqrt(3)/2})--(.5,{2*sqrt(6)/3},{sqrt(3)/2});
\draw[gray,dashed](-1,{2*sqrt(6)/3},0)--(1,{2*sqrt(6)/3},0);

\draw[gray,dashed](0,0,0)--(.5,{sqrt(6)/3},{-sqrt(3)/6});
\draw[gray,dashed](1,0,0)--(.5,{sqrt(6)/3},{-sqrt(3)/6});
\draw[gray,dashed](.5,0,{-sqrt(3)/2})--(.5,{sqrt(6)/3},{-sqrt(3)/6});

\draw[gray,dashed](0,{2*sqrt(6)/3},0)--(.5,{sqrt(6)/3},{-sqrt(3)/6});
\draw[gray,dashed](1,{2*sqrt(6)/3},0)--(.5,{sqrt(6)/3},{-sqrt(3)/6});
\draw[gray,dashed](.5,{2*sqrt(6)/3},{-sqrt(3)/2})--(.5,{sqrt(6)/3},{-sqrt(3)/6});

\draw[gray,dashed](-1,0,0)--(-.5,{sqrt(6)/3},{-sqrt(3)/6});
\draw[gray,dashed](0,0,0)--(-.5,{sqrt(6)/3},{-sqrt(3)/6});
\draw[gray,dashed](-.5,0,{-sqrt(3)/2})--(-.5,{sqrt(6)/3},{-sqrt(3)/6});

\draw[gray,dashed](-1,{2*sqrt(6)/3},0)--(-.5,{sqrt(6)/3},{-sqrt(3)/6});
\draw[gray,dashed](0,{2*sqrt(6)/3},0)--(-.5,{sqrt(6)/3},{-sqrt(3)/6});
\draw[gray,dashed](-.5,{2*sqrt(6)/3},{-sqrt(3)/2})--(-.5,{sqrt(6)/3},{-sqrt(3)/6});

\draw[gray,dashed](-.5,0,{sqrt(3)/2})--(0,{sqrt(6)/3},{sqrt(3)/3});
\draw[gray,dashed](.5,0,{sqrt(3)/2})--(0,{sqrt(6)/3},{sqrt(3)/3});
\draw[gray,dashed](0,0,0)--(0,{sqrt(6)/3},{sqrt(3)/3});

\draw[gray,dashed](-.5,{2*sqrt(6)/3},{sqrt(3)/2})--(0,{sqrt(6)/3},{sqrt(3)/3});
\draw[gray,dashed](.5,{2*sqrt(6)/3},{sqrt(3)/2})--(0,{sqrt(6)/3},{sqrt(3)/3});
\draw[gray,dashed](0,{2*sqrt(6)/3},0)--(0,{sqrt(6)/3},{sqrt(3)/3});
\draw[gray,dashed](0,{sqrt(6)/3},{sqrt(3)/3})--(.5,{sqrt(6)/3},{-sqrt(3)/6});
\draw[gray,dashed](0,{sqrt(6)/3},{sqrt(3)/3})--(-.5,{sqrt(6)/3},{-sqrt(3)/6});
\draw[gray,dashed](.5,{sqrt(6)/3},{-sqrt(3)/6})--(-.5,{sqrt(6)/3},{-sqrt(3)/6});

\draw[ball color = gray!20] (0,0,0)circle({.1/sqrt(2)});
\draw[ball color = gray!20] (.5,0,{-sqrt(3)/2})circle({.1/sqrt(2)});
\draw[ball color = gray!20] (-.5,0,{-sqrt(3)/2})circle({.1/sqrt(2)});
\draw[ball color = gray!20] (-.5,0,{sqrt(3)/2})circle({.1/sqrt(2)});
\draw[ball color = gray!20] (.5,0,{sqrt(3)/2})circle({.1/sqrt(2)});
\draw[ball color = gray!20] (1,0,0)circle({.1/sqrt(2)});
\draw[ball color = gray!20] (-1,0,0)circle({.1/sqrt(2)});

\draw[ball color = gray!20] (.5,{sqrt(6)/3},{-sqrt(3)/6}) circle({.1/sqrt(2)});
\draw[ball color = gray!20] (-.5,{sqrt(6)/3},{-sqrt(3)/6}) circle({.1/sqrt(2)});
\draw[ball color = gray!20] (0,{sqrt(6)/3},{sqrt(3)/3}) circle({.1/sqrt(2)});

\draw[ball color = gray!20] (0,{2*sqrt(6)/3},0)circle({.1/sqrt(2)});
\draw[ball color = gray!20] (.5,{2*sqrt(6)/3},{-sqrt(3)/2})circle({.1/sqrt(2)});
\draw[ball color = gray!20] (-.5,{2*sqrt(6)/3},{-sqrt(3)/2})circle({.1/sqrt(2)});
\draw[ball color = gray!20] (-.5,{2*sqrt(6)/3},{sqrt(3)/2})circle({.1/sqrt(2)});
\draw[ball color = gray!20] (.5,{2*sqrt(6)/3},{sqrt(3)/2})circle({.1/sqrt(2)});
\draw[ball color = gray!20] (1,{2*sqrt(6)/3},0)circle({.1/sqrt(2)});
\draw[ball color = gray!20] (-1,{2*sqrt(6)/3},0)circle({.1/sqrt(2)});

\end{scope}
  
\end{tikzpicture}
\caption{On the left: The $\mathrm{FCC}$-lattice. On the right: The $\mathrm{HCP}$-lattice. Pairs of points at distance one are connected via the dashed lines.}
\label{fig:HCPFCC}
\end{figure}
We define the $\mathrm{FCC}$-lattice as
\begin{align}\label{def:FCC}
\mathcal{L}_{{\mathrm{FCC}}}:=\mathrm{span}_{\mathbb{Z}}\left\{b_1,b_2,b_3\right\}\,
\end{align}
and the $\mathrm{HCP}$-lattice by
\begin{align}\label{def: HCP}
\mathcal{L}_{{\mathrm{HCP}}}:=\mathrm{span}_{\mathbb{Z}}\left\{e_1,e_2, e_3\right\} \cup \left(\mathrm{span}_{\mathbb{Z}}\left\{e_1,e_2, e_3\right\}+v_1\right)\,.
\end{align}
The two lattices are illustrated in Figure~\ref{fig:HCPFCC}. \BBB We shall write $\mathcal{L}$ to generically denote one of the two lattices defined above.
We define the \BBB\emph{neighborhood} \EEE of a point $x \in \mathcal{L}_{\mathrm{FCC}}$ as the set
\begin{align}\label{eq:neighbourhoodFCC}
\mathcal{N}_{\mathrm{FCC}}(x):= \{\pm b_1, \pm b_2,\pm b_3, \pm (b_1-b_2), \pm (b_1-b_3), \pm (b_2-b_3)\}+x\,.
\end{align}

Similarly, for a point $x \in \mathcal{L}_{\mathrm{HCP}}$ we define its \BBB \emph{neighborhood}  as follows: if $x\in \mathrm{span}_\mathbb{Z}\{e_1,e_2,e_3\}$ then
\begin{align}\label{eq:neighbourhoodHCP1}
\mathcal{N}_{\mathrm{HCP}}(x) :=\{\pm e_1,\pm  e_2, \pm (e_1-e_2),v_1,v_1-e_1,v_1-e_2,v_1-e_3,v_1-e_1-e_3,v_1-e_2-e_3\} +x\,,
\end{align}
while if $x \in \mathrm{span}_\mathbb{Z}\{e_1,e_2,e_3\}+v_1$ then
\begin{align}\label{eq:neighbourhoodHCP2}
\mathcal{N}_{\mathrm{HCP}}(x):=
\{\pm e_1,\pm e_2,\pm (e_1-e_2),-v_1, e_1-v_1,e_2-v_1, e_3-v_1,e_1+e_3-v_1,e_2+e_3-v_1\} +x\,.
\end{align}
Note that $\mathcal{N}_{\mathrm{FCC}}(x)=\mathcal{N}_{\mathrm{FCC}}(0)+x$ for all $x\in \mathcal{L}_{\mathrm{FCC}}$, while this is no more the case for $x\in \mathcal{L}_{\mathrm{HCP}}$.
Also for $\mathcal{N}$ we omit the subscript if we do not need to distinguish between $\mathrm{FCC}$ and $\mathrm{HCP}$.
  It is straightforward to check that for all $x,y\in \mathcal{L}$,
\begin{align*}
x \in \mathcal{N}(y) \iff |x-y|=1\,.
\end{align*} \EEE 
Given $\mathcal{L}$ we define the \emph{Voronoi cell} of $x \in \mathcal{L}$ (with respect to $\mathcal{L}$) by
\begin{align}\label{def:Voronoi}
\mathcal{V}_{\mathcal{L}}(x):= \{y \in \mathbb{R}^3 \colon |y-x|\leq |y-z| \text{ for all } z\in \mathcal{L}\}\,.
\end{align}
Accordingly, given $\varepsilon>0$ we write $\mathcal{V}_{\varepsilon\mathcal{L}}(x)$ for the Voronoi cell centered at $x \in \varepsilon\mathcal{L}$ with respect to \BBB the scaled lattice \EEE $\varepsilon\mathcal{L}$. Given $X\subset \varepsilon\mathcal{L}$ we say that $y \in \mathcal{N}_\varepsilon(x)$ if and only if $\varepsilon^{-1}y \in \mathcal{N}(\varepsilon^{-1}x)$.

\subsection*{Definition of the Energy} 
Given $X\subset \mathcal{L}$ and $A\subset \mathbb{R}^3$ we define the configurational energy of $X$ localized to the set $A$ as
\begin{align*}
E_{\mathcal{L}}(X,A) := \sum_{x\in X\cap A} (12-\#(\mathcal{N}(x)\cap X))\,.
\end{align*}
In the formula above we can interpret the set $X$ as the occupancy of the crystal $\mathcal{L}$, i.e., the set of those nodes of $\mathcal{L}$ occupied by \BBB atoms\EEE. The quantity $12-\#(\mathcal{N_\mathcal{L}}(x)\cap X)$ is also known as the valence of the point $x$ with respect to $X$, i.e., the number of neighbours missing in $X$ in order to have a neighbourhood of maximal cardinality (the number $12$ in the formula). Note that we can also rewrite the energy as
\begin{align*}
E_{\mathcal{L}}(X,A) =\frac{1}{2} \sum_{x\in \mathcal{L}\cap A}\sum_{y\in\mathcal{L}} c(x,y) |\chi_X(y)-\chi_X(x)|\,,
\end{align*}
where 
\begin{align}\label{def:coefficients}
c(x,y) = \begin{cases} 1 &\text{if } y \in \mathcal{N}(x)\,,\\
0&\text{otherwise.}
\end{cases}
\end{align}
\subsection*{Periodicity of the interaction coefficients}
By definition 
\begin{equation*}
{\mathcal L}_\mathrm{FCC}={\mathcal L}_\mathrm{FCC}+b_1={\mathcal L}_\mathrm{FCC}+b_2
={\mathcal L}_\mathrm{FCC}+b_3. 
\end{equation*}
As a consequence of that, for any $x,y\in \mathcal{L}_\mathrm{FCC}$ it holds that
\begin{align*}
\BBB c(x+b_1,y+b_1) = c(x+b_2,y+b_2) = c(x+b_3,y+b_3) = c(x,y)\,.\EEE
\end{align*}
According to the last two equalities, we say that the lattice $\mathcal{L}_{\mathrm{FCC}}$ as well as the interaction coefficients of its configurational energy are periodic with periodicity cell
\begin{align}\label{def:periodicitycellFCC}
T_{{\mathrm{FCC}}} = \{ \lambda_1 b_1 +\lambda_2 b_2 + \lambda_3 b_3 \colon \lambda_i \in [0,1)\}\,,
\end{align}
or simply that they are $T_{{\mathrm{FCC}}}$-periodic. 
Similarly, we \BBB observe \EEE that ${\mathcal{L}}_{\mathrm{HCP}}$ and its interaction coefficients 
are $T_{{\mathrm{HCP}}}$-periodic, where the periodicity cell is defined as 
\begin{align}\label{def:periodicitycellHCP}
T_{{\mathrm{HCP}}} = \{ \lambda_1 e_1 +\lambda_2 e_2 + \lambda_3 e_3 \colon \lambda_i \in [0,1)\}\,.
\end{align}

\subsection*{Surface scaling of the configurational energy}
For $\varepsilon>0$ such that $\varepsilon^3 \#X\to 1$ as $\varepsilon\to 0$ we consider the following family of scaled energies
\begin{align*}
G_{\mathcal{L},\varepsilon}(X) := \varepsilon^{2}\sum_{x \in X} \BBB(12-\#(\mathcal{N}_\varepsilon(x)\cap X))\EEE\,.
\end{align*}
Note that, modeling the points in $X\subset \varepsilon {\mathcal L}$ as hard-spheres of $\varepsilon$ diameter, the quantity $\varepsilon^3 \#X$ is of order one (according to the scaling assumption above) and proportional to the volume of the union of the spheres in $X$. Hence, the scaling factor $\varepsilon^2$ in the energy functional turns out to be a surface scaling. We also define the rescaled empirical measures associated to the configuration $X$ as
\begin{align}\label{def:empirical measures}
\mu_\varepsilon := \varepsilon^3\sum_{x\in X} \delta_x\,.
\end{align}
 Upon identifying $X\subset \varepsilon\mathcal{L}$ with its empirical measure $\mu_\varepsilon$, we can regard these energies to be defined on $\mathcal{M}_+(\mathbb{R}^3)$ by setting 
 \begin{align}\label{def:extensiondiscrete}
 E_{\mathcal{L},\varepsilon}(\mu) := \begin{cases} G_{\mathcal{L},\varepsilon}(X) &\text{if } \mu =\mu_\varepsilon \text{ given in } \eqref{def:empirical measures} \text{ for some } X \subset \varepsilon\mathcal{L}\,,\\
 +\infty &\text{otherwise.}
 \end{cases}
 \end{align}
\subsection*{The coarse grained continuum energy} For $\mathcal{L}$ we define the \emph{homogenized surface energy density} $\varphi_{\mathcal{L}} \colon \mathbb{R}^3\to [0,+\infty]$  as the convex positively homogeneous function of degree one such that for all $\nu \in \mathbb{S}^{2}$ we have
\begin{align}\label{def:varphi}
\varphi_{\mathcal{L}}(\nu) := \lim_{T\to +\infty} \frac{1}{T^{2}} \inf\left\{ E_{\mathcal{L}}(X,Q^\nu_T) : X \subset \mathcal{L}, \chi_X(i)= u_\nu(i) \text{ for } i\in \mathcal{L} \setminus Q^\nu_{T-3}\right\}\,,
\end{align}
where $u_\nu$ is given by \eqref{def:unu}.

In order to be able to apply \cite[Proposition 2.6]{ChaKre} and eventually obtain an alternative representation of $\varphi_{\mathcal{L}}$ (up to a coordinate transformation and reparametrization of the interaction coefficients), we define for $u \colon \mathcal{L} \to \mathbb{R}$, $A \subset \mathbb{R}^3$ the energy\BBB
\begin{align*}
F_\mathcal{L}(u,A) := \frac{1}{2}\sum_{x \in \mathcal{L}\cap A} \sum_{y \in \mathcal{L}} c(x,y)|u(y)-u(x)|\,. 
\end{align*}\EEE
We are now in position to state \cite[Proposition 2.6]{ChaKre}.

\begin{proposition}\label{prop:representation}\BBB Let $c(x,y)$ be as in \eqref{def:coefficients}.\EEE Then
\begin{align}\label{eq:propvarphi}
\varphi_{\mathcal{L}}(\nu) = \frac{1}{|T_{\mathcal{L}}|}\inf\left\{F_{\mathcal{L}}(u,T_{\mathcal{L}}) : u\colon \mathcal{L} \to \mathbb{R},  u(\cdot)-\langle\nu,\cdot\rangle \text{ is } T_{\mathcal{L}}\text{-periodic}\right\}\,.
\end{align} 
\end{proposition} 

With the definition of surface energy density at hand we can define the {\it coarse-grained continuum energy} $E_{\mathcal{L}}\colon \mathcal{M}_+(\mathbb{R}^3) \to [0,+\infty]$ as \BBB
\begin{align}\label{def:limitenergy}
E_\mathcal{L}(\mu) := \begin{cases} \int_{\partial^*V} \varphi_{\mathcal{L}}(\nu)\,\mathrm{d}\mathcal{H}^{2} &\text{if } \mu = \sqrt{2}\mathcal{L}^3\llcorner V, \chi_V\in BV_{\mathrm{loc}}(\mathbb{R}^3)\,,\\
+\infty &\text{otherwise.}
\end{cases}
\end{align}
\EEE
with $\varphi_\mathcal{L}$ given by \eqref{def:varphi}. Here, $\partial^*V$ denotes the reduced boundary of the set $V$, $\nu$ its outer normal and $\mathcal{H}^{2}$, \BBB as noted at the beginning of this section\EEE, stands for the $2$-dimensional Hausdorff measure in $\mathbb{R}^3$ (cf.~\cite{AFP}, Chapters 2.8 and 3.5). 

\BBB \noindent In what follows we say that $F_\varepsilon$ $\Gamma$-converges to $F$ if for all sequences $\{\varepsilon_j\}_j$ converging to $0$ we have $\Gamma$-$\lim_j F_{\varepsilon_j} = F$.\EEE


\subsection*{The Wulff Crystal} 
\BBB In this section we calculate the Wulff crystals of the coarse grained $\mathrm{FCC}$ and $\mathrm{HCP}$ lattices. To the best of our knowledge, this is the first time that such a calculation has been carried out in a rigorous analytical way. In what follows we introduce the notion of Wulff shape in the general case of $\mathbb{R}^n$. While in the rest of this section we limit ourselves to the case $n=3$, in   Section~\ref{appendix:general}  we consider general $n$.    \EEE

Given $\varphi \colon \mathbb{R}^n \to [0,+\infty)$ convex, non-degenerate, (i.e.~there exist $0<c<C$ such that $c \leq \varphi(\nu) \leq C$ for all $\nu \in \mathbb{S}^{n-1}$) positively homogeneous of degree one, we define the Wulff set of $\varphi$ by
\begin{align}\label{def:wulffcrystal}
W_\varphi := \{ \zeta \in \mathbb{R}^n\colon \langle \zeta,\nu \rangle \leq \varphi(\nu) \text{ for all } \nu \in \mathbb{S}^{n-1}\}\,.
\end{align}
Thanks to the anistropic isoperimetric inequality (cf.~\cite{FonsecaMueller}), we have that $ W_\varphi$ is the unique (up to rigid motions) minimizer of 
\begin{align*}
\min\left\{ \int_{\partial^* A} \varphi(\nu) \,\mathrm{d}\mathcal{H}^{n-1} \colon |A| = |W_\varphi| \right\}\,.
\end{align*}
Given $\lambda >0$ we set $W_\lambda = \left(\frac{\lambda}{|W_\varphi|}\right)^{1/n} W_\varphi$ so that $|W_\lambda|=\lambda$ and, by scaling, it solves the minimum problem above among all sets $A \subset \mathbb{R}^n$ with $|A|=\lambda$.
We recall here that, defining $\varphi^\circ \colon \mathbb{R}^n \to [0,+\infty)$ by
\begin{align*}
\varphi^\circ(\xi) = \sup_{\nu \in \mathbb{S}^{n-1}} \frac{\langle\nu,\xi\rangle}{\varphi(\nu)}\,,
\end{align*}
it holds that $W_\varphi = \{\varphi^\circ \leq 1\}$. 

\GGG\begin{definition}\label{def:Gamma-convergence} Let $(X,\tau)$ be a topological space and let $F_k \colon X \to [0,+\infty]$. For $x \in X$ we set
\begin{align*}
\Gamma\text{-}\limsup_{k \to +\infty} F_k(x) = \inf\left\{\limsup_{k\to +\infty} F_k(x_k) \colon x_k \overset{\tau}{\to} x\right\} 
\end{align*}
and
\begin{align*}
\Gamma\text{-}\liminf_{k \to +\infty} F_k(x) = \inf\left\{\liminf_{k\to +\infty} F_k(x_k) \colon x_k \overset{\tau}{\to} x\right\}\,.
\end{align*}
If there exists $F \colon X \to [0,+\infty]$ such that
\begin{align*}
F(x)=\Gamma\text{-}\limsup_{k \to +\infty} F_k(x)= \Gamma\text{-}\liminf_{k \to +\infty} F_k(x)\,,
\end{align*}
we say that $F_k$ $\Gamma$-converges with respect to $\tau$ to $F$ and we write
\begin{align*}
F(x) = \Gamma\text{-}\lim_{k \to +\infty} F_k(x)\,.
\end{align*}
If we have $(F_\varepsilon)_{\varepsilon>0} \colon X \to [0,+\infty]$ we say that $F_\varepsilon$ $\Gamma$-converges with respect to $\tau$ to $F$ if $F_{\varepsilon_k}$ $\Gamma$-converges with respect to $\tau$ to $F$ for all $\varepsilon_k \to 0$.
\end{definition}
\EEE

The following variational coarse-graining result is proved in Section \ref{sec:Gammaconvergence}.
\BBB\begin{theorem}\label{theorem:Gammaconvergence} Let $\varepsilon \to 0$, and let $E_{\mathcal L,\varepsilon}$ and $E_{\mathcal L}$ be the energy functionals defined in \eqref{def:extensiondiscrete} and \eqref{def:limitenergy}, respectively. 
\begin{itemize}
\item[i)] (Compactness) Let  $\{\mu_\varepsilon\}_\varepsilon\subset\mathcal{M}_+(\mathbb{R}^3)$ be such that 
\begin{align*}
\sup_{\varepsilon >0} E_{\mathcal{L},\varepsilon}(\mu_\varepsilon) <+\infty\,.
\end{align*}
Then there exists $V\subset \mathbb{R}^3$ such that $\chi_V \in BV_{\mathrm{loc}}(\mathbb{R}^3)$, $\mu=\sqrt{2}\mathcal{L}^3\llcorner V$,  and a subsequence (not relabeled) such that $\mu_\varepsilon \overset{*}{\rightharpoonup} \mu$. Furthermore, if $\mu_\varepsilon$ is such that 
\begin{align*}
E_{\mathcal{L},\varepsilon}(\mu_\varepsilon) =\inf_{\nu \in \mathcal{M}_+(\mathbb{R}^3):\, |\nu|(\mathbb{R}^3)=\varepsilon^3 n_\varepsilon} E_{\mathcal{L},\varepsilon}(\mu)\,,
\end{align*}
with $\varepsilon^3 n_\varepsilon\to \sqrt{2}v$,
then $\mu = \sqrt{2}\mathcal{L}^3\llcorner W_{\varphi_\mathcal{L}}^v$, where $ W_{\varphi_\mathcal{L}}^{v}=\lambda W_{\varphi_\mathcal{L}}$ (defined in \eqref{def:wulffcrystal}) for $\lambda>0$ such that  $|W_{\varphi_\mathcal{L}}^{v}|= v$.


\EEE

\item[ii)] (Liminf inequality)  Let $\mu\in \mathcal{M}_+(\mathbb{R}^3)$ be such that $\mu_\varepsilon \overset{*}{\rightharpoonup} \mu$. Then
\begin{align*}
E_\mathcal{L}(\mu) \leq \liminf_{\varepsilon\to 0}E_{\mathcal{L},\varepsilon}(\mu_\varepsilon)\,.
\end{align*}
\item[iii)] (Limsup inequality) Let $\mu \in \mathcal{M}_+(\mathbb{R}^3)$. Then there exists $\{\mu_\varepsilon\}_\varepsilon \subset \mathcal{M}_+(\mathbb{R}^3)$ such that $\mu_\varepsilon \overset{*}{\rightharpoonup} \mu$ and
\begin{align*}
E_\mathcal{L}(\mu) \geq \liminf_{\varepsilon\to 0}E_{\mathcal{L},\varepsilon}(\mu_\varepsilon)\,.
\end{align*}
\end{itemize}
\end{theorem}\EEE

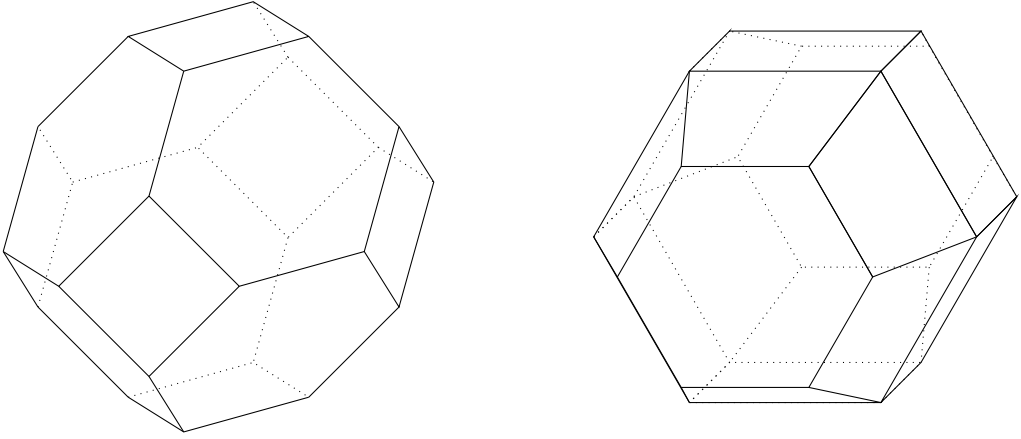
\begin{figure}[htp]
\begin{tikzpicture}[scale=.6]

\draw(4,2,0)--(4,0,2)--(4,-2,0)--(4,0,-2)--(4,2,0);
\draw(2,4,0)--(0,4,2)--(-2,4,0)--(0,4,-2)--(2,4,0);
\draw(0,2,4)--(2,0,4)--(0,-2,4)--(-2,0,4)--(0,2,4);

\draw(4,2,0)--(2,4,0);
\draw(4,0,2)--(2,0,4);
\draw(0,4,2)--(0,2,4);

\draw[dotted](0,2,-4)--(2,0,-4)--(0,-2,-4)--(-2,0,-4)--(0,2,-4);
\draw[dotted](0,2,-4)--(0,4,-2);
\draw[dotted](2,0,-4)--(4,0,-2);
\draw[dotted](0,-2,-4)--(0,-4,-2);
\draw(-4,2,0)--(-4,0,2);
\draw(-4,0,2)--(-4,-2,0);
\draw[dotted](-4,-2,0)--(-4,0,-2)--(-4,2,0);
\draw(-4,2,0)--(-2,4,0);
\draw(-4,0,2)--(-2,0,4);
\draw(-4,-2,0)--(-2,-4,0);
\draw[dotted](-4,0,-2)--(-2,0,-4);

\draw(2,-4,0)--(0,-4,2)--(-2,-4,0);
\draw[dotted](-2,-4,0)--(0,-4,-2)--(2,-4,0);
\draw(2,-4,0)--(4,-2,0);
\draw(0,-4,2)--(0,-2,4);

\begin{scope}[shift={(13,0)}]

\draw({2*sqrt(2)},0,{2*sqrt(3)})--({sqrt(2)},{sqrt(6)},{2*sqrt(3)})--({-sqrt(2)},{sqrt(6)},{2*sqrt(3)})--({-2*sqrt(2)},0,{2*sqrt(3)})--({-sqrt(2)},{-sqrt(6)},{2*sqrt(3)})--({sqrt(2)},{-sqrt(6)},{2*sqrt(3)})--({2*sqrt(2)},0,{2*sqrt(3)});

\draw({1.5*sqrt(2)},{1.5*sqrt(6)},{2/3*sqrt(3)})--({-1.5*sqrt(2)},{1.5*sqrt(6)},{2/3*sqrt(3)})--({-1.5*sqrt(2)},{1.5*sqrt(6)},{-2/3*sqrt(3)})--({1.5*sqrt(2)},{1.5*sqrt(6)},{-2/3*sqrt(3)})--({1.5*sqrt(2)},{1.5*sqrt(6)},{2/3*sqrt(3)});

\draw({sqrt(2)},{sqrt(6)},{2*sqrt(3)})--({1.5*sqrt(2)},{1.5*sqrt(6)},{2/3*sqrt(3)});
\draw({-sqrt(2)},{sqrt(6)},{2*sqrt(3)})--({-1.5*sqrt(2)},{1.5*sqrt(6)},{2/3*sqrt(3)});
\draw[dotted]({-sqrt(2)},{sqrt(6)},{-2*sqrt(3)})--({-1.5*sqrt(2)},{1.5*sqrt(6)},{-2/3*sqrt(3)});
\draw[dotted]({sqrt(2)},{sqrt(6)},{-2*sqrt(3)})--({1.5*sqrt(2)},{1.5*sqrt(6)},{-2/3*sqrt(3)});


\draw[dotted]({2*sqrt(2)},0,{-2*sqrt(3)})--({sqrt(2)},{sqrt(6)},{-2*sqrt(3)})--({-sqrt(2)},{sqrt(6)},{-2*sqrt(3)})--({-2*sqrt(2)},0,{-2*sqrt(3)})--({-sqrt(2)},{-sqrt(6)},{-2*sqrt(3)})--({sqrt(2)},{-sqrt(6)},{-2*sqrt(3)})--({2*sqrt(2)},0,{-2*sqrt(3)});


\draw({1.5*sqrt(2)},{-1.5*sqrt(6)},{2/3*sqrt(3)})--({-1.5*sqrt(2)},{-1.5*sqrt(6)},{2/3*sqrt(3)});

\draw[dotted]({1.5*sqrt(2)},{-1.5*sqrt(6)},{2/3*sqrt(3)})--({-1.5*sqrt(2)},{-1.5*sqrt(6)},{2/3*sqrt(3)})--({-1.5*sqrt(2)},{-1.5*sqrt(6)},{-2/3*sqrt(3)})--({1.5*sqrt(2)},{-1.5*sqrt(6)},{-2/3*sqrt(3)})--({1.5*sqrt(2)},{-1.5*sqrt(6)},{2/3*sqrt(3)});

\draw({sqrt(2)},{-sqrt(6)},{2*sqrt(3)})--({1.5*sqrt(2)},{-1.5*sqrt(6)},{2/3*sqrt(3)});
\draw({-sqrt(2)},{-sqrt(6)},{2*sqrt(3)})--({-1.5*sqrt(2)},{-1.5*sqrt(6)},{2/3*sqrt(3)});
\draw[dotted]({-sqrt(2)},{-sqrt(6)},{-2*sqrt(3)})--({-1.5*sqrt(2)},{-1.5*sqrt(6)},{-2/3*sqrt(3)});
\draw[dotted]({sqrt(2)},{-sqrt(6)},{-2*sqrt(3)})--({1.5*sqrt(2)},{-1.5*sqrt(6)},{-2/3*sqrt(3)});


\draw({1.5*sqrt(2)},{1.5*sqrt(6)},{2/3*sqrt(3)})--({3*sqrt(2)},0,{2/3*sqrt(3)})--({3*sqrt(2)},0,{-2/3*sqrt(3)})--({1.5*sqrt(2)},{1.5*sqrt(6)},{-2/3*sqrt(3)})--({1.5*sqrt(2)},{1.5*sqrt(6)},{2/3*sqrt(3)});

\draw({1.5*sqrt(2)},{-1.5*sqrt(6)},{2/3*sqrt(3)})--({3*sqrt(2)},0,{2/3*sqrt(3)})--({3*sqrt(2)},0,{-2/3*sqrt(3)})--({1.5*sqrt(2)},{-1.5*sqrt(6)},{-2/3*sqrt(3)})--({1.5*sqrt(2)},{-1.5*sqrt(6)},{2/3*sqrt(3)});

\draw({-1.5*sqrt(2)},{-1.5*sqrt(6)},{2/3*sqrt(3)})--({-3*sqrt(2)},0,{2/3*sqrt(3)});
\draw[dotted]({-3*sqrt(2)},0,{2/3*sqrt(3)})--({-3*sqrt(2)},0,{-2/3*sqrt(3)})--({-1.5*sqrt(2)},{-1.5*sqrt(6)},{-2/3*sqrt(3)})--({-1.5*sqrt(2)},{-1.5*sqrt(6)},{2/3*sqrt(3)});

\draw({-1.5*sqrt(2)},{1.5*sqrt(6)},{2/3*sqrt(3)})--({-3*sqrt(2)},0,{2/3*sqrt(3)});

\draw[dotted]({-3*sqrt(2)},0,{2/3*sqrt(3)})--({-3*sqrt(2)},0,{-2/3*sqrt(3)})--({-1.5*sqrt(2)},{1.5*sqrt(6)},{-2/3*sqrt(3)})--({-1.5*sqrt(2)},{1.5*sqrt(6)},{2/3*sqrt(3)});

\draw({3*sqrt(2)},0,{2/3*sqrt(3)})--({2*sqrt(2)},0,{2*sqrt(3)})--({sqrt(2)},{sqrt(6)},{2*sqrt(3)})--({1.5*sqrt(2)},{1.5*sqrt(6)},{2/3*sqrt(3)})--({3*sqrt(2)},0,{2/3*sqrt(3)});

\draw[dotted]({3*sqrt(2)},0,{-2/3*sqrt(3)})--({2*sqrt(2)},0,{-2*sqrt(3)});

\draw[dotted]({-3*sqrt(2)},0,{-2/3*sqrt(3)})--({-2*sqrt(2)},0,{-2*sqrt(3)});

\draw({-3*sqrt(2)},0,{2/3*sqrt(3)})--({-2*sqrt(2)},0,{2*sqrt(3)});

\end{scope}

\end{tikzpicture}
\caption{The Wulff Crystal of the FCC-lattice on the left and HCP-lattice on the right.}
\label{fig:WFCC}
\end{figure}

%
%
%

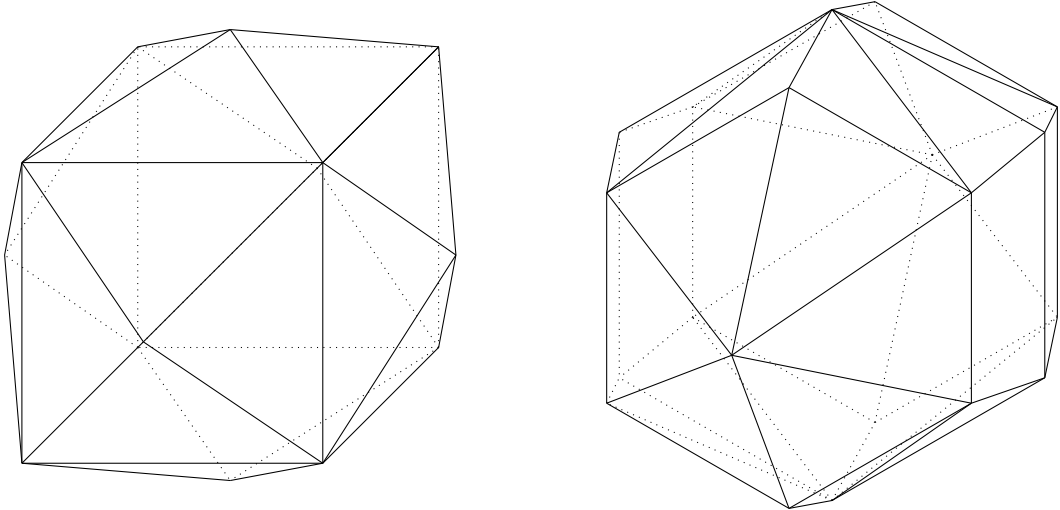
\begin{figure}
\begin{tikzpicture}

\draw[dotted](2,-2,-2)--(-2,-2,-2);
\draw[dotted](2,2,-2)--(-2,2,-2);

\draw[dotted](-2,-2,-2)--(-2,2,-2);

\draw[dotted](2,-2,-2)--(2,2,-2);

\foreach \j in {-1,1}{
\foreach \k in {-1,1}{

\draw(\j*2,2,\k*2)--(0,3,0);

\draw(2,\j*2,\k*2)--(3,0,0);

\draw[dotted](\j*2,\k*2,-2)--(0,0,-3);
\draw(\j*2,\k*2,2)--(0,0,3);

}}
\foreach \k in {-1,1}{

\draw(-2,\k*2,2)--(-3,0,0);
\draw[dotted](-2,\k*2,-2)--(-3,0,0);

\draw(\k*2,-2,2)--(0,-3,0);

\draw[dotted](\k*2,-2,-2)--(0,-3,0);

\draw(2,\k*2,2)--(-2,2*\k,2);

\draw(2,\k*2,2)--(2,2*\k,-2);

\draw(\k*2,2,2)--(2*\k,2,-2);

\draw(\k*2,2,2)--(2*\k,-2,2);

}

\begin{scope}[shift={(8,0)},scale=12]

\draw(0,0,{sqrt(3)/6})-- (0,{2*sqrt(6)/21},{sqrt(3)/14});
\draw(0,0,{sqrt(3)/6})-- (0,{-2*sqrt(6)/21},{sqrt(3)/14});


\draw[dotted](0,0,{-sqrt(3)/6})-- (0,{2*sqrt(6)/21},{-sqrt(3)/14});
\draw[dotted](0,0,{-sqrt(3)/6})-- (0,{-2*sqrt(6)/21},{-sqrt(3)/14});


\draw(0,0,{sqrt(3)/6})--({sqrt(2)/7},{sqrt(6)/21},{sqrt(3)/14});
\draw(0,0,{sqrt(3)/6})--({sqrt(2)/7},{-sqrt(6)/21},{sqrt(3)/14});
\draw(0,0,{sqrt(3)/6})--({-sqrt(2)/7},{sqrt(6)/21},{sqrt(3)/14});
\draw(0,0,{sqrt(3)/6})--({-sqrt(2)/7},{-sqrt(6)/21},{sqrt(3)/14});


\draw[dotted](0,0,{-sqrt(3)/6})--({sqrt(2)/7},{sqrt(6)/21},{-sqrt(3)/14});
\draw[dotted](0,0,{-sqrt(3)/6})--({sqrt(2)/7},{-sqrt(6)/21},{-sqrt(3)/14});
\draw[dotted](0,0,{-sqrt(3)/6})--({-sqrt(2)/7},{sqrt(6)/21},{-sqrt(3)/14});
\draw[dotted](0,0,{-sqrt(3)/6})--({-sqrt(2)/7},{-sqrt(6)/21},{-sqrt(3)/14});


\draw(0,{sqrt(6)/9},0)--({sqrt(2)/7},{sqrt(6)/21},{sqrt(3)/14});
\draw(0,{sqrt(6)/9},0)--({sqrt(2)/7},{sqrt(6)/21},{-sqrt(3)/14});
\draw(0,{sqrt(6)/9},0)--({-sqrt(2)/7},{sqrt(6)/21},{sqrt(3)/14});
\draw[dotted](0,{sqrt(6)/9},0)--({-sqrt(2)/7},{sqrt(6)/21},{-sqrt(3)/14});


\draw(0,{-sqrt(6)/9},0)--({sqrt(2)/7},{-sqrt(6)/21},{sqrt(3)/14});
\draw[dotted](0,{-sqrt(6)/9},0)--({sqrt(2)/7},{-sqrt(6)/21},{-sqrt(3)/14});
\draw[dotted](0,{-sqrt(6)/9},0)--({-sqrt(2)/7},{-sqrt(6)/21},{sqrt(3)/14});
\draw[dotted](0,{-sqrt(6)/9},0)--({-sqrt(2)/7},{-sqrt(6)/21},{-sqrt(3)/14});


\draw(0,{sqrt(6)/9},0)--(0,{2*sqrt(6)/21},{sqrt(3)/14});
\draw(0,{sqrt(6)/9},0)--(0,{2*sqrt(6)/21},{-sqrt(3)/14});


\draw(0,{-sqrt(6)/9},0)--(0,{-2*sqrt(6)/21},{sqrt(3)/14});
\draw[dotted](0,{-sqrt(6)/9},0)--(0,{-2*sqrt(6)/21},{-sqrt(3)/14});


\draw(0,{sqrt(6)/9},0)--({sqrt(2)/6},{sqrt(6)/18},0);
\draw(0,{sqrt(6)/9},0)--({-sqrt(2)/6},{sqrt(6)/18},0);


\draw(0,{-sqrt(6)/9},0)--({sqrt(2)/6},{-sqrt(6)/18},0);
\draw[dotted](0,{-sqrt(6)/9},0)--({-sqrt(2)/6},{-sqrt(6)/18},0);


\draw({sqrt(2)/7},{sqrt(6)/21},{sqrt(3)/14})--(0,{2*sqrt(6)/21},{sqrt(3)/14});
\draw({-sqrt(2)/7},{sqrt(6)/21},{sqrt(3)/14})--(0,{2*sqrt(6)/21},{sqrt(3)/14});


\draw({sqrt(2)/7},{-sqrt(6)/21},{sqrt(3)/14})--(0,{-2*sqrt(6)/21},{sqrt(3)/14});
\draw({-sqrt(2)/7},{-sqrt(6)/21},{sqrt(3)/14})--(0,{-2*sqrt(6)/21},{sqrt(3)/14});

\draw({sqrt(2)/7},{sqrt(6)/21},{-sqrt(3)/14})--(0,{2*sqrt(6)/21},-{sqrt(3)/14});
\draw[dotted]({-sqrt(2)/7},{sqrt(6)/21},{-sqrt(3)/14})--(0,{2*sqrt(6)/21},{-sqrt(3)/14});

\draw[dotted]({sqrt(2)/7},{-sqrt(6)/21},{-sqrt(3)/14})--(0,{-2*sqrt(6)/21},{-sqrt(3)/14});
\draw[dotted]({-sqrt(2)/7},{-sqrt(6)/21},{-sqrt(3)/14})--(0,{-2*sqrt(6)/21},{-sqrt(3)/14});


\draw({sqrt(2)/7},{sqrt(6)/21},{sqrt(3)/14})--({sqrt(2)/6},{sqrt(6)/18},0);
\draw({sqrt(2)/7},{sqrt(6)/21},{-sqrt(3)/14})--({sqrt(2)/6},{sqrt(6)/18},0);


\draw({-sqrt(2)/7},{sqrt(6)/21},{sqrt(3)/14})--({-sqrt(2)/6},{sqrt(6)/18},0);
\draw[dotted]({-sqrt(2)/7},{sqrt(6)/21},{-sqrt(3)/14})--({-sqrt(2)/6},{sqrt(6)/18},0);

\draw({sqrt(2)/7},{-sqrt(6)/21},{sqrt(3)/14})--({sqrt(2)/6},{-sqrt(6)/18},0);
\draw({sqrt(2)/7},{-sqrt(6)/21},{-sqrt(3)/14})--({sqrt(2)/6},{-sqrt(6)/18},0);

\draw[dotted]({-sqrt(2)/7},{-sqrt(6)/21},{sqrt(3)/14})--({-sqrt(2)/6},{-sqrt(6)/18},0);
\draw[dotted]({-sqrt(2)/7},{-sqrt(6)/21},{-sqrt(3)/14})--({-sqrt(2)/6},{-sqrt(6)/18},0);


\draw({sqrt(2)/7},{sqrt(6)/21},{sqrt(3)/14})--({sqrt(2)/7},{-sqrt(6)/21},{sqrt(3)/14});

\draw({-sqrt(2)/7},{sqrt(6)/21},{sqrt(3)/14})--({-sqrt(2)/7},{-sqrt(6)/21},{sqrt(3)/14});

\draw({sqrt(2)/7},{sqrt(6)/21},{-sqrt(3)/14})--({sqrt(2)/7},{-sqrt(6)/21},{-sqrt(3)/14});

\draw[dotted]({-sqrt(2)/7},{sqrt(6)/21},{-sqrt(3)/14})--({-sqrt(2)/7},{-sqrt(6)/21},{-sqrt(3)/14});


\draw({sqrt(2)/6},{sqrt(6)/18},0)--({sqrt(2)/6},{-sqrt(6)/18},0);

\draw[dotted]({-sqrt(2)/6},{sqrt(6)/18},0)--({-sqrt(2)/6},{-sqrt(6)/18},0);

\end{scope}

\end{tikzpicture}
\caption{The sublevel set $\{\varphi_{\mathrm{FCC}}\leq 1\}$ on the left and the sublevel set $\{\varphi_{\mathrm{HCP}}\leq 1\}$ on the right.}
\label{fig:FCCleq1}
\end{figure}

\subsection*{Explicit formula of the surface energy densities}
Taking advantage of the representation formula \eqref{eq:propvarphi} stated in Proposition \eqref{prop:representation}, we provide the explicit formulas of the surface energy density $\varphi_{{\mathcal L}_{\mathrm{FCC}}}$ and $\varphi_{{\mathcal L}_{\mathrm{HCP}}}$. \GGG Their sublevel sets are depicted in Figure~\ref{fig:FCCleq1}. \EEE With the two explicit formulas at hand we can calculate the polar  functions  of both densities, the associated Wulff shapes and the surface energy per unit volume of both the $\mathrm{FCC}$ and $\mathrm{HCP}$ crystals.  In order not to overburden the reader with notation, we write $\varphi_{\mathrm{FCC}}$ and $\varphi_{\mathrm{HCP}}$ for $\varphi_{{\mathcal L}_{\mathrm{FCC}}}$ and $\varphi_{{\mathcal L}_{\mathrm{HCP}}}$ as well as $W_{\mathrm{FCC}}$ and $W_{\mathrm{HCP}}$ instead of $W_{\varphi_{\mathcal{L}_\mathrm{FCC}}}$ and $W_{\varphi_{\mathcal{L}_\mathrm{HCP}}}$. \GGG  $W_{\mathrm{FCC}}$ and $W_{\mathrm{HCP}}$  are depicted in Figure~\ref{fig:WFCC}. \EEE

\begin{proposition} \label{prop:main1} The following formulas hold true. 
\begin{align}\label{eq:varphiFCC}
\varphi_{\mathrm{FCC}}(\nu) = |\nu_1+\nu_2| + |\nu_1+\nu_3|+ |\nu_2+\nu_3|+ |\nu_1-\nu_2|+|\nu_1-\nu_3|+|\nu_2-\nu_3|\,,
\end{align}
and
\begin{align}\label{eq:FCCdual}
\varphi_{\mathrm{FCC}}^\circ(\zeta) = \max\left\{\frac{1}{4}\|\zeta\|_\infty,\frac{1}{6}\|\zeta\|_1\right\}\,.
\end{align}
In particular, $W_{\mathrm{FCC}}$ is a  truncated octahedron and its surface energy per unit volume is
\begin{align}\label{eq:surfaceFCC}
|W_{\mathrm{FCC}}|^{-2/3}\int_{\partial^* W_{\mathrm{FCC}}} \varphi_{\mathrm{FCC}}(\nu)\,\mathrm{d}\mathcal{H}^2 =3 \cdot 2^2 \cdot 2^{2/3} \,.
\end{align}
\end{proposition}

\begin{proposition} \label{prop:main2} The following formulas hold true. 
\begin{align}\label{eq:varphiHCP}
\begin{split}
\varphi_{\mathrm{HCP}}(\nu) &= \sqrt{2}\left(|\langle e_1,\nu\rangle|+ |\langle e_2,\nu\rangle|+ |\langle e_1-e_2,\nu\rangle|\right) + \frac{1}{\sqrt{2}}|\langle e_3,\nu\rangle| \\&\quad+\sqrt{2}\max\left\{|\langle e_1,\nu\rangle|,|\langle e_2,\nu\rangle|,|\langle e_3,\nu\rangle|,|\langle e_1-e_2,\nu\rangle|\right\}\,,
\end{split}
\end{align}
and
\begin{align}\label{eq:HCPdual}
\begin{split}
\varphi^\circ_{\mathrm{HCP}}(\zeta) =\max\Big\{ &\frac{2}{7\sqrt{2}}\left(|\zeta_1| +\frac{1}{\sqrt{3}}|\zeta_2| + \frac{3}{2\sqrt{6}}|\zeta_3|\right),\frac{1}{2\sqrt{3}}|\zeta_3|,\\&   \frac{2}{3\sqrt{6}}|\zeta_2|, \frac{4}{7\sqrt{6}}|\zeta_2|+\frac{3}{14\sqrt{3}}|\zeta_3|, \frac{1}{3\sqrt{2}}\left(|\zeta_1| +\frac{1}{\sqrt{3}}|\zeta_2|\right)  \Big\} \,.
\end{split}
\end{align}
In particular, $W_{\mathrm{HCP}}$ is a truncated elongated hexagonal bipyramid and its surface energy per unit volume is
\begin{align}\label{eq:surfaceHCP}
|W_{\mathrm{HCP}}|^{-2/3}\int_{\partial^* W_{\mathrm{HCP}}} \varphi_{\mathrm{HCP}}(\nu)\,\mathrm{d}\mathcal{H}^2 = 3\cdot 2^{2/3} \cdot 65^{1/3}\,.
\end{align}
\end{proposition}

\section{Proof of Proposition \ref{prop:main1} and Proposition \ref{prop:main2}} \label{sec:proofs}

In this section we prove Proposition \ref{prop:main1} and Proposition \ref{prop:main2}. To this end, we use Proposition \ref{prop:representation} to note that $\varphi_\mathcal{L}$ is given by \eqref{eq:propvarphi}.

\begin{proof}[Proof of Proposition \ref{prop:main1}] We divide the proof into several steps. First, we calculate $\varphi_{{\mathrm{FCC}}}$. Then, we calculate $\varphi^\circ_{\mathrm{FCC}}$. Lastly, we calculate \eqref{eq:surfaceFCC}. Recall \eqref{def:vectors2}. \\
\begin{step}{1}(Calculation of $\varphi_{\mathrm{FCC}}$)\BBB
We make use of Proposition \ref{prop:representation} in order to calculate $\varphi_{\mathrm{FCC}}$. First of all, owing to \eqref{def:periodicitycellFCC}, we note that 
\begin{align}\label{eq:VolumeFCC}
|T_{{\mathrm{FCC}}}|= \frac{1}{3}\sqrt{6} \cdot \frac{1}{2}\sqrt{3} = \frac{1}{2}\sqrt{2}\,.
\end{align}
Given $u \colon \mathcal{L}_{\mathrm{FCC}} \to \mathbb{R}$  such that $u(\cdot)-\langle\nu,\cdot\rangle$ is $T_{{\mathrm{FCC}}}$-periodic we have that $u(x+ b_i) = u(x)+\langle b_i,\nu\rangle$ for all $i=1,2,3$. Therefore, \BBB $u$ is an affine function of the form \EEE $u(x) = \langle x,\nu\rangle +c\, , x \in \mathcal{L}_{\mathrm{FCC}}$ for some $c\in \mathbb{R}$. Lastly, note that $\mathcal{L}_{\mathrm{FCC}} \cap T_{{\mathrm{FCC}}}= \{0\}$. Using \eqref{eq:propvarphi} and \eqref{eq:VolumeFCC}, we obtain \BBB
\begin{align*}
\varphi_{\mathrm{FCC}}(\nu) = \frac{1}{2}\sqrt{2} \sum_{\xi \in \mathcal{N}_{{\mathrm{FCC}}}} |u(\xi)-u(0)| = \frac{1}{2}\sqrt{2}\sum_{\xi \in \mathcal{N}_{{\mathrm{FCC}}}} |\langle\xi,\nu\rangle|\,.
\end{align*} \EEE
Employing now \eqref{eq:neighbourhoodFCC}, we obtain \eqref{eq:varphiFCC}.\EEE
\end{step}\\
\BBB \begin{step}{2}(Calculation of $\varphi^\circ_{\mathrm{FCC}}$)
 Let $G$ be the isometry group on $\mathbb{R}^3$ whose elements $g \in G$ are the linear isometries $g \colon \mathbb{R}^3 \to \mathbb{R}^3$ defined by $g(\nu_1,\nu_2,\nu_3) = (b_1 \nu_{\pi_1},b_2 \nu_{\pi_2},b_3 \nu_{\pi_3})$ where $\pi$ is a permutation on $\{1,2,3\}$ and $b_i \in \{-1,1\}$. Since $\varphi_{\mathrm{FCC}}(g(\nu)) = \varphi_{\mathrm{FCC}}(\nu)$ for all $g \in G$, $\nu \in \mathbb{R}^3$, we infer that
 \begin{align*}
 \varphi^\circ_{\mathrm{FCC}}(\zeta) := \underset{\varphi_{\mathrm{FCC}}(\nu)\leq 1}{\max_{\nu \in \mathbb{R}^3 }} \langle \zeta,\nu\rangle  =\underset{\varphi_{\mathrm{FCC}}(g^{-1}(\nu))\leq 1}{\max_{\nu \in \mathbb{R}^3 }} \langle \zeta,g^{-1}(\nu)\rangle =\underset{\varphi_{\mathrm{FCC}}(\nu)\leq 1}{\max_{\nu \in \mathbb{R}^3 }} \langle g(\zeta),\nu\rangle = \varphi_{\mathrm{FCC}}^\circ(g(\zeta))\,,
 \end{align*}
 also relying on the property $g^T=g^{-1}$. Therefore, we can assume that $ 0\leq \zeta_1\leq \zeta_2\leq \zeta_3$. Thus, if we want to maximize $\langle\zeta,\nu\rangle$ under the condition $\varphi_{\mathrm{FCC}}(\nu)\leq 1$, we can as well assume that $0\leq \nu_1\leq \nu_2\leq \nu_3$, so that condition $\varphi_{\mathrm{FCC}}(\nu)\leq 1$ becomes equivalent to
 \begin{align*}
 4\nu_3 + 2\nu_2 \leq 1\,.
 \end{align*}
 Therefore, noting that any linear function attains its maximum at the extreme points of a convex set and consulting Figure~\ref{fig:optimization}, we obtain
 \begin{align*}
\underset{4\nu_3+2\nu_2\leq 1}{ \max_{0\leq \nu_1\leq \nu_2\leq \nu_3}} \zeta_1 \nu_1 + \zeta_2 \nu_2 + \zeta_3\nu_3 = \underset{4\nu_3+2\nu_2\leq 1}{ \max_{0\leq\nu_2\leq \nu_3}} (\zeta_1+\zeta_2)\nu_2 +\zeta_3\nu_3 &= \max\left\{\frac{1}{4}\zeta_3, \frac{1}{6}(\zeta_1+\zeta_2+\zeta_3)\right\} \\&= \max\left\{\frac{1}{4}\|\zeta\|_\infty, \frac{1}{6}\|\zeta\|_1\right\}\,.
 \end{align*}
 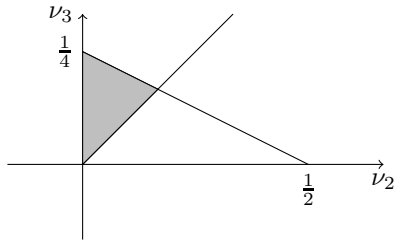
\begin{figure}[htp]
 \begin{tikzpicture}
 \draw[->](0,-1)--++(0,3);
 \draw(0,2) node[anchor=east]{$\nu_3$};
  \draw(0,1.5) node[anchor=east]{$\frac{1}{4}$};
 \draw[->](-1,0)--++(5,0);
 \draw(4,0) node[anchor=north]{$\nu_2$};
  \draw(3,0) node[anchor=north]{$\frac{1}{2}$};
 \draw(0,1.5)--(3,0);
 \draw(0,0)--++(2,2);
 \draw[fill=gray!50!white](0,0)--(0,1.5)--(1,1)--(0,0);
 \end{tikzpicture}
 \caption{The set $\{0\leq \nu_2\leq \nu_3\} \cap \{4\nu_3 +2\nu_1\leq 1\}$ depicted in gray.}
 \label{fig:optimization}
 \end{figure}
This is the desired formula \eqref{eq:FCCdual} and concludes Step~2.
\end{step}\\ \EEE
\begin{step}{3}(Calculation of \eqref{eq:surfaceFCC}) 
Note that the set $W_{\varphi_{\mathrm{FCC}}}$ is the intersection of a cube $\|\zeta\|_\infty\leq 4$ with an octahedron $\|\zeta\|_1 \leq 6$, see Fig.~\ref{fig:WFCC}. Its boundary has $6$ square faces, where $\nu = \pm (1,0,0)$ (resp.~$\pm(0,1,0)$ or $\pm(0,0,1)$)  and $8$ hexagonal faces, where $\nu = \frac{1}{\sqrt{3}}(\pm 1 \pm 1 \pm 1)$. First, we consider the set where $\nu=(1,0,0)$, the other cases where $\varphi_{\mathrm{FCC}}^\circ(\zeta)=\frac{1}{4}\|\zeta\|_\infty=1$ contributing with the same value. The square is given by 
\begin{align*}
S_1^+=\{(4,\zeta_2,\zeta_3) \colon |\zeta_2|+|\zeta_3|\leq 2\} = \left\{\frac{1}{4}\|\zeta\|_\infty= \frac{1}{4}\zeta_1=1 \right\}\cap \left\{\frac{1}{6}\|\zeta\|_1 \leq 1\right\}\,.
\end{align*}
Therefore, $\mathcal{H}^2(S_1^+)= 8$ and $\varphi_{\mathrm{FCC}}((1,0,0)) = 4$. Similarly, we obtain the same measure and value of $\varphi_{\mathrm{FCC}}$ for the other squares $S_1^-$, $S_2^\pm$,$S_3^\pm$, where $\nu$ is (up to sign) one of the coordinate unit vectors. Hence, 
\begin{align}\label{eq:intSi}
\sum_{i=1}^3\int_{S_i^+} \varphi_{\mathrm{FCC}}(\nu)\,\mathrm{d}\mathcal{H}^2+ \sum_{i=1}^3\int_{S_i^-} \varphi_{\mathrm{FCC}}(\nu)\,\mathrm{d}\mathcal{H}^2 = 6\cdot 8\cdot 4 = 3\cdot 2^6\,.
\end{align}
Next, we consider the contribution of a hexagon. We consider the hexagon contained in the set $\zeta_i \geq 0$ for all $i$. Here, we have $\nu=\frac{1}{\sqrt{3}}(1,1,1)$ and $\varphi_{\mathrm{FCC}}(\nu)=2\sqrt{3}$. The $6$ sides of the hexagon have all side-length $2\sqrt{2}$. To see this, there are sides of the form $(4,2-t,t), t \in [0,2]$ or $(4-t,0,2+t), t\in [0,2]$ and their permutations (up to identifying $t$ with $2-t$ in the first case and $4-t$ and $2+t$ in the second case). An equilateral hexagon $H$ of side-length $2\sqrt{2}$ satisfies $\mathcal{H}^2(H) = 12\sqrt{3}$. Labeling the hexagons by $H_i$, $i=0,\ldots,7$, we obtain
\begin{align}\label{eq:intHi}
\sum_{i=0}^7 \int_{H_i} \varphi_{\mathrm{FCC}}(\nu)\,\mathrm{d}\mathcal{H}^2 = 8 \cdot \mathcal{H}^2(H_i) \cdot \varphi_{\mathrm{FCC}}\left(\frac{1}{\sqrt{3}}(\pm 1, \pm 1,\pm 1)\right) = 8  \cdot 12\sqrt{3} \cdot 2\sqrt{3} = 3^2 \cdot 2^6\,.
\end{align}
Using \eqref{eq:intSi} and \eqref{eq:intHi}, we obtain
\begin{align}\label{eq:intWFCC}
\int_{\partial W_{\mathrm{FCC}}} \varphi_{\mathrm{FCC}}(\nu) \,\mathrm{d}\mathcal{H}^2= 3\cdot 2^6 + 3^2 \cdot 2^6= 3\cdot 2^8\,.
\end{align}
\BBB Let $C:=  \{\zeta \in \mathbb{R}^3 \colon \zeta_i \geq 0 \text{ for all } i=1,2,3 \text{ and }  \frac{1}{4}\|\zeta\|_\infty \geq \frac{1}{6}\|\zeta\|_1\} $ and $C^c :=\{\zeta \in \mathbb{R}^3 \colon \zeta_i \geq 0 \text{ for all } i=1,2,3 \text{ and }  \frac{1}{4}\|\zeta\|_\infty < \frac{1}{6}\|\zeta\|_1\} $.
 We split the calculation of the volume $W \cap \{\zeta \in \mathbb{R}^3 \colon \zeta_i \geq 0 \text{ for all i}\}$ into the set $C\cap W_{{\mathrm{FCC}}}$ and $C^c \cap W_{\mathrm{FCC}}$. \EEE Noting that on this set $|\nabla \varphi^\circ_{\mathrm{FCC}}(\zeta)|=\frac{1}{4}$ $\mathcal{L}^3$-a.e. on $C$, due to the coarea-formula, we have
\begin{align*}
|\{C\cap W_{\mathrm{FCC}}\}| &= 4\int_{C\cap W_{\mathrm{FCC}}} |\nabla \varphi_{\mathrm{FCC}}^\circ(\zeta)|\,\mathrm{d}\zeta = 4 \int_0^1 \mathcal{H}^2(C\cap \{\varphi^\circ_{\mathrm{FCC}}(\zeta)=s\})\,\mathrm{d}s \\&= \int_0^14\cdot s^2 \cdot 6\,\mathrm{d}s=8\,.
\end{align*}
Here we used that, $C\cap \{\varphi^\circ_{\mathrm{FCC}}(\zeta)=s\}= s(S_1^+\cup S_2^+\cup S_3^+)\cap \{\zeta_i \geq 0\}$ and the scaling properties of the $2$-dimensional Hausdorff-measure. On the other hand, using that $|\nabla \varphi^\circ_{\mathrm{FCC}}(\zeta)|=\frac{\sqrt{3}}{6}$ $\mathcal{L}^3$-a.e. on $C^c$, we have \BBB
\begin{align*}
|\{C^c\cap W_{\mathrm{FCC}}\}| &= 2\sqrt{3} \int_{C^c\cap W_{\mathrm{FCC}}} |\nabla \varphi^\circ_{\mathrm{FCC}}(\zeta)|\,\mathrm{d}\zeta = 2\sqrt{3}\int_0^1\mathcal{H}^2(C^c\cap \{\varphi^\circ_{\mathrm{FCC}}(\zeta)=s\})\,\mathrm{d}s\\&= 2\sqrt{3}\int_0^1 s^2\cdot 12\sqrt{3}\,\mathrm{d}s =3\cdot 2^3\,.
\end{align*} 
Taking into account also the sets $\{\pm\zeta_i\geq 0\}$, we obtain
\begin{align*}
|W_{\mathrm{FCC}}| = 8(8+3\cdot 2^3)=2^8 \,.
\end{align*}\EEE
Now, this together with $\eqref{eq:intWFCC}$ yields \eqref{eq:surfaceFCC}.
\end{step}
\end{proof}

\begin{proof}[Proof of Proposition \ref{prop:main2}] We divide the proof into several steps. First, we calculate $\varphi_{\mathrm{HCP}}$. Then, we calculate $\varphi^\circ_{\mathrm{HCP}}$. Lastly, we calculate \eqref{eq:surfaceHCP}.\\ \noindent\begin{step}{1}(Calculation of $\varphi_{\mathrm{HCP}}$) We make use of Proposition \ref{prop:representation} in order to calculate $\varphi_{\mathrm{HCP}}$. First of all, due to \eqref{def:periodicitycellHCP}, note that 
\begin{align}\label{eq:VolumeHCP}
|T_{{\mathrm{HCP}}}|= \frac{2}{3}\sqrt{6} \cdot \frac{1}{2}\sqrt{3} = \sqrt{2}\,.
\end{align}
Given $u \colon \mathcal{L}_{\mathrm{FCC}} \to \mathbb{R}$ such that $u(\cdot)-\langle\nu,\cdot\rangle$ is $T_{\mathrm{HCP}}$-periodic we have that $u(x+e_i) = u(x)+\langle e_i,\nu\rangle$ for all $i=,1,2,3$ and $\mathcal{L}_{\mathrm{HCP}} \cap T_{{\mathrm{HCP}}}= \{0,v_1\}$. Hence, there exist $c_1,c_2 \in \mathbb{R}$ such that
\begin{align*}
u(x) = \begin{cases} \langle x,\nu\rangle +c_1 &x \in  \mathrm{span}_{\mathbb{Z}}\left\{e_1,e_2, e_3\right\}\,;\\
\langle x_0,\nu\rangle +c_2 &x =x_0 +v_1, \text{ with } x_0 \in  \mathrm{span}_{\mathbb{Z}}\left\{e_1,e_2, e_3\right\}\,.
\end{cases}
\end{align*}
Setting $c_2-c_1=t$, recalling \eqref{eq:neighbourhoodHCP1} and \eqref{eq:neighbourhoodHCP2}, we therefore obtain 
\begin{align*}
F_{\mathcal{L}_{\mathrm{HCP}}}(u,T_{{\mathrm{HCP}}})&= 2\left(|\langle e_1,\nu\rangle|+|\langle e_2,\nu\rangle|+|\langle e_1-e_2,\nu\rangle| \right) +|t| +|t-\langle e_1,\nu\rangle| +|t-\langle e_2,\nu\rangle| \\&\quad+ |t-\langle e_3,\nu\rangle| +|t-\langle e_3+e_1,\nu\rangle| +|t-\langle e_3+e_2,\nu\rangle| \,.
\end{align*}
Employing Proposition \ref{prop:representation} and \eqref{eq:VolumeHCP}, we have
\begin{align}\label{eq:HCPmin}
\varphi_{\mathrm{HCP}}(\nu) = \sqrt{2}\left(|\langle e_1,\nu\rangle|+|\langle e_2,\nu\rangle|+|\langle e_1-e_2,\nu\rangle| \right) +\frac{1}{\sqrt{2}}\min_{t\in \mathbb{R}} g_\nu(t)\,,
\end{align}
where
\begin{align*}
g_\nu(t) :=|t| +|t-\langle e_1,\nu\rangle| +|t-\langle e_2,\nu\rangle|+ |t-\langle e_3,\nu\rangle| +|t-\langle e_3+e_1,\nu\rangle| +|t-\langle e_3+e_2,\nu\rangle|\,.
\end{align*}
Next,  we show that
\begin{align}\label{eq:calcgnut}
\min_{t\in \mathbb{R}}g_\nu(t) = |\langle e_3,\nu\rangle| +2\max\{|\langle e_1,\nu\rangle|,\, |\langle e_2,\nu\rangle|,\,|\langle e_1-e_2,\nu\rangle|, \,|\langle e_3,\nu\rangle|\} \,.
\end{align}
Note that if \eqref{eq:calcgnut} is shown, \eqref{eq:varphiHCP} is proven and Step 1 is concluded. In order to prove \eqref{eq:calcgnut}, we first note that $g_\nu(t)$ is a piecewise affine function such that $g_\nu(t) \to +\infty$ as $|t|\to +\infty$. Hence, it attains its minimum at a point of non-differentiability. The function $g_\nu$ is not differentiable for $t\in \{0,\langle e_1,\nu\rangle,\langle e_2,\nu\rangle, \langle e_3,\nu\rangle,\langle e_3+ e_1,\nu\rangle,\langle e_3 +e_2,\nu\rangle\}$ and therefore
\begin{align*}
\min_{t\in \mathbb{R}}g_\nu(t) = |\langle e_3,\nu\rangle| + \min\{f_k(\nu) \colon k\in \{0,\ldots,5\}\}\,,
\end{align*}
where
\begin{align*}
f_0(\nu)&=|\langle e_1,\nu\rangle| + |\langle e_2,\nu\rangle| + |\langle e_3+ e_1,\nu\rangle| + |\langle e_3 + e_2,\nu\rangle|\,,\\
f_1(\nu)&=|\langle e_1,\nu\rangle| + |\langle e_1- e_2,\nu\rangle| + |\langle e_3- e_1,\nu\rangle| + |\langle e_3 + e_2-e_1,\nu\rangle| \,,\\
f_2(\nu)&=|\langle e_2,\nu\rangle| + |\langle e_1- e_2,\nu\rangle| + |\langle e_3- e_2,\nu\rangle| + |\langle e_3 + e_1-e_2,\nu\rangle| \,,\\
f_3(\nu)&=|\langle e_1,\nu\rangle| + |\langle e_2,\nu\rangle| + |\langle e_3-e_1,\nu\rangle| + |\langle e_3 -e_2,\nu\rangle| \,,\\
f_4(\nu)&=|\langle e_1,\nu\rangle| + |\langle e_1- e_2,\nu\rangle| + |\langle e_3+ e_1,\nu\rangle| + |\langle e_3 + e_1-e_2,\nu\rangle| \,,\\
f_5(\nu)&=|\langle e_2,\nu\rangle| + |\langle e_1- e_2,\nu\rangle| + |\langle e_3+e_2,\nu\rangle| + |\langle e_3 + e_2-e_1,\nu\rangle|\,.
\end{align*}
It is easy to see that
\begin{align*}
\min_{t\in \mathbb{R}}g_\nu(t) = |\langle e_3,\nu\rangle|+ \min\big\{|\langle e_1,R_k\nu\rangle| &+ |\langle e_2,R_k\nu\rangle| + |\langle e_3+ e_1,R_k\nu\rangle| \\&+ |\langle e_3 + e_2,R_k\nu\rangle| \colon k\in \{0,\ldots,5\}\big\}\,,
\end{align*}
where $R_k$ is the rotation of angle $k\pi/3$ around the $x_3$-axis. Noting also that $\min_{t\in \mathbb{R}} g_\nu(t)= \min_{t\in \mathbb{R}} g_{-\nu}(t)$, it is not restrictive to assume that $\langle e_1,\nu\rangle \geq 0, \langle e_2,\nu\rangle \geq 0, \langle e_3,\nu\rangle \geq 0$. We only consider the case, where $ \langle e_1,\nu\rangle \geq \langle e_2,\nu\rangle \geq  \langle e_3,\nu\rangle \geq 0$, the other being dealt with in a similar fashion. In this case we have $f_0(\nu)\geq f_3(\nu)$, $f_4(\nu)\geq f_5(\nu)$, and
\begin{align*}
f_1(\nu) &= \langle e_1,\nu\rangle +\langle e_1-e_2,\nu\rangle + \langle e_1-e_3,\nu\rangle +|\langle e_3+e_2- e_1,\nu\rangle| \\&= 2\langle e_1,\nu\rangle + \langle e_1-e_2-e_3,\nu\rangle + |\langle e_3+e_2-e_1,\nu\rangle| \geq 2\langle e_1,\nu\rangle\,;\\
f_2(\nu) &= \langle e_2,\nu\rangle + \langle e_1-e_2,\nu\rangle + \langle e_2-e_3,\nu\rangle + \langle e_3+e_1-e_2,\nu\rangle =2\langle e_1,\nu\rangle\,;\\
f_3(\nu) &= \langle e_1,\nu\rangle + \langle e_2,\nu\rangle + \langle e_1-e_3,\nu\rangle + \langle e_2-e_3,\nu\rangle = 2\langle e_1,\nu\rangle + 2\langle e_2,\nu\rangle -2\langle e_3,\nu\rangle \geq 2\langle e_1,\nu\rangle\,;\\ f_5(\nu) &= \langle e_2,\nu\rangle +\langle e_1-e_2,\nu\rangle + \langle e_3+e_2,\nu\rangle + |\langle e_3+e_2-e_1,\nu\rangle| \\&= 2\langle e_1,\nu\rangle + \langle e_3+e_2-e_1,\nu\rangle + |\langle e_3+e_2-e_1,\nu\rangle| \geq 2\langle e_1,\nu\rangle\,.
\end{align*}
Hence, we see that \eqref{eq:calcgnut} holds true. This together with \eqref{eq:HCPmin} establishes \eqref{eq:varphiHCP} and concludes Step~1.
\end{step}\\
\noindent\begin{step}{2}(Calculation of $\varphi^\circ_{\mathrm{HCP}}$) In order to calculate $\varphi^\circ_{\mathrm{HCP}}$, we exploit the symmetries of $\varphi^\circ_{\mathrm{HCP}}$. Note that 
\begin{align}\label{eq:HCPsym1}
\varphi_{\mathrm{HCP}}\left(-\begin{pmatrix}
\nu_1\\
\nu_2\\
\nu_3
\end{pmatrix}\right) =\varphi_{\mathrm{HCP}}\left(\begin{pmatrix}
\nu_1\\
\nu_2\\
-\nu_3
\end{pmatrix}\right) = \varphi_{\mathrm{HCP}}\left(\begin{pmatrix}
-\nu_1\\
\nu_2\\
\nu_3
\end{pmatrix}\right) = \varphi_{\mathrm{HCP}}\left(\begin{pmatrix}
\nu_1\\
\nu_2\\
\nu_3
\end{pmatrix}\right)\,.
\end{align}
Given $\zeta \in \mathbb{R}^3$ we can find $R= T_1^{\alpha_1}\circ T_2^{\alpha_2}\circ T_3^{\alpha_3}$, $\alpha_i \in \{0,1\}$ such that $(R\zeta)_i \geq 0 $ for all $i$. Thus, 
\begin{align}\label{eq:HCPdualsym}
\begin{split}
\varphi_{\mathrm{HCP}}^\circ(\zeta)&=\max_{\varphi_{\mathrm{HCP}}(\nu)\leq 1} \langle\nu,\zeta\rangle  = \max_{\varphi_{\mathrm{HCP}}(\nu)\leq 1} \langle R\nu,R\zeta\rangle \\&= \max_{\varphi_{\mathrm{HCP}}(R^{-1}\nu)\leq 1} \langle \nu,R\zeta\rangle = \max_{\varphi_{\mathrm{HCP}}(\nu)\leq 1} \langle \nu,R\zeta\rangle=\varphi_{\mathrm{HCP}}^\circ(R\zeta)\,.
\end{split}
\end{align}
It therefore suffices to calculate $\varphi_{\mathrm{HCP}}^\circ$ for $\zeta \in \mathbb{R}^3$ such that $\zeta_i\geq 0$.  This together with \eqref{eq:HCPsym1} implies that if $\nu=(\nu_1,\nu_2,\nu_3)$ is such that $\varphi_{{\mathrm{HCP}}}(\nu)\leq 1$ and 
\begin{align*}
\langle \nu,\zeta\rangle = \max_{\varphi_{\mathrm{HCP}}(\nu) \leq 1} \langle\nu,\zeta\rangle\,,
\end{align*}
then $\nu_i \geq 0$ for all $i$. Additionally, a maximizer $\nu$ can be chosen such that $\varphi_{\mathrm{HCP}}$ is not differentiable at $\nu$. Therefore, there are the following cases to consider:
\begin{itemize}
\item[(a)] $\langle e_1-e_2,\nu\rangle=0$;
\item[(b)] $\langle e_1-e_3,\nu\rangle=0, \langle e_1 - e_2,\nu\rangle\geq 0$;
\item[(c)] $\langle e_2-e_3,\nu\rangle=0, \langle e_3 - e_1,\nu\rangle\geq 0$;
\item[(d)] $\langle e_3,\nu\rangle=0$;
\item[(e)] $\langle e_1,\nu\rangle=0$.
\end{itemize}
Here, we point out that the points on the boundary $\nu_2=0$ are excluded as possible maximum points by arguing in the following way: If there were a point $\nu$ such that $\nu_2=0$, then $\varphi_{\mathrm{HCP}}(\nu)$ would either be differentiable and thus $\nu$ would not be a maximum point or $\nu$ would satisfy one of the cases (a)-(e).\\
\noindent\emph{Maximum of case {\rm (a)}.} Since $\langle e_1-e_2,\nu\rangle=0$, we have $\nu_1= \sqrt{3}\nu_2$. Hence, $\nu= (t, \frac{1}{\sqrt{3}}t,s)$ for some $t,s \geq 0$. Now, using \eqref{eq:varphiHCP}, we have
\begin{align*}
\varphi_{\mathrm{HCP}}(\nu) = \sqrt{2}\left(2\nu_1 + \frac{1}{3}\sqrt{6}\nu_3 + \max\{\nu_1,\frac{2}{3}\sqrt{6}\nu_3\} \right)\,.
\end{align*}
(a.1) $t \geq \frac{2}{3}\sqrt{6}s$:
Since the maximum is attained for $\varphi_{\mathrm{HCP}}(\nu)=1$, we have $t = \frac{1}{3\sqrt{2}} -\frac{1}{9}\sqrt{6}s$. Now, $t\geq 0$ together with $t \geq \frac{2}{3}\sqrt{6}s$ implies $ 0\leq s \leq \frac{3}{14\sqrt{3}}$. Noting that 
\begin{align*}
\langle\nu,\zeta\rangle = t\left(\zeta_1 +\frac{1}{\sqrt{3}}\zeta_2\right) +s \zeta_3 = \left(\frac{1}{3\sqrt{2}}-\frac{1}{9}\sqrt{6}s\right)\left(\zeta_1 +\frac{1}{\sqrt{3}}\zeta_2\right) + s\zeta_3\,,
\end{align*}
we obtain
\begin{align}\label{eq:maxa1}
\max_{\nu \text{ sat. (a.1)} }\langle\nu,\zeta\rangle = \max\left\{ \frac{2}{7\sqrt{2}}\left(\zeta_1 +\frac{1}{\sqrt{3}}\zeta_2 + \frac{3}{2\sqrt{6}}\zeta_3\right), \frac{1}{3\sqrt{2}}\left(\zeta_1 +\frac{1}{\sqrt{3}}\zeta_2\right)\right\}\,.
\end{align}
(a.2) $t \leq \frac{2}{3}\sqrt{6}s$: 
Using $\varphi_{\mathrm{HCP}}(\nu)=1$, we obtain $t=\frac{1}{2\sqrt{2}} -\frac{\sqrt{6}}{2}s$. Now, $t\geq 0$ together with $t\leq \frac{2}{3}\sqrt{6}s$ implies $\frac{3}{14\sqrt{3}}\leq s\leq \frac{1}{2\sqrt{3}}$. Noting that 
\begin{align*}
\langle\nu,\zeta\rangle = t\left(\zeta_1 +\frac{1}{\sqrt{3}}\zeta_2\right) +s \zeta_3 = \left(\frac{1}{3\sqrt{2}}-\frac{1}{9}\sqrt{6}s\right)\left(\zeta_1 +\frac{1}{\sqrt{3}}\zeta_2\right) + s\zeta_3\,,
\end{align*}
we obtain
\begin{align}
\max_{\nu \text{ sat. } (a.2)}\langle\nu,\zeta\rangle =\max\left\{ \frac{2}{7\sqrt{2}}\left(\zeta_1 +\frac{1}{\sqrt{3}}\zeta_2 + \frac{3}{2\sqrt{6}}\zeta_3\right), \frac{1}{2\sqrt{3}}\zeta_3 \right\}\,.
\end{align}
\noindent \emph{Maximum of case {\rm (b)}.} Since $\langle e_1-e_3,\nu\rangle =0$, we have $\nu_1 = \frac{2}{3}\sqrt{6}\nu_3$. Hence, $\nu=(t,s, \frac{3}{2\sqrt{6}}t)$ for some $t,s \geq 0$. Now using \eqref{eq:varphiHCP}, we have
\begin{align*}
\varphi_{\mathrm{HCP}}(\nu) = \frac{7}{2}\sqrt{2}\nu_1\,.
\end{align*}
Hence, since the maximum is attained for $\varphi_{\mathrm{HCP}}(\nu)=1$, we have $\nu_1 = \frac{2}{7\sqrt{2}}$. Additionally, since $\langle e_1-e_2,\nu\rangle\geq 0$, we have $\nu_2 \leq \frac{2}{7\sqrt{6}}$, and due to the form of $\nu$, we have $\nu_3 = \frac{3}{14\sqrt{3}}$. This implies
\begin{align}
\max_{\nu \text{ sat. (b)}} \langle\nu,\zeta\rangle = \frac{2}{7\sqrt{2}}\left(\zeta_1 +\frac{1}{\sqrt{3}}\zeta_2 + \frac{3}{2\sqrt{6}}\zeta_3\right)\,.
\end{align}
\noindent \emph{Maximum of case {\rm (c)}.} Since $\langle e_2-e_3,\nu\rangle =0$, we have $\frac{1}{2}\nu_1+\frac{1}{2}\sqrt{3}\nu_2 = \frac{2}{3}\sqrt{6}\nu_3$. Now using \eqref{eq:varphiHCP}, we have
\begin{align*}
\varphi_{\mathrm{HCP}}(\nu) = \frac{7}{2}\sqrt{2}\langle e_3,\nu\rangle = \frac{14}{3}\sqrt{3}\nu_3\,.
\end{align*}
Hence, since the maximum is attained for $\varphi_{\mathrm{HCP}}(\nu)=1$, we have $\nu_3 = \frac{3}{14\sqrt{3}}$. Additionally, since $\langle e_3-e_1,\nu\rangle\geq 0$, we have $\nu_1 \leq \frac{2}{7\sqrt{2}}$. Due to the form of $\langle e_2-e_3,\nu\rangle=0$, we have $\nu_2 = \frac{4}{7\sqrt{6}}-\frac{1}{\sqrt{3}}\nu_1$. Note that $\nu_2 \geq 0$ for all $0\leq \nu_1\leq \frac{2}{7\sqrt{2}}$. Therefore
\begin{align*}
\langle \nu,\zeta\rangle = \nu_1\zeta_1 + \left(\frac{4}{7\sqrt{
6}}-\frac{1}{\sqrt{3}}\nu_1\right)\zeta_2 + \frac{3}{14\sqrt{3}}\zeta_3\,.
\end{align*}
This implies 
\begin{align}
\max_{\nu \text{ sat. (c)}} \langle\nu,\zeta\rangle =\max\left\{\frac{4}{7\sqrt{6}}\zeta_2+\frac{3}{14\sqrt{3}}\zeta_3, \frac{2}{7\sqrt{2}}\left(\zeta_1 +\frac{1}{\sqrt{3}}\zeta_2 + \frac{3}{2\sqrt{6}}\zeta_3\right)\right\}\,.
\end{align}
\noindent\emph{Maximum of case {\rm (d)}.} We have $\nu_3=0$ and therefore 
\begin{align*}
\varphi_{\mathrm{HCP}}(\nu) = \sqrt{2}\left(\langle e_1,\nu\rangle+ \langle e_2,\nu\rangle +|\langle e_1-e_2,\nu\rangle| + \max\{\langle e_1,\nu\rangle, \langle e_2,\nu\rangle\}\right)
\end{align*}
We distinguish to cases
\begin{itemize}
\item[(d.1)] $\langle e_1-e_2,\nu\rangle \geq 0$;
\item[(d.2)] $\langle e_1-e_2,\nu\rangle \leq 0$.
\end{itemize}
\emph{Maximum of case {\rm (d.1)}.} In the case $\langle e_1-e_2,\nu\rangle \geq 0$ we  have $\varphi_{\mathrm{HCP}}(\nu) = 3\sqrt{2}\nu_1$ and therefore, since $\varphi_{\mathrm{HCP}}(\nu)=1$, $\nu_1 =\frac{1}{3\sqrt{2}}$.  The inequality $\langle e_1-e_2,\nu\rangle \geq 0$ implies that $0\leq \nu_2 \leq \frac{1}{\sqrt{3}}\nu_1= \frac{1}{3\sqrt{6}}$. Hence,
\begin{align}
\max_{\nu \text{ sat. (d.1)}} \langle\nu,\zeta\rangle = \frac{1}{3\sqrt{2}}\left(\zeta_1 +\frac{1}{\sqrt{3}}\zeta_2\right)\,.
\end{align}
\noindent\emph{Maximum of case {\rm (d.2)}.} In the case $\langle e_1-e_2,\nu\rangle \leq 0$ we  have 
\begin{align*}
\varphi_{\mathrm{HCP}}(\nu) = 3\sqrt{2}\langle e_2,\nu\rangle = \sqrt{2}\left(\frac{3}{2}\nu_1+\frac{3}{2}\sqrt{3}\nu_2\right)\,.
\end{align*}
This, together with $\varphi_{\mathrm{HCP}}(\nu)=1$, implies, $\nu_1 = \frac{2}{3\sqrt{2}}-\sqrt{3}\nu_2$ and therefore $\nu_2 \leq \frac{2}{3\sqrt{6}}$. Additionally, since $\langle e_1-e_2\rangle \leq 0$, we have $\frac{1}{3\sqrt{6}}\leq \nu_2$. Therefore,
\begin{align*}
\langle \nu,\zeta\rangle = \nu_1\zeta_1 + \nu_2 \zeta_2 = \left(\frac{3}{2\sqrt{2}}- \sqrt{3}\nu_2\right)\zeta_1+\nu_2\zeta_2\,.
\end{align*}
This implies
\begin{align}
\max_{\nu \text{ sat. (d.2)}} \langle\nu,\zeta\rangle = \max\left\{\frac{2}{3\sqrt{6}}\zeta_2,\frac{1}{3\sqrt{2}}\left(\zeta_1 +\frac{1}{\sqrt{3}}\zeta_2\right)\right\}\,.
\end{align}
\noindent\emph{Maximum of case {\rm (e)}.} In the case $\nu_1=0$ we  have 
\begin{align*}
\varphi_{\mathrm{HCP}}(\nu) = \sqrt{2}\left(\sqrt{3}\nu_2 +\frac{1}{3}\sqrt{6}\nu_3 +\max\left\{\frac{\sqrt{3}}{2}\nu_2,\frac{2}{3}\sqrt{6}\nu_3\right\}\right)\,.
\end{align*}
We distinguish between two cases:
\begin{itemize}
\item[(e.1)] $\langle e_2,\nu\rangle \geq \langle e_3,\nu\rangle$;
\item[(e.2)] $\langle e_2,\nu\rangle \leq \langle e_3,\nu\rangle$.
\end{itemize}
\noindent\emph{Maximum of case {\rm (e.1)}.} In this case, we have
\begin{align*}
\varphi_{\mathrm{HCP}}(\nu) = \sqrt{2}\left(\frac{3}{2}\sqrt{3}\nu_2+\frac{1}{3}\sqrt{6}\nu_3\right)\,.
\end{align*}
Therefore, since $\varphi_{\mathrm{HCP}}(\nu)=1$, we have $\nu_2 = \frac{2}{3\sqrt{6}} -\frac{2}{9}\sqrt{2}\nu_3$. Hence, $\nu_3 \leq \frac{3}{2\sqrt{3}}$. Additionally, since $\langle e_2-e_3,\nu\rangle \geq 0$, we have $\nu_3\leq \frac{3}{14\sqrt{3}}$. Therefore,
\begin{align*}
\langle\nu,\zeta\rangle = \nu_2\zeta_2 + \nu_3\zeta_3 = \left(\frac{2}{3\sqrt{6}}-\frac{2}{9}\sqrt{2}\nu_3\right)\zeta_2 +\nu_3\zeta_3\,.
\end{align*} 
Hence, 
\begin{align}
\max_{\nu \text{ sat. (e.1)}} \langle\nu,\zeta\rangle = \max\left\{\frac{2}{3\sqrt{6}}\zeta_2,\frac{4}{7\sqrt{6}}\zeta_2+\frac{3}{14\sqrt{3}}\zeta_3\right\}\,.
\end{align}
\noindent\emph{Maximum of case {\rm (e.2)}.} In this case, we have
\begin{align*}
\varphi_{\mathrm{HCP}}(\nu) = \sqrt{2}\left(\sqrt{3}\nu_2+\sqrt{6}\nu_3\right)\,.
\end{align*}
Therefore, since $\varphi_{\mathrm{HCP}}(\nu)=1$, we have $\nu_2 = \frac{1}{\sqrt{6}} -\sqrt{2}\nu_3$. Hence, $\nu_3 \leq \frac{1}{2\sqrt{3}}$. Additionally, since $\langle e_2-e_3,\nu\rangle \leq 0$, we have $\nu_3\geq \frac{3}{14\sqrt{3}}$. Therefore,
\begin{align*}
\langle\nu,\zeta\rangle = \nu_2\zeta_2 + \nu_3\zeta_3 = \left(\frac{1}{3\sqrt{6}}-\sqrt{2}\nu_3\right)\zeta_2 +\nu_3\zeta_3\,.
\end{align*} 
Hence, 
\begin{align}\label{eq:maxe2}
\max_{\nu \text{ sat. (e.1)}} \langle\nu,\zeta\rangle = \max\left\{\frac{1}{2\sqrt{3}}\zeta_3,\frac{4}{7\sqrt{6}}\zeta_2+\frac{3}{14\sqrt{3}}\zeta_3\right\}\,.
\end{align}

 Exploiting \eqref{eq:maxa1}--\eqref{eq:maxe2}, and \eqref{eq:HCPdualsym}, we obtain \eqref{eq:HCPdual}. This concludes Step 2.
\end{step}\\
\noindent \begin{step}{3}(Calculation of \eqref{eq:surfaceHCP})  In order to calculate \eqref{eq:surfaceHCP}, we split the calculation of $ \partial^* W_{\mathrm{HCP}}= \{\varphi^\circ_{\mathrm{HCP}}(\zeta)=1\}$ into different sets, where the maximum of $\varphi_{\mathrm{HCP}}^\circ$ is attained. We consider the following cases
\begin{itemize}
\item[(a)] $A_a:=\{ \zeta \in \mathbb{R}^3 \colon\varphi_{\mathrm{HCP}}^\circ(\zeta)=\frac{1}{2\sqrt{3}}|\zeta_3|=1\}$;
\item[(b)] $A_b:=\{\zeta \in \mathbb{R}^3 \colon\varphi_{\mathrm{HCP}}^\circ(\zeta)=\frac{2}{3\sqrt{6}}|\zeta_2|=1\}$;
\item[(c)] $A_c:=\{\zeta \in \mathbb{R}^3 \colon\varphi_{\mathrm{HCP}}^\circ(\zeta)=\frac{4}{7\sqrt{6}}|\zeta_2| +\frac{3}{14\sqrt{3}}|\zeta_3| =1\}$;
\item[(d)] $A_d:=\{\zeta \in \mathbb{R}^3 \colon\varphi_{\mathrm{HCP}}^\circ(\zeta)=\frac{1}{3\sqrt{2}}(|\zeta_1| +\frac{1}{\sqrt{3}}|\zeta_2|) =1\}$;
\item[(e)] $A_e:=\{\zeta \in \mathbb{R}^3 \colon\varphi_{\mathrm{HCP}}^\circ(\zeta)=\frac{2}{7\sqrt{2}}(|\zeta_1| +\frac{1}{\sqrt{3}}|\zeta_2|+\frac{3}{2\sqrt{6}}|\zeta_3|) =1\}$.
\end{itemize}
In each of the cases, one can determine the area, shape and normal of the set, by invoking the condition that the maximum is attained $\varphi^\circ_{\mathrm{HCP}}$ for the respective function and therefore all the other functions $f$ in the definition of $\varphi^\circ_{\mathrm{HCP}}$ satisfy $f\leq 1$. In the following, we only collect the results, since the calculations are elementary (but very long).

\noindent\emph{Calculations for case {\rm (a)}.} In this case, we see that $\nu=(0,0,\pm 1)$ $\mathcal{H}^2$-a.e., since this set is contained in the level set of the function $|\zeta_3|=c$ for some $c >0$. Additionally, we see that the set is a union of two hexagons of side length  $2\sqrt{2}$. Therefore, for each of the two hexagons $H_{i}$ we have $\mathcal{H}^2(H_i) =12\sqrt{3}$. Furthermore, $\varphi_{\mathrm{HCP}}(\nu)=2\sqrt{3}$. Hence
\begin{align}\label{eq:intHCPa}
\int_{A_a} \varphi_{\mathrm{HCP}}(\nu)\,\mathrm{d}\mathcal{H}^2 = 2\cdot 12\sqrt{3}\cdot 2\sqrt{3}= 2^4 \cdot 3^2\,.
\end{align}

\noindent \emph{Calculations for case {\rm (b)}.} In this case, we see that $\nu=(0,\pm 1,0)$ $\mathcal{H}^2$-a.e., since this set is contained in the level set of the function $|\zeta_2|=c$ for some $c >0$. Additionally, we see that the set is a union of two rectangles with side lengths  $3\sqrt{2}$ and $\frac{4}{3}\sqrt{3}$. Therefore, for each of the two rectangles $S_{i}$ we have $\mathcal{H}^2(S_i) =4\sqrt{6}$. Furthermore, $\varphi_{\mathrm{HCP}}(\nu)=\frac{3}{2}\sqrt{6}$. Hence
\begin{align}
\int_{A_b} \varphi_{\mathrm{HCP}}(\nu)\,\mathrm{d}\mathcal{H}^2 = 2\cdot 4\sqrt{6}\cdot \frac{3}{2}\sqrt{6}= 2^3 \cdot 3^2\,.
\end{align}
\noindent \emph{Calculations for case {\rm (c)}.} In this case, we see that $\nu=(3/41)^{1/2} (0,\pm 8/\sqrt{6},\pm\sqrt{3})$ $\mathcal{H}^2$-a.e., since this set is contained in the level set of the function $\frac{4}{7\sqrt{6}}|\zeta_2| +\frac{3}{14\sqrt{3}}|\zeta_3|=c$ for some $c >0$. Additionally, we see that the set is a union of four trapezoids with height $(41/6)^{1/2}$ and two parallel sides o lengths  $3\sqrt{2}$ and $2\sqrt{2}$. Therefore, for each of the four trapezoids $T_{i}$ we have $\mathcal{H}^2(T_i) =\frac{5}{2}(\frac{41}{3})^{1/2}$. Furthermore, $\varphi_{\mathrm{HCP}}(\nu)=14(\frac{3}{41})^{1/2}$. Hence
\begin{align}
\int_{A_c} \varphi_{\mathrm{HCP}}(\nu)\,\mathrm{d}\mathcal{H}^2 = 4\cdot \frac{5}{2}\left(\frac{41}{3}\right)^{1/2}\cdot 14\left(\frac{3}{41}\right)^{1/2}= 2^2\cdot 5\cdot 7\,.
\end{align}

\noindent \emph{Calculations for case {\rm (d)}.} In this case, we see that $\nu=\frac{1}{2} (\pm\sqrt{3},\pm 1,0)$ $\mathcal{H}^2$-a.e., since this set is contained in the level set of the function $|\zeta_1| +\frac{1}{\sqrt{3}}|\zeta_2|=c$ for some $c >0$. Additionally, we see that the set is a union of four rectangles with side length $3\sqrt{2}$ and $\frac{4}{3}\sqrt{3}$. Therefore, for each of the four rectangles $R_{i}$ we have $\mathcal{H}^2(R_i) =4\sqrt{6}$. Furthermore, $\varphi_{\mathrm{HCP}}(\nu)=\frac{3}{2}\sqrt{6}$. Hence
\begin{align}
\int_{A_d} \varphi_{\mathrm{HCP}}(\nu)\,\mathrm{d}\mathcal{H}^2 = 4\cdot 4\sqrt{6}\cdot \frac{3}{2}\sqrt{6}= 2^4\cdot 3^2\,.
\end{align}

\noindent \emph{Calculations for case {\rm (e)}.} In this case, we see that $\nu=2(6/41)^{1/2} (\pm 1,\pm \frac{1}{\sqrt{3}},\frac{3}{2\sqrt{6}})$ $\mathcal{H}^2$-a.e., since this set is contained in the level set of the function $|\zeta_1| +\frac{1}{\sqrt{3}}|\zeta_2|+ \frac{3}{2\sqrt{6}}|\zeta_3|=c$ for some $c >0$. Additionally, we see that the set is a union of eight trapezoids with height $(41/6)^{1/2}$ and two parallel sides o lengths  $3\sqrt{2}$ and $2\sqrt{2}$. Therefore, for each of the eight trapezoids $Z_{i}$ we have $\mathcal{H}^2(Z_i) =\frac{5}{2}(\frac{41}{3})^{1/2}$. Furthermore, $\varphi_{\mathrm{HCP}}(\nu)=14(\frac{3}{41})^{1/2}$. Hence
\begin{align}\label{eq:intHCPe}
\int_{A_e} \varphi_{\mathrm{HCP}}(\nu)\,\mathrm{d}\mathcal{H}^2 = 8\cdot \frac{5}{2}\left(\frac{41}{3}\right)^{1/2}\cdot 14\left(\frac{3}{41}\right)^{1/2}= 2^3\cdot 5\cdot 7\,.
\end{align}
Taking into account \eqref{eq:intHCPa}--\eqref{eq:intHCPe}, we obtain
\begin{align}\label{eq:bdryint}
\int_{\partial^* W_{\mathrm{HCP}}} \varphi_{\mathrm{HCP}}(\nu)\,\mathrm{d}\mathcal{H}^2 = 2^5\cdot3^2+ 2^3\cdot 3^2+2^2\cdot 5\cdot 7 + 2^4\cdot 3^2 + 2^3\cdot 5 \cdot 7=780\,.
\end{align}
Next, we need to calculate $|W_{\mathrm{HCP}}|$, since $W_{\mathrm{HCP}} = \{\varphi_{\mathrm{HCP}}^\circ \leq 1\} \cap (C_a\cup C_b \cup C_c \cup C_d \cup C_e)$, where 
\begin{align*}
&C_a = \{ \zeta \in \mathbb{R}^3 \colon \varphi^{\mathrm{\circ}}_{\mathrm{HCP}}(\zeta) = \frac{1}{2\sqrt{3}}|\zeta_3|\}\,, \\&C_b:=\{\zeta \in \mathbb{R}^3 \colon\varphi_{\mathrm{HCP}}^\circ(\zeta)=\frac{2}{3\sqrt{6}}|\zeta_2|\}\,, \\&C_c:=\{\zeta \in \mathbb{R}^3 \colon\varphi_{\mathrm{HCP}}^\circ(\zeta)=\frac{4}{7\sqrt{6}}|\zeta_2| +\frac{3}{14\sqrt{3}}|\zeta_3|\}\,,\\&C_d:=\{\zeta \in \mathbb{R}^3 \colon\varphi_{\mathrm{HCP}}^\circ(\zeta)=\frac{1}{3\sqrt{2}}(|\zeta_1| +\frac{1}{\sqrt{3}}|\zeta_2|)\}\,, \\&C_e:=\{\zeta \in \mathbb{R}^3 \colon\varphi_{\mathrm{HCP}}^\circ(\zeta)=\frac{2}{7\sqrt{2}}(|\zeta_1| +\frac{1}{\sqrt{3}}|\zeta_2|+\frac{3}{2\sqrt{6}}|\zeta_3|)\}\,.
\end{align*} 
Note that $ \mathcal{H}^2(C_\alpha \cap \{\varphi_{\mathrm{HCP}}^\circ(\zeta)=s\}) = s^2 \mathcal{H}^2(A_\alpha)$ for all $\alpha \in \{a,b,c,d,e\}$. In the set $C_a$ we have that $|\nabla \varphi_{\mathrm{HCP}}^\circ(\zeta)|= \frac{1}{2\sqrt{3}}$ $\mathcal{L}^3$-a.e.. Due to the coarea formula, we have
\begin{align}\label{eq:intCa}
\begin{split}
|C_a \cap W_{\mathrm{HCP}}| &= 2\sqrt{3}\int_{C_a \cap W_{\mathrm{HCP}}} |\nabla \varphi_{\mathrm{HCP}}^\circ(\zeta)|\,\mathrm{d}\zeta\\&= 2\sqrt{3}\int_0^1  \mathcal{H}^2(C_a \cap \{\varphi_{\mathrm{HCP}}^\circ(\zeta)=s\})\,\mathrm{d}s= \frac{2}{3}\sqrt{3}\mathcal{H}^2(A_a)= 2^4 \cdot 3\,.
\end{split}
\end{align}
In the set $C_b$, we have that $|\nabla \varphi_{\mathrm{HCP}}^\circ(\zeta)|= \frac{2}{3\sqrt{6}}$ $\mathcal{L}^3$-a.e.. Due to the coarea formula, we have
\begin{align}
\begin{split}
|C_b \cap W_{{\mathrm{HCP}}}| &= \frac{3}{2}\sqrt{6}\int_{C_b \cap W_{\mathrm{HCP}}} |\nabla \varphi_{\mathrm{HCP}}^\circ(\zeta)|\,\mathrm{d}\zeta\\&= \frac{3}{2}\sqrt{6}\int_0^1  \mathcal{H}^2(C_b \cap \{\varphi_{\mathrm{HCP}}^\circ(\zeta)=s\})\,\mathrm{d}s= \frac{1}{2}\sqrt{6}\mathcal{H}^2(A_b)= 2^3\cdot 3\,.
\end{split}
\end{align}
In the set $C_c$, we have that $|\nabla \varphi_{\mathrm{HCP}}^\circ(\zeta)|= \frac{1}{14}(41/3)^{1/2}$ $\mathcal{L}^3$-a.e.. Due to the coarea formula, we have
\begin{align}
\begin{split}
|C_c \cap W_{\mathrm{HCP}}| &= 14\left(\frac{3}{41}\right)^{1/2} \int_{C_c \cap W_{\mathrm{HCP}}} |\nabla \varphi_{\mathrm{HCP}}^\circ(\zeta)|\,\mathrm{d}\zeta\\&= 14\left(\frac{3}{41}\right)^{1/2} \int_0^1  \mathcal{H}^2(C_c \cap \{\varphi_{\mathrm{HCP}}^\circ(\zeta)=s\})\,\mathrm{d}s\\&= \frac{1}{3}14\left(\frac{3}{41}\right)^{1/2}\mathcal{H}^2(A_c)=\frac{2^2\cdot 5\cdot 7}{3} \,.
\end{split}
\end{align}
In the set $C_d$, we have that $|\nabla \varphi_{\mathrm{HCP}}^\circ(\zeta)|=\frac{2}{3\sqrt{6}}$ $\mathcal{L}^3$-a.e.. Due to the coarea formula, we have
\begin{align}
\begin{split}
|C_d\cap W_{\mathrm{HCP}}| &= \frac{3}{2}\sqrt{6} \int_{C_d \cap W_{\mathrm{HCP}}} |\nabla \varphi_{\mathrm{HCP}}^\circ(\zeta)|\,\mathrm{d}\zeta\\&=  \frac{3}{2}\sqrt{6}\int_0^1  \mathcal{H}^2(C_d \cap \{\varphi_{\mathrm{HCP}}^\circ(\zeta)=s\})\,\mathrm{d}s\\&= \frac{1}{2}\sqrt{6}\mathcal{H}^2(A_d)= 2^4\cdot 3  \,.
\end{split}
\end{align}
In the set $C_e$, we have that $|\nabla \varphi_{\mathrm{HCP}}^\circ(\zeta)|=\frac{1}{14}(41/3)^{1/2}$ $\mathcal{L}^3$-a.e.. Due to the coarea formula, we have
\begin{align}\label{eq:intCe}
\begin{split}
|C_e \cap W_{\mathrm{HCP}}| &= 14\left(\frac{3}{41}\right)^{1/2} \int_{C_e \cap W_{\mathrm{HCP}}} |\nabla \varphi_{\mathrm{HCP}}^\circ(\zeta)|\,\mathrm{d}\zeta\\&=  14\left(\frac{3}{41}\right)^{1/2}\int_0^1  \mathcal{H}^2(C_e \cap \{\varphi_{\mathrm{HCP}}^\circ(\zeta)=s\})\,\mathrm{d}s\\&= \frac{1}{3}14\left(\frac{3}{41}\right)^{1/2}\mathcal{H}^2(A_e)= \frac{2^3\cdot 5\cdot 7}{3}  \,.
\end{split}
\end{align}
Using \eqref{eq:intCa}--\eqref{eq:intCe}, we obtain $|W_{\mathrm{HCP}}| =260$. This together with \eqref{eq:bdryint} yields \eqref{eq:surfaceHCP}.
\end{step}
\end{proof}

\section{$\Gamma$-convergence analysis on the $\mathrm{FCC}$ and $\mathrm{HCP}$ lattices}\label{sec:Gammaconvergence}
In this section we prove Theorem \ref{theorem:Gammaconvergence}. In order to prove the compactness statement, we provide some preliminary lemmata about the shape of the Voronoi cells of the $\mathrm{FCC}$-lattice as well as the $\mathrm{HCP}$-lattice (see Figure~\ref{fig:VFCC}). In what follows we use the notation $\mathcal N_{\mathrm {FCC}}=\mathcal N_{{\mathcal L_\mathrm {FCC}}}(0)$ and $\mathcal N_{\mathrm {HCP}}=\mathcal N_{{\mathcal L_\mathrm {HCP}}}(0)$.

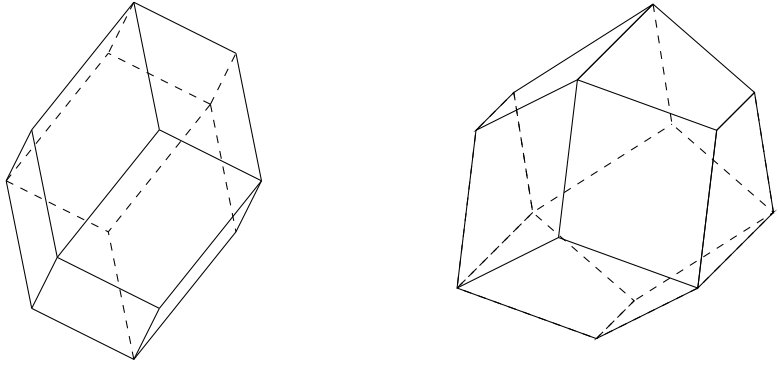
\begin{figure}[htp]
\begin{tikzpicture}[scale=.8]

\begin{scope}[scale=6,x=0.5cm,y=0.7cm,z=0.3cm]
\draw({.5*sqrt(2)},0,0)--({.25*sqrt(2)},{.25*sqrt(2)},{.25*sqrt(2)});
\draw({.5*sqrt(2)},0,0)--({.25*sqrt(2)},{.25*sqrt(2)},{-.25*sqrt(2)});
\draw({.5*sqrt(2)},0,0)--({.25*sqrt(2)},{-.25*sqrt(2)},{.25*sqrt(2)});
\draw({.5*sqrt(2)},0,0)--({.25*sqrt(2)},{-.25*sqrt(2)},{-.25*sqrt(2)});

\draw[dashed]({-.5*sqrt(2)},0,0)--({-.25*sqrt(2)},{.25*sqrt(2)},{.25*sqrt(2)});
\draw({-.5*sqrt(2)},0,0)--({-.25*sqrt(2)},{.25*sqrt(2)},{-.25*sqrt(2)});
\draw[dashed]({-.5*sqrt(2)},0,0)--({-.25*sqrt(2)},{-.25*sqrt(2)},{.25*sqrt(2)});
\draw({-.5*sqrt(2)},0,0)--({-.25*sqrt(2)},{-.25*sqrt(2)},{-.25*sqrt(2)});

\draw(0,{.5*sqrt(2)},0)--({.25*sqrt(2)},{.25*sqrt(2)},{.25*sqrt(2)});
\draw(0,{.5*sqrt(2)},0)--({.25*sqrt(2)},{.25*sqrt(2)},{-.25*sqrt(2)});
\draw[dashed](0,{.5*sqrt(2)},0)--({-.25*sqrt(2)},{.25*sqrt(2)},{.25*sqrt(2)});
\draw(0,{.5*sqrt(2)},0)--({-.25*sqrt(2)},{.25*sqrt(2)},{-.25*sqrt(2)});

\draw(0,{-.5*sqrt(2)},0)--({.25*sqrt(2)},{-.25*sqrt(2)},{.25*sqrt(2)});
\draw(0,{-.5*sqrt(2)},0)--({.25*sqrt(2)},{-.25*sqrt(2)},{-.25*sqrt(2)});
\draw[dashed](0,{-.5*sqrt(2)},0)--({-.25*sqrt(2)},{-.25*sqrt(2)},{.25*sqrt(2)});
\draw(0,{-.5*sqrt(2)},0)--({-.25*sqrt(2)},{-.25*sqrt(2)},{-.25*sqrt(2)});

\draw[dashed](0,0,{.5*sqrt(2)})--({.25*sqrt(2)},{.25*sqrt(2)},{.25*sqrt(2)});
\draw[dashed](0,0,{.5*sqrt(2)})--({-.25*sqrt(2)},{.25*sqrt(2)},{.25*sqrt(2)});
\draw[dashed](0,0,{.5*sqrt(2)})--({.25*sqrt(2)},{-.25*sqrt(2)},{.25*sqrt(2)});
\draw[dashed](0,0,{.5*sqrt(2)})--({-.25*sqrt(2)},{-.25*sqrt(2)},{.25*sqrt(2)});

\draw(0,0,{-.5*sqrt(2)})--({.25*sqrt(2)},{.25*sqrt(2)},{-.25*sqrt(2)});
\draw(0,0,{-.5*sqrt(2)})--({-.25*sqrt(2)},{.25*sqrt(2)},{-.25*sqrt(2)});
\draw(0,0,{-.5*sqrt(2)})--({.25*sqrt(2)},{-.25*sqrt(2)},{-.25*sqrt(2)});
\draw(0,0,{-.5*sqrt(2)})--({-.25*sqrt(2)},{-.25*sqrt(2)},{-.25*sqrt(2)});

\end{scope}

\begin{scope}[shift={(8,0)},scale=4]
\draw[dashed](0,{1/3*sqrt(3)},{-1/6*sqrt(6)})--(0,0,{-1/4*sqrt(6)})--(1/2,{-1/6*sqrt(3)},{-1/6*sqrt(6)})--(1/2,{1/6*sqrt(3)},{-1/12*sqrt(6)});

\draw(-1/2,{1/6*sqrt(3)},{1/12*sqrt(6)})--(-1/2,{-1/6*sqrt(3)},{1/6*sqrt(6)})--(0,0,{1/4*sqrt(6)});

\draw(-1/2,{-1/6*sqrt(3)},{1/6*sqrt(6)})--(0,{-1/3*sqrt(3)},{1/12*sqrt(6)});

\draw[dashed](-1/2,{1/6*sqrt(3)},{-1/12*sqrt(6)})--(-1/2,{-1/6*sqrt(3)},{-1/6*sqrt(6)})--(0,0,{-1/4*sqrt(6)});

\draw(0,{1/3*sqrt(3)},{1/6*sqrt(6)})--(0,0,{1/4*sqrt(6)})--(1/2,{-1/6*sqrt(3)},{1/6*sqrt(6)})--(1/2,{1/6*sqrt(3)},{1/12*sqrt(6)});

\draw(0,{1/3*sqrt(3)},{1/6*sqrt(6)})--(0,{1/3*sqrt(3)},{-1/6*sqrt(6)})--(1/2,{1/6*sqrt(3)},{-1/12*sqrt(6)})--(1/2,{1/6*sqrt(3)},{1/12*sqrt(6)})--(0,{1/3*sqrt(3)},{1/6*sqrt(6)});

\draw[dashed](-.5,{-1/6*sqrt(3)},{-1/6*sqrt(6)})--(-.5,{-1/6*sqrt(3)},{1/6*sqrt(6)})--(-.5,{1/6*sqrt(3)},{1/12*sqrt(6)})--(-.5,{1/6*sqrt(3)},{-1/12*sqrt(6)})--(-.5,{-1/6*sqrt(3)},{-1/6*sqrt(6)});

\draw(0,{1/3*sqrt(3)},{-1/6*sqrt(6)})--(0,{1/3*sqrt(3)},{1/6*sqrt(6)})--(-.5,{1/6*sqrt(3)},{1/12*sqrt(6)})--(-.5,{1/6*sqrt(3)},{-1/12*sqrt(6)})--(0,{1/3*sqrt(3)},{-1/6*sqrt(6)});

\draw[dashed](0,{-1/3*sqrt(3)},{-1/12*sqrt(6)})--(0,{-1/3*sqrt(3)},{1/12*sqrt(6)})--(.5,{-1/6*sqrt(3)},{1/6*sqrt(6)})--(.5,{-1/6*sqrt(3)},{-1/6*sqrt(6)})--(0,{-1/3*sqrt(3)},{-1/12*sqrt(6)});

\draw(0,{-1/3*sqrt(3)},{1/12*sqrt(6)})--(.5,{-1/6*sqrt(3)},{1/6*sqrt(6)});

\draw(.5,{-1/6*sqrt(3)},{-1/6*sqrt(6)})--(.5,{-1/6*sqrt(3)},{1/6*sqrt(6)})--(.5,{1/6*sqrt(3)},{1/12*sqrt(6)})--(.5,{1/6*sqrt(3)},{-1/12*sqrt(6)})--(.5,{-1/6*sqrt(3)},{-1/6*sqrt(6)});

\draw[dashed](-1/2,{-1/6*sqrt(3)},{1/6*sqrt(6)})--(-1/2,{-1/6*sqrt(3)},{-1/6*sqrt(6)})--(0,{-1/3*sqrt(3)},{-1/12*sqrt(6)})--(0,{-1/3*sqrt(3)},{1/12*sqrt(6)})--(-1/2,{-1/6*sqrt(3)},{1/6*sqrt(6)});

\end{scope}

\end{tikzpicture}
\caption{Left: the Voronoi cell $V_{\mathrm{FCC}}$ of the ${\mathrm {FCC}}$ lattice. Right: the Voronoi cell $V_{\mathrm{HCP}}$ of the ${\mathrm {HCP}}$ lattice.}
\label{fig:VFCC}
\end{figure}

\begin{lemma}(Voronoi cell in the $\mathrm{FCC}$-lattice) \label{lemma:shapeFCC} Let  us take $x\in \mathcal{L}_{\mathrm{FCC}}$. Then 
\begin{align}\label{eq:VoronoiFCClemma}
{\mathcal V}_{\mathcal L_{\mathrm{FCC}}}(x) = x + V_{\mathrm{FCC}}\,, \quad \text{ where } V_{\mathrm{FCC}} := \left\{y \in \mathbb{R}^3\colon \max_{b \in \mathcal{N}_{\mathrm{FCC}}}\langle b,y\rangle \leq \frac{1}{2} \right\}\,.
\end{align}
Given $b_0 \in \mathcal{N}_{\mathrm{FCC}}$ the face  
\begin{align}\label{eq:faceFCClemma}
S_{b_0}:=\left\{y \in \mathbb{R}^3\colon \max_{b \in \mathcal{N}_{\mathrm{FCC}}}\langle b,y\rangle = \langle b_0,y\rangle = \frac{1}{2} \right\}
\end{align} 
is a rhombus with $\mathcal{H}^2(S_{b_0})=\frac{1}{4}\sqrt{2}$. Moreover, for each $b_0 \in \mathcal{N}_{\mathrm{FCC}}$ the face $S_{b_0}$ of ${\mathcal V}_{\mathrm{FCC}}(0)$ is shared with the Voronoi cell ${\mathcal V}_{\mathrm{FCC}}(b_0)$. Lastly, we have $|{\mathcal V}_{\mathrm{FCC}}(x)|= \frac{1}{2}\sqrt{2}$ for all $x\in \mathcal{L}_{\mathrm{FCC}}$.
\end{lemma}

\begin{lemma}(Voronoi cell in the $\mathrm{HCP}$-lattice) \label{lemma:shapeHCP} Let  us take $x\in \mathcal{L}_{\mathrm{HCP}}$. Then 

\begin{align}\label{eq:VoronoiHCPlemma}
{\mathcal V}_{\mathcal L_{\mathrm{HCP}}}(x) = \begin{cases} x+V_{\mathrm{HCP}} &\text{if } x \in \mathrm{span}_{\mathbb{Z}}\{e_1,e_2,e_3\}\,,\\
x-V_{\mathrm{HCP}} &\text{if } x \in (v_1+\mathrm{span}_{\mathbb{Z}}\{e_1,e_2,e_3\})\,, 
\end{cases}
\end{align}
where 
\begin{align*}
V_{\mathrm{HCP}} := \left\{y \in \mathbb{R}^3\colon \max_{b \in \mathcal{N}_{\mathrm{HCP}}}\langle b,y\rangle \leq \frac{1}{2} \right\}\,.
\end{align*}
For $b_0 \in \mathcal{N}_{\mathrm{HCP}}$ we set 
\begin{align}\label{def:faceVHCP}
S_{b_0}:=\left\{y \in \mathbb{R}^3\colon \max_{b \in \mathcal{N}_{\mathrm{HCP}}}\langle b,y\rangle = \langle b_0,y\rangle = \frac{1}{2} \right\}\,.
\end{align}
If $b_0 \in \{\pm e_1,\pm e_2,\pm (e_1-e_2)\}$ the face  $S_{b_0}$  is a trapezoid of area $\frac{1}{4}\sqrt{2}$. If $b_0 \in \{v_1,v_1-e_1,v_1-e_2,v_1-e_3,v_1-e_1-e_3,v_1-e_2-e_3\} $ the face $S_{b_0}$  is a rhombus of area $\frac{1}{8}\sqrt{6}$. Moreover, for each $b_0 \in \mathcal{N}_{\mathrm{HCP}}$ the face $S_{b_0}$ is shared with the Voronoi cell $\mathcal{V}_{\mathcal L_{\mathrm{HCP}}}(b_0)$. Lastly, we have $|\mathcal{V}_{\mathcal L_\mathrm{HCP}}(x)|= \frac{1}{2}\sqrt{2}$ for all $x\in \mathcal{L}_{\mathrm{HCP}}$.
\end{lemma}

\begin{proof}[Proof of Lemma \ref{lemma:shapeFCC}] We split the proof of the lemma into four steps. First, we prove \eqref{eq:VoronoiFCClemma}. In the second step, we show that each face is a rhombus and calculate its area. Lastly, we show that each neighboring Voronoi cell $V_{\mathrm{FCC}}(b)$, $b \in \mathcal{N}_{\mathrm{FCC}}$ shares one face with the Voronoi cell $V_{\mathrm{FCC}}(0)$.\\
\noindent\begin{step}{1}(Proof of \eqref{eq:VoronoiFCClemma}) To check \eqref{eq:VoronoiFCClemma}, since $\mathcal{L}_{\mathrm{FCC}}$ is a Bravais-lattice (see \eqref{def:FCC}), it suffices to consider the case $x=0$. Let $\mathcal V_{{\mathcal L_\mathrm {FCC}}}(0)$ denote the Voronoi cell of $\mathcal L_{\mathrm{FCC}}$ at $x=0$ defined according to \eqref{def:Voronoi}.\\
\noindent\begin{step}{1.1}($\mathcal V_{{\mathcal L_\mathrm {FCC}}}(0)\subset V_{\mathrm{FCC}}$) Let $y \in \mathcal V_{{\mathcal L_\mathrm {FCC}}}(0)$. By the very definition of Voronoi cell we have that for all $b\in{\mathcal N_{\mathrm{FCC}}}$ it holds
$
|y| \leq |y-b|\,.
$
Noting that $|b|=1$ for all $b \in \mathcal{N}_{\mathrm{FCC}} \subset \mathcal{L}_{\mathrm{FCC}}$, we have 
\begin{align*}
|y| \leq |y-b| \iff|y|^2 \leq |y-b|^2 = |y|^2 -2\langle b,y\rangle +|b|^2  \iff \langle b,y\rangle \leq \frac{1}{2}\,,
\end{align*}
that is the inclusion $\mathcal V_{\mathcal L_{\mathrm{FCC}}}(0) \subset V_{\mathrm{FCC}}$.
\end{step}\\
\noindent\begin{step}{1.2}($V_{\mathrm{FCC}} \subset \mathcal V_{\mathcal L_{\mathrm{FCC}}}(0)$) We show that for $y \in V_{\mathrm{FCC}}$ we have $|y|\leq |y-z|$ for all $z \in \mathcal{L}_{\mathrm{FCC}}$. This is equivalent to
\begin{align}\label{ineq:inclusionFCC}
y \in V_{\mathrm{FCC}}\implies\langle y,z\rangle \leq \tfrac{1}{2}|z|^2 \quad \text{ for  all } z \in \mathcal{L}_{\mathrm{FCC}}\,.
\end{align}
We first observe that if $z\in\mathcal N_{\mathrm{FCC}}$, \eqref{ineq:inclusionFCC} is trivial since $|z|=1$. Next, we prove \eqref{ineq:inclusionFCC} for all $z \in \mathcal{L}_{\mathrm{FCC}}\setminus \mathcal{N}_{\mathrm{FCC}}$. We distinguish two cases:
\begin{itemize}
\item[(a)] $z=\lambda_1 b_j +\lambda_2 b_k$ $\lambda_1,\lambda_2 \in \mathbb{Z}$, $j,k \in \{1,2,3\}, j \neq k\,$;
\item[(b)] $z=\lambda_1 b_1 +\lambda_2 b_2+\lambda_3 b_3$ $\lambda_1,\lambda_2,\lambda_3 \in \mathbb{Z}\,$.
\end{itemize}
\noindent \emph{Proof in case {\rm(a)}.} We only show the statement for  $z =\lambda_1b_1+\lambda_2b_2$ for $\lambda_1,\lambda_2 \in \mathbb{Z}$, the cases with any other combination of two vectors being analogous. If $\lambda_1 \lambda_2 \geq 0$, since $\langle b_1,b_2\rangle\geq 0$, we have
\begin{align*}
\langle y, z\rangle &= \langle y,\lambda_1b_1+\lambda_2b_2\rangle \leq \frac{1}{2}|\lambda_1 b_1|^2 +  \frac{1}{2}|\lambda_2 b_2|^2 \\&= \frac{1}{2}|\lambda_1 b_1 +\lambda_2 b_2|^2- \lambda_1\lambda_2\langle b_1,b_2\rangle \leq  \frac{1}{2}|\lambda_1 b_1 +\lambda_2 b_2|^2=\frac{1}{2}|z|^2\,.
\end{align*}
On the other hand, if $\lambda_1\lambda_2 \leq 0$ and without loss of generality $|\lambda_1|\leq |\lambda_2|$, noting that $b_1-b_2 \in \mathcal{N}_{\mathrm{FCC}}$,  we have 
\begin{align*}
\langle y, z\rangle&= \langle y,\lambda_1b_1+\lambda_2b_2\rangle = \langle y, (\lambda_2+\lambda_1) b_2 + \lambda_1(b_1-b_2)\rangle \leq \frac{1}{2}| (\lambda_2+\lambda_1) b_2|^2 + \frac{1}{2}|\lambda_1(b_1-b_2)|^2\\&= \frac{1}{2}|\lambda_1b_1+\lambda_2b_2|^2 -\lambda_1(\lambda_2+\lambda_1)\langle(b_1-b_2),b_2\rangle \leq \frac{1}{2}|\lambda_1b_1+\lambda_2b_2|^2\,.
\end{align*}
Here, the last inequality follows, since $|b_1|=|b_2|$ and therefore $\lambda_1(\lambda_2+\lambda_1)\langle(b_1-b_2),b_2\rangle\geq 0$. This concludes case {\rm (a)}.\\
\noindent\emph{Proof in case {\rm (b)}.} We now show that \eqref{eq:VoronoiFCClemma} holds true in the case of $b= \lambda_1b_1+\lambda_2b_2+\lambda_3b_3$ with $\lambda_i \in \mathbb{Z}$. We restrict to the case $\lambda_1 \geq 0, \lambda_2 \geq 0$ and $\lambda_3 \leq 0$, since if all $\lambda_i$ are of the same sign, \eqref{eq:VoronoiFCClemma} can be deduced from the fact that it holds true for $b\in \mathcal{N}_{\mathrm{FCC}}$ and the fact that $\langle b_j,b_k\rangle \geq 0$. Without loss of generality, we assume $|\lambda_2|\leq |\lambda_3|$. Hence, observing that $b_2-b_3\in\mathcal N_{\mathrm{FCC}}$, we have
\begin{align*}
\langle y, z\rangle&= \langle y,\lambda_1b_1+\lambda_2b_2+\lambda_3 b_3\rangle  =  \langle y,\lambda_1b_1+(\lambda_3+\lambda_2)b_3+\lambda_2(b_2-b_3)\rangle \\&\leq \frac{1}{2}|\lambda_1b_1+(\lambda_3+\lambda_2)b_3|^2+ \frac{1}{2}|\lambda_2(b_2-b_3)|^2 \\&= \frac{1}{2}|\lambda_1b_1+\lambda_2b_2+\lambda_3 b_3|^2 - \langle(\lambda_1b_1+(\lambda_3+\lambda_2)b_3),\lambda_2(b_2-b_3)\rangle \\&=\frac{1}{2}|\lambda_1b_1+\lambda_2b_2+\lambda_3 b_3|^2 - (\lambda_3+\lambda_2)\lambda_2\langle b_2-b_3,b_3\rangle \leq \frac{1}{2}|\lambda_1b_1+\lambda_2b_2+\lambda_3 b_3|^2\,.
\end{align*}
Here, the last inequality follows from $|b_2|=|b_3|$ and $\lambda_3+\lambda_2\leq 0$ whereas the equality in the last line is due to $\langle b_1,b_2\rangle= \langle b_1,b_3\rangle=\langle b_2,b_3\rangle$. This concludes case {\rm (b)} and with that Step 1.2.
\end{step}
\end{step}\\
\begin{step}{2}(The faces of the Voronoi cell) To show that each face of the Voronoi cell $V_\mathrm{FCC}$ is a rhombus with area $\frac{1}{4}\sqrt{2}$ we first exploit its symmetries. Let $i \in \{1,2,3\}$ and let $T_i\colon \mathbb{R}^3 \to \mathbb{R}^3$ be the linear mapping that flips the $i$-th entry, i.e.
\begin{align*}
(T_ix)_j =\begin{cases} - x_i &\text{if } i=j\,,\\
x_j &\text{otherwise.}
\end{cases}
\end{align*}
We observe that
\begin{align*}
T_i \mathcal{N}_{\mathrm{FCC}} = \{T_ib\colon b \in \mathcal{N}_{\mathrm{FCC}}\} = \mathcal{N}_{\mathrm{FCC}}\,, \text{ for all } i\in \{1,2,3\}.
\end{align*}
Moreover, given a permutation $\pi \in S_3$ we have that
\begin{align*}
\pi \mathcal{N}_{\mathrm{FCC}} = \{\pi b\colon b \in \mathcal{N}_{\mathrm{FCC}}\} = \mathcal{N}_{\mathrm{FCC}}\,.
\end{align*}
It therefore suffices to restrict only to the case in which the vector $b_0$ agrees with the vector $b_1 \in \mathcal{N}_{\mathrm{FCC}}$. We claim that this face has corners given by
\begin{align}\label{eq:cornersFCC}
c_1 = \left(\frac{1}{2}\sqrt{2},0,0\right)\,, c_2 = \left(0,\frac{1}{2}\sqrt{2},0\right)\,, c_3 = \frac{1}{4}\left(\sqrt{2},\sqrt{2},\sqrt{2}\right)\,,  c_4 = \frac{1}{4}\left(\sqrt{2},\sqrt{2},-\sqrt{2}\right)\,.
\end{align}
Note that, if this were true then it is easy to see that $S_{b_0}$ is a rhombus and $\mathcal{H}^2(S_{b_0})=\frac{1}{4}\sqrt{2}$. It remains to prove \eqref{eq:cornersFCC}. Let us denote by $y$ a corner of $S_{b_0}$. We can assume that $y_1,y_2 \geq 0$. Were this not the case, then there could be $b' \in \mathcal{N}_{\mathrm{FCC}}$ such that  $\langle b',y\rangle > \langle b,y\rangle$, thus contradicting the definition of $S_{b_0}$ in \eqref{eq:faceFCClemma}. If $y_1=0$ (or $y_2=0$), then $y_2=\frac{1}{2}\sqrt{2}$ (resp.   $y_1=\frac{1}{2}\sqrt{2}$) and since, $ \langle b',y\rangle\leq \frac{1}{2}$ for all $b'\in \mathcal{N}_{\mathrm{FCC}}$ we have $y_3=0$. Hence, we find the two corners with coordinates $(\frac{1}{2}\sqrt{2},0,0)$ and $(0,\frac{1}{2}\sqrt{2},0)$. Now, if $y_1 >0$ and $y_2 >0$, then assuming that $y_3 \geq 0$ we have that the corner is equal to $\langle b_1,y\rangle=\langle b_2,y\rangle=\langle b_3,y\rangle=\frac{1}{2}\sqrt{2}$. Thus, necessarily $y_1=y_2=y_3 =\frac{1}{4}\sqrt{2}$. If instead $y_3 <0$, then the corner is equal to $\langle b_1,y\rangle=\langle b_2,y\rangle=\langle b_1- b_3,y\rangle=\frac{1}{2}\sqrt{2}$ which implies $y_1=y_2=-y_3 =\frac{1}{4}\sqrt{2}$. Hence \eqref{eq:cornersFCC} holds true and this concludes Step~2.
\end{step}\\
\noindent 
\begin{step}{3}(Neighbors share faces)  We want to show that  for each $b_0 \in \mathcal{N}_{\mathrm{FCC}}$ we have that the face $S_{b_0}$ of $V_{\mathrm{FCC}}(0)$ is shared with the Voronoi cell $V_{\mathrm{FCC}}(b_0)$. By the symmetries shown in Step~2 it suffices to prove this statement only for $b_0=b_1$. Using \eqref{eq:cornersFCC} we see that the corners of the face $S_{b_0}$ of the Voronoi cell $V_{\mathrm{FCC}}(0)$ coincide with the corners of the face $S_{-b_0}+b_0$ of the Voronoi cell  $V_{\mathrm{FCC}}(b_0)$.
\end{step}\\
\noindent  
\begin{step}{4}(Volume of the Voronoi cell) In order to calculate the volume of the Voronoi cell we note that $\mathcal{L}_{\mathrm{FCC}}$ is a Bravais-lattice with spanning vectors $b_1,b_2,b_3$. Since, the Voronoi cells of all the points are the same, it suffices to calculate the fraction of points per unit volume. This, then gives also the volume per point. Since, the Voronoi cells are space filling the volume per point is equal to the volume of each Voronoi cell. Due to \eqref{def:periodicitycellFCC} we have that
\begin{align*}
|T_{\mathrm{FCC}}| = \frac{1}{2}\sqrt{2}\,.
\end{align*}
Furthermore, we have that
\begin{align*}
\bigcup_{x \in \mathcal{L}_{\mathrm{FCC}}} (x+ T_{{\mathrm{FCC}}}) =\mathbb{R}^3\,, \text{ and } \mathcal{L}_{\mathrm{FCC}} \cap T_{\mathrm{FCC}} = \{0\}\,.
\end{align*}
Hence, each points of the lattice occupies a volume $|T_{\mathrm{FCC}}|=\frac{1}{2}\sqrt{2}$ and the volume of the Voronoi cell must be the same. This concludes Step 3 and thus the proof of the lemma.
\end{step}
\end{proof}

\begin{proof}[Proof of Lemma \ref{lemma:shapeHCP}] We split the proof of the lemma into four steps. First,  we prove \eqref{eq:VoronoiHCPlemma}. In the second step, we show that $6$ of the faces are rhombi, the $6$ other faces are trapezoids, and we calculate the area of each face. Lastly, given  $x \in \mathcal{L}_{\mathrm{HCP}}$, we show that each neighboring Voronoi cell $\mathcal V_{\mathcal L_{\mathrm{HCP}}}(y), y \in \mathcal{N}_{\mathrm{HCP}}(x)$ shares a face with the Voronoi cell $V_{\mathrm{HCP}}(x)$.\\
\begin{step}{1}(Shape of the Voronoi cell) The purpose of this step is to prove \eqref{eq:VoronoiHCPlemma}. Here, we only show this equality in the case that $x=0$, the case $x\neq 0$ being treated in a similar fashion.
\begin{step}{1.1}($\mathcal V_{\mathcal L_{\mathrm{HCP}}}(0) \subset V_{\mathrm{HCP}}$) Given $y \in \mathcal V_{\mathcal L_{\mathrm{HCP}}}(0)$ we have that 
$
|y| \leq |y-b|\,.
$
Now, noting that $|b|=1$ for all $b \in \mathcal{N}_{\mathrm{HCP}} \subset \mathcal{L}_{\mathrm{HCP}}$, we have 
\begin{align*}
|y|^2 \leq |y-b|^2 = |y|^2 -2\langle y,b\rangle +|b|^2  \iff \langle b,y\rangle \leq \frac{1}{2}\,.
\end{align*}
This concludes Step 1.1.
\end{step}\\
\begin{step}{1.2}($V_{\mathrm{HCP}} \subset \mathcal V_{\mathcal L_{\mathrm{HCP}}}(0)$) We show that for $y \in V_{\mathrm{HCP}}$ we have $|y|\leq |y-z|$, for all $z \in \mathcal{L}_{\mathrm{HCP}}$. This is equivalent to
\begin{align}\label{ineq:inclusionHCP}
y \in V_{\mathrm{HCP}}\implies\langle y,z\rangle \leq \tfrac{1}{2}|z|^2 \quad \text{ for  all } z \in \mathcal{L}_{\mathrm{HCP}}\,.
\end{align}
Since, $|b|=1$ for all $b \in \mathcal{N}_{\mathrm{HCP}}$ \eqref{ineq:inclusionHCP} is true for all $b \in\mathcal{N}_{\mathrm{HCP}}$. Next, we prove \eqref{ineq:inclusionHCP} for all $z \in \mathcal{L}_{\mathrm{HCP}}\setminus \mathcal{N}_{\mathrm{HCP}}$. We distinguish several cases:
\begin{itemize}
\item[(a)] $z=\lambda_1 e_1 +\lambda_2 e_2, \lambda_1,\lambda_2 \in \mathbb{Z}\,;$
\item[(b)] $z=\lambda_1 e_1 +\lambda_2 e_2+\lambda_3 e_3, \lambda_1,\lambda_2,\lambda_3 \in \mathbb{Z}\,;$
\item[(c)] $z=v_1+\lambda_1 e_1 +\lambda_2 e_2, \lambda_1,\lambda_2 \in \mathbb{Z}\,;$
\item[(d)] $z=v_1+\lambda_1 e_1 +\lambda_2 e_2+\lambda_3 e_3, \lambda_1,\lambda_2,\lambda_3 \in \mathbb{Z}\,;$
\end{itemize}
\noindent \emph{Proof in case {\rm (a)}.} If $\lambda_1,\lambda_2 \geq 0$, using that $\langle e_1,e_2\rangle \geq 0$, we have
\begin{align*}
\langle y, z\rangle &= \langle y, \lambda_1 e_1 +\lambda_2 e_2 \rangle \leq \frac{1}{2}|\lambda_1 e_1|^2 + \frac{1}{2}|\lambda_2 e_2|^2\\& = \frac{1}{2}|\lambda_1 e_1 +\lambda_2e_2|^2 - \lambda_1\lambda_2\langle e_1,e_2\rangle \leq \frac{1}{2}|\lambda_1 e_1 +\lambda_2e_2|^2 =\frac{1}{2}|z|^2\,.
\end{align*}
On the other hand, if $\lambda_1\lambda_2\leq 0$ and without loss of generality $\lambda_1 \geq |\lambda_2|\geq 0$, noting that $e_2-e_1 \in \mathcal{N}_{\mathrm{HCP}}$, we have
\begin{align*}
\langle y, z\rangle &= \langle y, \lambda_1 e_1 +\lambda_2 e_2\rangle = \langle y, \lambda_2(e_2-e_1) + (\lambda_1+\lambda_2)e_1\rangle \leq \frac{1}{2}|\lambda_2(e_2-e_1)|^2+ \frac{1}{2}|(\lambda_2+\lambda_1)e_1|^2\\& = \frac{1}{2}|\lambda_1e_1+\lambda_2e_2|^2 - \lambda_2(\lambda_1+\lambda_2)\langle e_2-e_1,e_1\rangle \leq \frac{1}{2}|\lambda_1e_1+\lambda_2e_2|^2= \frac{1}{2}|z|^2\,.
\end{align*}
Here, the last inequality follows, since $\lambda_2\leq 0\leq \lambda_1+\lambda_2$ and $\langle e_2-e_1,e_1\rangle\leq 0$. This concludes case {\rm (a)}.\\
\noindent\emph{Proof in case {\rm (b)}.} We first show that $\langle y, e_3\rangle \leq \frac{1}{2}|e_3|^2$. Using that $v_1,v_1-e_1,v_1-e_2 \in \mathcal{N}_{\mathrm{HCP}}$, that $3v_1-e_1-e_2=\frac32 e_3$ we have
\begin{align*}
\langle y,e_3\rangle &= \frac{2}{3}\langle y,v_1 + v_1-e_1 +v_1-e_2\rangle \\&\leq \frac{1}{3}\left(\left|\frac{1}{3}(e_1+e_2)\right|^2+ \left|\frac{1}{3}(e_2-2e_1)\right|^2+ \left|\frac{1}{3}(e_1-2e_2)\right|^2\right) + \left|\frac{1}{2}e_3\right|^2\leq \frac{1}{2}|e_3|^2\,.
\end{align*}
Here, the last inequality follows by calculating the norms of $e_1+e_2, e_1-2e_2,e_2-2e_1$ and $e_3$ by using \eqref{def:vectors}. Note that now, the case of $z=\lambda_1e_1+\lambda_2e_2+\lambda_3e_3$ follows from case (a) using that $\langle e_3, e_1\rangle = \langle e_3, e_2\rangle = 0$.\\
\noindent \emph{Proof of case {\rm (c)}.} Let $z=v_1+\lambda_1e_1+\lambda_2 e_2$. If $\lambda_1,\lambda_2 \geq 0$ we have
\begin{align*}
\langle y,z\rangle &= \langle y, v_1 +\lambda_1 e_1 +\lambda_2 e_2,\rangle \leq 
\frac12|v_1|^2+\frac12|\lambda_1e_1|^2+\frac12|\lambda_2e_2|^2
\leq
\frac{1}{2}|v_1|^2 + \frac{1}{2}|\lambda_1 e_1+\lambda_2 e_2|^2 \\&= \frac{1}{2}|v_1+\lambda_1 e_1+\lambda_2e_2|^2 - \langle v_1, \lambda_1e_1+\lambda_2e_2\rangle\leq\frac{1}{2}|v_1+\lambda_1 e_1+\lambda_2e_2|^2 \,.
\end{align*}
The second inequality uses that $\langle e_1,e_2\rangle\geq0$ and the last inequality uses that $\langle v_1,e_1\rangle,\langle v_1,e_2\rangle \geq 0$. Now assume that $\lambda_1 \geq 0, \lambda_2< 0$. Then, since $\langle(v_1-e_2), e_1\rangle=0$ and $\langle v_1-e_2,e_2\rangle\leq 0$, again exploiting that $v_1-e_2\in\mathcal {N}_{\mathrm {HCP}}$ it holds that
\begin{align*}
\langle y,z\rangle &= \langle y, (v_1-e_2) + \lambda_1 e_1 +(\lambda_2+1) e_2  \rangle \leq \frac{1}{2}|v_1-e_2|^2 + \frac{1}{2}|\lambda_1 e_1 +(\lambda_2+1)e_2 |^2\\&= \frac{1}{2}|v_1+\lambda_1 e_1+\lambda_2 e_2|^2 - \langle v_1-e_2,\lambda_1e_1+(\lambda_2+1)e_2\rangle\leq \frac{1}{2}|v_1+\lambda_1 e_1+\lambda_2 e_2|^2= \frac{1}{2}|z|^2\,.
\end{align*}
The case where $\lambda_1 <0,\lambda_2\geq 0$ (resp.~$\lambda_1, \lambda_2 <0$) is being treated in a similar fashion by replacing $v_1-e_2$ with $v_1-e_1$ (resp.~$v_1-e_1-e_2$).\\
\noindent \emph{Proof of case {\rm (d)}.} Here, we only treat the case of $z= v_1 +\lambda_1 e_1 +\lambda_2 e_2+\lambda_3 e_3$, $\lambda_3 \geq 0$.  Since $\langle v_1 +\lambda_1 e_1 +\lambda_2 e_2, e_3\rangle\geq 0$, we have
\begin{align*}
\langle y,z\rangle &= \langle y,v_1 +\lambda_1 e_1 +\lambda_2 e_2+\lambda_3 e_3\rangle \leq \frac{1}{2}|v_1 +\lambda_1 e_1 +\lambda_2 e_2|^2 +\frac{1}{2}|\lambda_3 e_3|^2\\&= \frac{1}{2}|v_1 +\lambda_1 e_1 +\lambda_2 e_2+\lambda_3 e_3|^2 - \langle v_1 +\lambda_1 e_1 +\lambda_2 e_2, e_3\rangle\leq \frac{1}{2}|v_1 +\lambda_1 e_1 +\lambda_2 e_2+\lambda_3 e_3|^2=\frac{1}{2}|z|^2\,.
\end{align*}
The case of $\lambda_3 <0$ follows by replacing $v_1$ with $v_1-e_3$ in the last two cases {\rm (c)} and {\rm (d)}. 
This concludes Step 1.2 and, together with Step 1.1, shows \eqref{eq:VoronoiHCPlemma}.
\end{step}
\end{step}\\
\noindent
\begin{step}{2}(The faces of the Voronoi cell) In order to calculate the faces of $V_{\mathrm{HCP}}$ we use \eqref{eq:VoronoiHCPlemma} and exploit its symmetries. We note that if $R \in SO(3)$ is any rotation of integer multiples of $2\pi/3$ around the $x_3$-axis we have that 
\begin{align}\label{eq:rotationalsymmetryHCP}
R\mathcal{N}_{\mathrm{HCP}}=\{Rb\colon b \in \mathcal{N}_{\mathrm{HCP}}\} = \mathcal{N}_{\mathrm{HCP}}\,.
\end{align}
Moreover, if $T_3 \colon \mathbb{R}^3\to \mathbb{R}^3$ is the reflection with respect to the $(x_1,x_2)$-plane, i.e. 
\begin{align}\label{def:T3}
(T_3x)_j := \begin{cases} x_j &j =1,2\,,\\
-x_3 &j=3\,,
\end{cases} 
\end{align}
we have that
\begin{align}\label{eq:reflectionsymmetryHCP}
T_3\mathcal{N}_{\mathrm{HCP}} = \{T_3 b \colon b \in \mathcal{N}_{\mathrm{HCP}}\} = \mathcal{N}_{\mathrm{HCP}}\,.
\end{align}
Exploiting \eqref{eq:rotationalsymmetryHCP} and \eqref{eq:reflectionsymmetryHCP}, it suffices to find the corners of $S_{b_0}$ in \eqref{def:faceVHCP}  for
\begin{align*}
{\rm (a)}\quad b_0=e_1\,,\quad {\rm (b)}\quad b_0 = -e_1\,,\quad {\rm (c)}\quad b_0=v_1\,.
\end{align*}
\emph{Corners in case {\rm (a)}.}  We claim that in the case of $b_0=e_1$ that the corners of $S_{b_0}$ are given by the points
\begin{align}\label{eq:cornersHCPa}
\begin{split}
&c_1 = \left(\frac{1}{2},\frac{1}{6}\sqrt{3},\frac{1}{12}\sqrt{6}\right)\,,  c_2 = \left(\frac{1}{2},\frac{1}{6}\sqrt{3},-\frac{1}{12}\sqrt{6}\right)\,,\\&c_3 = \left(\frac{1}{2},-\frac{1}{6}\sqrt{3},\frac{1}{6}\sqrt{6}\right)\,, c_4 = \left(\frac{1}{2},-\frac{1}{6}\sqrt{3},-\frac{1}{6}\sqrt{6}\right)\,.
\end{split}
\end{align}
In particular, the face $S_{b_0}$ is a trapezoid with two bases of length $\frac{1}{6}\sqrt{6}$, $\frac{1}{3}\sqrt{6}$ and height $\frac{1}{3}\sqrt{3}$. Hence, $\mathcal{H}^2(S_{b_0})=\frac{1}{4}\sqrt{2}$. It remains to prove \eqref{eq:cornersHCPa}. Let $y \in S_{b_0}$ be a corner. Due to \eqref{eq:reflectionsymmetryHCP}, we can assume that $y_3\geq 0$, since the other corners are just found by applying the mapping $T_3$ (see \eqref{def:T3}) to the corners with positive coordinates. By the definition of $S_{b_0}$ we have that $\langle y,e_1\rangle \geq \langle y,e_1-e_2\rangle$ which is equivalent to $\langle y,e_2\rangle \geq 0$.  Now, if $\langle y, e_2\rangle > 0$, then $y$ is given by $\langle y, e_1\rangle = \langle y, e_2\rangle = \langle y, v_1\rangle=\frac{1}{2}$. This linear system has a unique solution given by $c_1=\left(\frac{1}{2},\frac{1}{6}\sqrt{3},\frac{1}{12}\sqrt{6}\right)$. On the other hand, if $\langle y, e_2 \rangle = 0$, then $y$ is given by $\langle y, e_2\rangle =0, \langle y, e_1\rangle = \langle y, v_1\rangle=\frac{1}{2}$. The unique solution of this linear system is given by $c_3 = \left(\frac{1}{2},-\frac{1}{6}\sqrt{3},\frac{1}{6}\sqrt{6}\right)$. This shows \eqref{eq:cornersHCPa} and concludes case {\rm (a)}.\\
\noindent \emph{Corners in case {\rm (b)}.}
We claim that in the case of $b_0=-e_1$ that the corners of $S_{b_0}$ are given by the points
\begin{align}\label{eq:cornersHCPb}
\begin{split}
&c_1 = \left(-\frac{1}{2},\frac{1}{6}\sqrt{3},\frac{1}{12}\sqrt{6}\right)\,,  c_2 = \left(-\frac{1}{2},\frac{1}{6}\sqrt{3},-\frac{1}{12}\sqrt{6}\right)\,,\\&c_3 = \left(-\frac{1}{2},-\frac{1}{6}\sqrt{3},\frac{1}{6}\sqrt{6}\right)\,, c_4 = \left(-\frac{1}{2},-\frac{1}{6}\sqrt{3},-\frac{1}{6}\sqrt{6}\right)\,.
\end{split}
\end{align}
In particular, the face $S_{b_0}$ is a trapezoid with two bases of length $\frac{1}{6}\sqrt{6}$, $\frac{1}{3}\sqrt{6}$ and height $\frac{1}{3}\sqrt{3}$. Hence, $\mathcal{H}^2(S_{b_0})=\frac{1}{4}\sqrt{2}$. It remains to prove \eqref{eq:cornersHCPb}.  Let $y \in S_{b_0}$ be a corner. Due to \eqref{eq:reflectionsymmetryHCP}, as in case $\rm{(a)}$, we can assume that $y_3\geq 0$. By the definition of $S_{b_0}$ we have that $\langle y,-e_1\rangle \geq \langle y,e_2-e_1\rangle$ which is equivalent to $\langle y,e_2\rangle \leq 0$. Now, if $\langle y, e_2\rangle=0$, then $y$ is given by $\langle y,e_2\rangle=0, \langle y,v_1-e_1\rangle=\langle y,-e_1\rangle=\frac{1}{2}$. We see that the unique solution of this linear system is given by $c_1= \left(-\frac{1}{2},\frac{1}{6}\sqrt{3},\frac{1}{12}\sqrt{6}\right)$. On the other hand, if $\langle y, e_2 \rangle <0$, then $y$ is given by $\langle y,v_1\rangle=0, \langle y,-e_1\rangle=\langle y,-e_2\rangle=\frac{1}{2}$. The unique solution is now given by $c_3= \left(-\frac{1}{2},-\frac{1}{6}\sqrt{3},\frac{1}{6}\sqrt{6}\right)$. This shows \eqref{eq:cornersHCPb} and concludes case {\rm (b)}.\\
\noindent \emph{Corners in case {\rm (c)}.}
We claim that in the case of $b_0=v_1$ that the corners of $S_{b_0}$ are given by the points
\begin{align}\label{eq:cornersHCPc}
\begin{split}
&c_1 = \left(\frac{1}{2},\frac{1}{6}\sqrt{3},\frac{1}{12}\sqrt{6}\right)\,,  c_2 = \left(0,0,\frac{1}{4}\sqrt{6}\right)\,,\\&c_3 = \left(0,\frac{1}{3}\sqrt{3},\frac{1}{6}\sqrt{6}\right)\,, c_4 = \left(\frac{1}{2},-\frac{1}{6}\sqrt{3},\frac{1}{6}\sqrt{6}\right)\,.
\end{split}
\end{align}
In particular, the face $S_{b_0}$ is a rhombus. Hence, $\mathcal{H}^2(S_{b_0})= \frac{1}{8}\sqrt{6}$. It remains to prove \eqref{eq:cornersHCPa}. Let $y \in S_{b_0}$ be a corner. By the definition of $S_{b_0}$ we have that $\langle y,v_1\rangle \geq \langle y,v_1-e_1\rangle,\langle y,v_1-e_2\rangle$ which is equivalent to $\langle y,e_1\rangle,\langle y,e_2\rangle \geq 0$.  Now if, $\langle y,e_2\rangle >0$ then the corner solves the linear system $\langle y,e_1\rangle = \langle y,e_2\rangle = \langle y,v_1\rangle=\frac{1}{2}$. Its unique solution is $c_1=(\frac{1}{2},\frac{1}{6}\sqrt{3},\frac{1}{12}\sqrt{6})$. On the other hand if  $\langle y,e_2\rangle=0$, then the corners are given by those $y$ such that $\langle y,e_2\rangle=0,  \langle y,e_1\rangle = \langle y,v_1\rangle=\frac{1}{2}$ or $\langle y,e_1\rangle=\langle y,e_2\rangle =0, \langle y,v_1\rangle=\frac{1}{2}$. These points have coordinates $c_2=(\frac{1}{2},-\frac{1}{6}\sqrt{3},\frac{1}{6}\sqrt{6})$ and $c_3=(0,0,\frac{1}{4}\sqrt{6})$. Finally, if $\langle y,e_1\rangle=0$ and $\langle y,e_2\rangle >0$, then $y$ is obtained by solving $\langle y,e_1\rangle=0,  \langle y,e_2\rangle = \langle y,v_1\rangle=\frac{1}{2}$. Hence it has coordinates $c_4=(0,\frac{\sqrt{3}}{3},\frac{\sqrt{6}}{6})$. This proves \eqref{eq:cornersHCPc} and concludes Step 2.
\end{step}\\
\noindent \begin{step}{3}(Neighbors share faces) We want to show that  for each $b_0 \in \mathcal{N}_{\mathrm{HCP}}$ we have that the face $S_{b_0}$ of $\mathcal V_{\mathcal L_{\mathrm{HCP}}}(0)$ is shared with the Voronoi cell $\mathcal V_{\mathcal L_{\mathrm{HCP}}}(b_0)$. By Step 1 we have that $\mathcal V_{\mathcal L_{\mathrm{HCP}}}(0)= V_{\mathrm{HCP}}$ and $\mathcal V_{\mathcal L_{\mathrm{HCP}}}(b_0) = b_0 - V_{\mathrm{HCP}}$. Hence, they share the side $\langle y,b_0\rangle = \frac{1}{2} = \langle b_0-y,b_0\rangle$. 
\end{step}

\noindent \begin{step}{4}(Volume of the Voronoi cell)
In order to calculate the volume of the Voronoi cell we note that $\mathcal{L}_{\mathrm{HCP}}$ is periodic with respect to the vectors $e_1,e_2,e_3$. Since the Voronoi cells of all the points occupy the same volume, it suffices to calculate the fraction of points per unit volume. The inverse of this number is the volume per point. Since the Voronoi cells are space filling the volume per point is equal to the volume of each Voronoi cell. Due to \eqref{def:periodicitycellHCP} we have that
\begin{align*}
|T_{\mathrm{HCP}}| = \sqrt{2}\,.
\end{align*}
Furthermore, we have that
\begin{align*}
\bigcup_{x \in\mathrm{span}_{\mathbb{Z}}\{e_1,e_2,e_3\} } (x+ T_{{\mathrm{HCP}}}) =\mathbb{R}^3\,, \text{ and } \mathcal{L}_{\mathrm{HCP}} \cap T_{{\mathrm{HCP}}} = \{0, v_1\}\,.
\end{align*}
Hence, the volume per point is $\frac{1}{2}| T_{{\mathrm{HCP}}}|=\frac{1}{2}\sqrt{2}$ and it agrees with the volume of the Voronoi cell.  This concludes Step 4 and thus the proof of the lemma.
\end{step}
\end{proof}

We are now in the position to prove Theorem \ref{theorem:Gammaconvergence}.

\begin{proof}[Proof of Theorem \ref{theorem:Gammaconvergence}] All the statements are consequences of Proposition \ref{prop:Compactnessgeneral}, Lemma \ref{lemma:equivalenceoftopgeneral}, Theorem \ref{theorem:intrepperiodic} and Lemma \ref{lemma:concentration} once we show that $\mathcal{L}_{\mathrm{FCC}}$ and $\mathcal{L}_{\mathrm{HCP}}$ are periodic admissible sets (according to Definition \ref{def:admissible} and Definition \ref{def:periodicity}) and we observe that, due to Lemma \ref{lemma:shapeFCC} and Lemma \ref{lemma:shapeHCP}, $\mathcal{N}_{\mathrm{FCC}}(x)= \mathcal{NN}(x)$ (in the sense of Definition \ref{def:nearestneighbours}) as well as $\mathcal{N}_{\mathrm{HCP}}(x)= \mathcal{NN}(x)$ in the respective cases. We first show that both lattices are admissible sets. Let us first observe that 
\begin{align*}
(T_{\mathrm{FCC}} +x) \cap \mathcal{L}_{\mathrm{FCC}}\neq \emptyset \quad \text{and}\quad (T_{\mathrm{HCP}} +x) \cap \mathcal{L}_{\mathrm{HCP}}\neq \emptyset \text{ for all } x\in \mathbb{R}^3\,.
\end{align*}
Therefore, (L1) is satisfied for both $\mathcal{L}_{\mathrm{FCC}}$ and $\mathcal{L}_{\mathrm{HCP}}$ with 
\begin{align*}
R:= \max\{\mathrm{diam}(T_{\mathrm{FCC}}),\mathrm{diam}(T_{\mathrm{FCC}})\}\leq \max\left\{\sum_{k=1}^3|e_k|, \sum_{k=1}^3|b_k|\right\} <+\infty\,,
\end{align*}
where we recalled Definition \ref{def:periodicitycellFCC} and Definition \ref{def:periodicitycellHCP}.
 On the other hand, (L2) is satisfied with $r=1$, see the discussion at the definition of the FCC and HCP lattice in Section \ref{sec:Notation}.  Concerning periodicity: We observe that for all $z=(z_1,z_2,z_3) \in \mathbb{Z}^3$ we have
\begin{align*}
\mathcal{L}_{\mathrm{FCC}} = \mathcal{L}_{\mathrm{FCC}} + \sum_{k=1}^3 z_kb_k\,,\text{ and }  \mathcal{L}_{\mathrm{HCP}} =  \mathcal{L}_{\mathrm{HCP}} +\sum_{k=1}^3 z_ke_k 
\end{align*}
and thus both $\mathcal{L}_{\mathrm{FCC}}$ and $\mathcal{L}_{\mathrm{HCP}}$ are periodic according to Definition \ref{def:periodicity}. \GGG The statement follows by Theorem \ref{theorem:intrepperiodic} with $c_{nn}(x)=1$. \EEE
\end{proof}
\section{General periodic lattices}\label{appendix:general}
This section deals with integral representation and concentrated-compactness properties of energies defined on general periodic lattices. 

\begin{definition}\label{def:admissible}
Let $\Sigma \subset \mathbb{R}^n$ be a countable set of points in $\mathbb{R}^n$. We call $\Sigma$ an \emph{admissible} set of points if the following two conditions hold:
\begin{itemize}
\item[(L1)] There exists $R>0$ such that $\inf_{ x\in \mathbb{R}^n} \#(\Sigma \cap B_R(x) ) \geq 1$;
\item[(L2)] There exists $r>0$ such that $ \mathrm{dist}(x,\Sigma \setminus \{x\}) \geq r$ for all $x\in \Sigma$.
\end{itemize}
\end{definition}
\begin{definition} \label{def:nearestneighbours} We define the Voronoi cell of $x \in \Sigma$ as
\begin{align}\label{def:VoronoiSigma}
\mathcal{V}(x):=\{z \in \mathbb{R}^n \colon |x-z|\leq |y-z| \text{ for all } y \in \Sigma\}\,.
\end{align}
 The set of nearest neighbors of $\Sigma$ is defined by
\begin{align*}
\mathcal{NN}(\Sigma) :=\{(x,y) \in \Sigma\times \Sigma \colon \mathcal{H}^{n-1}(\mathcal{V}(x) \cap \mathcal{V}(y))>0\}\,,
\end{align*}
 We set $\mathcal{NN}(x)= \{y\in \Sigma \colon (x,y) \in \mathcal{NN}(\Sigma)\}$. Given $\varepsilon>0$ we denote by $\varepsilon\Sigma := \{\varepsilon x \colon x \in \Sigma\}$ and for $x \in \varepsilon\Sigma$ we set $\mathcal{V}_\varepsilon(x)= \varepsilon\mathcal{V}(\varepsilon^{-1} x)$ the Voronoi cell of $x \in \varepsilon \Sigma$, and $\mathcal{NN}_\varepsilon(x) = \{y \in \varepsilon\Sigma \colon \varepsilon^{-1}(x,y) \in \mathcal{NN}(\Sigma)\}$ the set of nearest neighbors of $x$ in $ \varepsilon\Sigma$.  
\end{definition}
 \GGG
We now define for $u \colon \Sigma \to\{0,1\}$ the two energy functionals given by
\begin{align}\label{def:Feps}
F_\varepsilon(u,A)  := \underset{\varepsilon x \in A}{\sum_{(x,y) \in \mathcal{NN}(\Sigma)}} \varepsilon^{n-1} c_{nn}(x-y) |u(\varepsilon x)-u(\varepsilon  y)|\,
\end{align}
and 
\begin{align}\label{def:hatFeps}
\hat F_\varepsilon(u,A)  := \underset{\varepsilon x, \varepsilon y \in A}{\sum_{(x,y) \in \mathcal{NN}(\Sigma)}} \varepsilon^{n-1} c_{nn}(x-y) |u(\varepsilon x)-u(\varepsilon  y)|\,,
\end{align}
where $c_{nn} \colon \mathbb{R}^n \to [0,+\infty]$ satisfies 
\begin{align}\label{ass:boundcnn}
C^{-1}\leq c_{nn}(x) \leq C\,, \quad \text{for all } x\in \mathbb{R}^n\,.
\end{align}
When $A=\mathbb{R}^n$ we omit the dependence on it and write $F_\varepsilon(u)=F_\varepsilon(u,\mathbb{R}^n)$ and $\hat F_\varepsilon(u)= \hat F_\varepsilon(u,\mathbb{R}^n)$. \begin{remark}(Difference between $F_\varepsilon$ and $\hat F_\varepsilon$) We want to point out the difference between $F_\varepsilon$ and $\hat F_\varepsilon$: In the formula defining $F_\varepsilon$ the sum is taken over all $(x,y) \in \mathcal{NN}(\Sigma)$ such that $\varepsilon x \in A$. Instead in the case of $\hat{F}_\varepsilon$ the sum takes only those $(x,y) \in \mathcal{NN}(\Sigma)$ such that both $\varepsilon x \in A$ and $\varepsilon y \in A$. The functional $F_\varepsilon(u,\cdot)$ is an additive set function on disjoint sets, i.e., given $A,B \subset \mathbb{R}^n$ such that $A\cap B =\emptyset$, we have 
\begin{align*}
F_\varepsilon(u,A\cup B) = F_\varepsilon(u,A) + F_\varepsilon(u,B)\,,
\end{align*}
whereas $\hat F_\varepsilon(u,\cdot)$ is only super-additive on disjoint sets. Our $\Gamma$-convergence result will be stated for the functional $F_\varepsilon$. The reason for us to introduce $\hat{F}_\varepsilon$ is that our proof will use the integral representation result proven in \cite{AliCicRuf}, see Theorem \ref{theorem:AliCicRuf}. However, we will show later on that the $\Gamma$-convergence of $F_\varepsilon$ is equivalent to that of $\hat F_\varepsilon$. 
\end{remark}
Given $X \subset \varepsilon\Sigma$ we write with a slight abuse of notation
\begin{align*}
F_\varepsilon(X,A)= F_\varepsilon(\chi_X,A)\,.
\end{align*}
Note that, due to \eqref{ass:boundcnn}, we have
\begin{align*}
C^{-1}\sum_{x \in X} \varepsilon^{n-1} \#(\mathcal{NN}_\varepsilon(x)\setminus X)\leq F_{\varepsilon}(X,A) \leq C\sum_{x \in X} \varepsilon^{n-1} \#(\mathcal{NN}_\varepsilon(x)\setminus X)\,.
\end{align*}
Hypothesis \eqref{ass:boundcnn} corresponds to   \cite[Hypothesis 1]{AliCicRuf} in the case that, according to the notation in  \cite[Equation (5.23)]{AliCicRuf},  $c_{nn}^\varepsilon(x,y)=c_{nn}(x-y)$ and $c_{lr}^\varepsilon(x,y)=0$. It is worth observing that in \cite{AliCicRuf} a more general class of functionals was investigated, namely those for which also certain long-range interactions between points in $\Sigma$ contribute to the energy, i.e., $c_{lr}(x,y)\neq 0$. For the sake of exposition and simplicity, here we consider the case $c_{lr}=0$, that is the energy accounts only for the nearest neighbor interactions. However, with some more involved multi-scale constructions, all the statements below extend to the more general case where also long-range interactions are considered. 
\begin{definition}\label{def:extension} Given $X\subset \varepsilon\mathcal{L}$ we define the rescaled empirical measures associated to $X$  as
\begin{align}\label{def:empiricalmeasures}
\mu_\varepsilon = \varepsilon^n \sum_{x \in X} \delta_x\,.
\end{align}
Furthermore, recalling \eqref{def:VoronoiSigma}, we define 
\begin{align}\label{def:piecewiseconstantinterpolant}
V_\varepsilon(X) :=\bigcup_{x\in X} {\mathcal V}_{\varepsilon}(x)\,.
\end{align}
Henceforth, we drop the dependence on $X$ and simply write $V_\varepsilon$.
Given $A\subset \mathbb{R}^n$ open with $\partial A \in \mathrm{Lip}$, with slight abuse of notation we define $F_\varepsilon \colon \mathcal{M}_+(A)\to [0,+\infty]$ (similarly $\hat F_\varepsilon \colon \mathcal{M}_+(A) \to [0,+\infty]$) by
\begin{align*}
F_\varepsilon(\mu,A) = \begin{cases} F_\varepsilon(X,A) &\mu \text{ is given by } \eqref{def:empiricalmeasures} \text{ for some } X\subset \varepsilon\mathcal{L}\,;\\
+\infty &\text{otherwise.}
\end{cases}
\end{align*}
Additionally, we define  $F_\varepsilon \colon L^1_{\mathrm{loc}}(A)\to [0,+\infty]$ (similarly $\hat F_\varepsilon \colon L^1_{\mathrm{loc}}(A)\to [0,+\infty]$) by
\begin{align*}
F_\varepsilon(u,A) = \begin{cases} F_\varepsilon(X,A) &u=\chi_{V_\varepsilon} \text{ and } V_\varepsilon \text{ is given by } \eqref{def:piecewiseconstantinterpolant} \text{ for some } X \subset \varepsilon\mathcal{L}\,;\\
+\infty &\text{otherwise.}
\end{cases}
\end{align*}
\end{definition}
It is necessary for us to introduce two different domains of definition for the extended functional $F_\varepsilon$, since we want to make use of \cite[Theorem 5.5]{AliCicRuf}. As it will turn out the two types of extension are equivalent, cf.~Lemma \ref{lemma:equivalenceoftopgeneral} and Corollary \ref{cor:equivalenceGamma}.

\GGG Let $(\Omega,\mathbb{P},\mathcal{F})$ be a probability space. Hereafter we recall some definitions from \cite{AliCicRuf} (Definition~5.1 and Definition~5.4):

\begin{definition} \label{def:groupaction}  We say that a family $(\tau_z)_{z \in \mathbb{Z}^n}$, $\tau_z \colon \Omega \to \Omega$, is an \emph{additive group action} on $\Omega$ if
\begin{align*}
\tau_{z_1+z_2} = \tau_{z_2} \circ \tau_{z_1} \quad \text{for all } z_1,z_2 \in \mathbb{Z}^n\,.
\end{align*}
Such an additive group action is called \emph{measure preserving} if
\begin{align*}
\mathbb{P}(\tau_zB) = \mathbb{P}(B) \quad \text{for all } B \in \mathcal{F}, z \in \mathbb{Z}^n\,.
\end{align*}
If in addition, for all $B \in \mathcal{F}$ we have
\begin{align*}
(\tau_z(B)=B \quad \text{for all } z\in \mathbb{Z}^n) \implies \mathbb{P}(B) \in \{0,1\}\,,
\end{align*}
then $(\tau_z)_{z \in \mathbb{Z}^n}$ is called \emph{ergodic}.
\end{definition}

\begin{definition}\label{def:statergod} A random variable $\mathcal{L} \colon \Omega \to (\mathbb{R}^n)^{\mathbb{Z}^n}$, $\omega \mapsto \mathcal{L}(\omega) = \{\mathcal{L}(\omega)(i)\}_{i \in \mathbb{Z}^n}$is called a stochastic
lattice. We say that $\mathcal{L}$ is admissible if $\mathcal{L}(\omega)$ is admissible in the sense of Definition \ref{def:admissible} and the constants $r,R$ can be chosen independent of $\omega$ $\mathbb{P}$-almost surely. The stochastic lattice $\mathcal{L}$ is said to be \emph{stationary} if
there exists a measure preserving group action $(\tau_z)_{z\in \mathbb{Z}^n}$ on $\Omega$ such that, for $\mathbb{P}$-almost every $\omega \in \Omega$,
$\mathcal{L}(\tau_z\omega) =\mathcal{L}(\omega)+z$.
If in addition $(\tau_z)_{z \in \mathbb{Z}^n}$ is ergodic, then $\mathcal{L}$ is called ergodic, too.
\end{definition} \EEE

 We now state a simplified version of  \cite[Theorem 5.5]{AliCicRuf}  which is enough for our purposes.

\begin{theorem}\label{theorem:AliCicRuf}(Stochastic homogenization of spin systems) \GGG Let $\mathcal{L}$ be a stationary and ergodic stochastic lattice and let $\hat F_\varepsilon$ be defined by \eqref{def:hatFeps}. Let $A\subset \mathbb{R}^n$ be open and bounded with $\partial A \in \mathrm{Lip}$. For $\mathbb{P}$-almost every $\omega$ the functionals $F_\varepsilon(\omega)$ $\Gamma$-converge with respect to the strong $L^1(A)$-topology to the functional $F_\mathrm{hom}\colon L^1(A) \to [0,+\infty]$ defined by \EEE
\begin{align*} 
F_\mathrm{hom}(u,A):=\begin{cases} \int_{J(u)\cap A} \varphi_{\mathrm{hom}}(\nu_u)\,\mathrm{d}\mathcal{H}^{n-1} &\text{if } u \in BV(A;\{0,1\}),\\
+\infty &\text{otherwise.}
\end{cases}
\end{align*}
The function $\varphi_{\mathrm{hom}} \colon \mathbb{R}^n \to [0,+\infty]$ is given by
\begin{align*}
\varphi_{\mathrm{hom}}(\nu) =\lim_{T\to +\infty} \frac{1}{T^{n-1}} \int_\Omega \inf\big\{&F(u,Q_T^\nu) \colon u \colon \mathcal{L}(\omega) \to \{0,1\}, \\&u(i) = u_\nu(i) \text{ for } i \in \mathcal{L}(\omega) \setminus Q^\nu_{T-l_T}\big\}\,\mathrm{d}\mathbb{P}(\omega)\,,
\end{align*}
where $l_T \to +\infty$ and $l_T/T \to 0$  as $T\to +\infty$.
\end{theorem}
\begin{definition} \label{def:periodicity} Let $\mathcal{L} \subset \mathbb{R}^n$ be an admissible set of points. We say that $\mathcal{L}$ is periodic if there exists a basis $\{e_1,\ldots,e_n\} \subset \mathbb{R}^n$ such that
\begin{align*}
\mathcal{L}+e_k = \mathcal{L} \text{ for all } k =1,\ldots,n\,.
\end{align*}
We denote by
\begin{align*}
Q:=\left\{\sum_{k=1}^n \lambda_k e_k \colon 0\leq \lambda_k<1\right\}
\end{align*}
the periodicity cell of $\mathcal{L}$. We set
\begin{align}\label{def:rho}
\rho:= \frac{\#(\mathcal{L}\cap Q)}{|Q|}\,.
\end{align}
\end{definition}

In the following we assume, up to a change of coordinates that $\{e_1,\ldots, e_n\}$ is the standard orthonormal basis of $\mathbb{R}^n$.

We collect the following general properties of periodic admissible set of points. 
\begin{lemma}(Properties of periodic admissible sets) \label{lemma:propertiesadmissiblesets} Let $\mathcal{L}$ be a periodic admissible set of points. The following holds true:
\begin{itemize}
\item[(i)] $B_{r/2}(x) \subset \mathcal{V}(x) \subset B_R(x)$ for all $x \in \mathcal{L}$;
\item[(ii)] $\mathcal{NN}(x) \subset B_{2R}(x)$ for all $x\in \mathcal{L}$;
\item[(iii)] There exists $C=C(n,r,R) \in (0,+\infty)$ such that $\displaystyle\sup_{x \in \mathcal{L}} \#\mathcal{NN}(x) \leq C$. In particular, $\partial \mathcal{V}(x)$ is made out of finitely many $(n-1)$-dimensional polyhedral faces;
\item[(iv)] There exists $C=C_\mathcal{L}>0$ such that for all $X \subset \varepsilon\mathcal{L}$ and $A\subset \mathbb{R}^n$ there holds
\begin{align*}
 F_{\varepsilon}(X,A)  \leq C\,\mathrm{Per}(V_\varepsilon,(A)_{R\varepsilon}) \text{ and }  
\,\mathrm{Per}(V_\varepsilon,A) \leq C  F_{\varepsilon}(X,(A)_{R\varepsilon}) \,.
\end{align*}
\end{itemize}
\end{lemma}
\begin{proof} Apart from (iv) all of these facts are classical. We collect their proof here for completeness.\\
\noindent\textit{proof of (i), (ii):}   Let $x \in \mathcal{L}$. The inclusion $B_{r/2}(x)\subset \mathcal{V}(x)$ follows from (L2) since for all $y \in B_{r/2}(x)$ and $z \in \mathcal{L}\setminus \{x\}$
\begin{align*}
|z-y| \geq |x-z|-|x-y| \geq r- r/2=r/2\geq |x-y|\,.
\end{align*}
As for the inclusion $\mathcal{V}(x) \subset B_R(x)$  assume that there exists $y \in\mathcal{V}(x) \setminus B_R(x) $. We have for all $z\in \mathcal{L}\setminus \{x\}$ 
\begin{align*}
|y-z|\geq |y-x|\geq R\,.
\end{align*}
This implies that $B_R(y) \cap \mathcal{L}=\emptyset$ contradicting (L1).  Finally  $\mathcal{NN}(x) \subset B_{2R}(x)$ since for $y \in \mathcal{NN}(x)$ we have that $\mathcal{V}(x) \cap \mathcal{V}(y) \neq \emptyset$ which implies $B_{R}(x) \cap B_R(y)\neq \emptyset$.\\

\noindent\textit{proof of (iii):}  Due to (i) and (ii) we have that $B_{r/2}(y) \cap B_{r/2}(z) =\emptyset$ $y,z \in \mathcal{NN}(x)$, $y\neq z$ and $B_{r/2}(y) \subset B_{2R+r}(x)$ for all $y \in \mathcal{NN}(x)$. Therefore
\begin{align*}
\omega_n \left(\frac{r}{2}\right)^n\#\mathcal{NN}(x) = \bigcup_{y \in \mathcal{NN}(x)} |B_{r/2}(y)| \leq |B_{2R+r}(x)| \leq \omega_n (2R+r)^n
\end{align*}
and thus the claim follows with $C= (2+ 4R/r)^n$.\\
\noindent \textit{proof of (iv):} The desired inequalities follow from the following observation: Given $x \in X$, we have that
\begin{align}\label{equiv:voronoinn}
\mathcal{H}^{n-1}(\mathcal{V}_\varepsilon(x) \cap \partial V_\varepsilon) >0 \iff \mathcal{NN}_\varepsilon(x)\setminus X \neq \emptyset\,.
\end{align}
Additionally, we note  that, for $x\in X$ such that  $\mathcal{NN}_\varepsilon(x)\setminus X \neq \emptyset$, there exists $C>0$ such that
\begin{align}\label{ineq:voronoiinterface}
C^{-1}\varepsilon^{n-1}\leq \mathcal{H}^{n-1}(\mathcal{V}_\varepsilon(x) \cap \partial V_\varepsilon) \leq C\varepsilon^{n-1}\,.
\end{align}
Now, summing over all $x \in X \cap A$ and noting that each Voronoi cell intersects only a finite number of other Voronoi cells, using \eqref{equiv:voronoinn},\eqref{ineq:voronoiinterface},(i), and (iii), we obtain
\begin{align*}
F_{\varepsilon}(X,A)=\sum_{x \in X \cap A} \varepsilon^{n-1} \#(\mathcal{NN}_\varepsilon(x) \setminus X) &\leq C\sum_{x \in X \cap A} \mathcal{H}^{n-1}(\mathcal{V}_\varepsilon(x) \cap \partial V_\varepsilon)\\&\leq C \mathcal{H}^{n-1}(\partial V_\varepsilon \cap (A)_{R\varepsilon})=C\,\mathrm{Per}(V_\varepsilon,(A)_{R\varepsilon})\,.
\end{align*}
This yields the first inequality in (iv). On the other hand, owing to (i) we have $\partial V_\varepsilon \cap A \subset \bigcup_{x \in X \cap (A)_{R\varepsilon}} (\mathcal{V}_\varepsilon(x) \cap\partial V_\varepsilon)$, and thus, by \eqref{equiv:voronoinn} and \eqref{ineq:voronoiinterface},
\begin{align*}
\mathrm{Per}(V_\varepsilon,A)=\mathcal{H}^{n-1}(\partial V_\varepsilon \cap A) &\leq \sum_{x \in X \cap (A)_{R\varepsilon}} \mathcal{H}^{n-1}(\mathcal{V}_\varepsilon(x) \cap \partial V_\varepsilon) \\&\leq C \varepsilon^{n-1} \sum_{x \in X \cap (A)_{R\varepsilon}} \#(\mathcal{NN}_\varepsilon(x) \setminus X)=  C F_{\varepsilon}(X,(A)_{R\varepsilon})\,.
\end{align*} 
This shows the second inequality in (iv) and concludes the proof.
\end{proof}

\begin{proposition}(Compactness of the piecewise-constant interpolants)\label{prop:Compactnessgeneral} Let $\mathcal{L}$  be an admissible periodic set of points and $F_\varepsilon$ defined in \eqref{def:Feps} with $\mathcal{L}$ in place of $\Sigma$.  Let $A \subset\mathbb{R}^n$ be open and let  $\{X_\varepsilon\}_\varepsilon \subset \varepsilon\mathcal{L}$ be such that 
\begin{align*}
\sup_{\varepsilon >0} F_{\varepsilon}(X_\varepsilon,A) <+\infty\,.
\end{align*}
Then there exists a set of finite perimeter $V\subset A$ and a subsequence (not relabeled) such that $\chi_{V_\varepsilon} \to  \chi_V$ with respect to the strong $L^1_{\mathrm{loc}}(A)$-topology.
\end{proposition}
\begin{proof} Let $X_\varepsilon$ be as above and let $A'\subset\subset A$ such that $(A')_{R\varepsilon} \subset A$, and $\partial A'\in \mathrm{Lip}$. We observe, due to the second inequality of Lemma \ref{lemma:propertiesadmissiblesets}(iv),
\begin{align*}
 \mathrm{Per}(V_\varepsilon, A')\leq CF_\varepsilon(X_\varepsilon,  ( A')_{R \varepsilon}) \leq CF_\varepsilon(X_\varepsilon,A)\leq C\,.
\end{align*}
Therefore 
\begin{align*}
\|\chi_{V_\varepsilon}\|_{L^1(A')} +|D\chi_{V_\varepsilon}|(A') \leq C(|A'| + \mathrm{Per}(V_\varepsilon, A'))\leq  C\,.
\end{align*}
We use \cite[Theorem 3.39]{AFP} to deduce that there exists a subsequence (depending on $A'$) and a set of finite perimeter $V$ such that $\chi_{V_\varepsilon} \to \chi_{V}$ in $L^1(A')$. By a diagonal argument on a sequence $A'_k \uparrow A$ as $k \to +\infty$, we obtain the claim.
\end{proof}

\begin{lemma}(Equivalence of convergences)\label{lemma:equivalenceoftopgeneral} Let $\mathcal{L}$ be an admissible periodic set of points and $F_\varepsilon$ defined in \eqref{def:Feps} with $\mathcal{L}$ in place of $\Sigma$. Let $A \subset\mathbb{R}^n$ be open and let $V \subset A$ be a set of finite perimeter and let $\{X_\varepsilon\}_\varepsilon \subset \varepsilon\mathcal{L}$ for each $\varepsilon>0$ be such that
\begin{align}\label{ineq:uniform energy}
\sup_{\varepsilon >0} F_\varepsilon(X_\varepsilon,A) <+\infty\,.
\end{align}
 Then, setting $\mu_\varepsilon$ and $V_\varepsilon$ as in \eqref{def:empiricalmeasures} and \eqref{def:piecewiseconstantinterpolant}, the following are equivalent:
\begin{itemize}
\item[(i)] $\mu_\varepsilon \overset{*}{\rightharpoonup} \mu$ with respect to the weak star topology of measures and $\mu = \rho\mathcal{L}^n\llcorner V$.
\item[(ii)] $\chi_{V_\varepsilon} \to \chi_V$ with respect to the strong $L^1_{\mathrm{loc}}(A)$-topology.
\end{itemize}
\end{lemma}
\begin{proof} We proceed in two steps. First, we construct a sequence of auxiliary measure $\nu_\varepsilon$ and show that its weak*-convergence is equivalent to the weak*-convergence of the sequence of measures $\mu_\varepsilon$. Then, for this sequence of measures we show that its weak*-convergence is equivalent to~(ii).\\
\begin{step}{1}(Construction of the auxiliary measure)
Let $\{X_\varepsilon\}_\varepsilon$ be as in the assumptions of the lemma and let $v \in C_c(A)$ such that $\mathrm{supp}\, v \subset\subset A$. We assume that $\varepsilon>0$ is small enough such that for all $k \in \varepsilon\mathbb{Z}^n$ there holds
\begin{align}\label{cond:cubes}
Q_{(3+R)\varepsilon}(k) \cap \mathrm{supp}\, v \neq \emptyset \implies Q_{(3+R)\varepsilon}(k) \subset\subset A\,.
\end{align}
 Fix $R_0>0$ such that $\mathrm{supp}\, v \subset B_{R_0}$. Since $v$ is uniformly continuous, it admits a modulus of continuity $\omega=\omega_v \colon [0,+\infty) \to [0,+\infty)$, i.e., an increasing function such that $\omega(0)=0$ and
\begin{align}\label{ineq:uniformcontinuity}
|v(x)-v(y)|\leq \omega(|x-y|) \text{ for all } x,y \in \mathbb{R}^n\,.
\end{align}
 We set 
\begin{align*}
\begin{split}
&\mathcal{I}_\varepsilon^\mathrm{full}:=\{k \in \varepsilon\mathbb{Z}^d \colon Q_{(3+R)\varepsilon}(k)\cap X = Q_{(3+R)\varepsilon}(k)\cap \varepsilon\mathcal{L}\}\,,\\& \mathcal{I}_\varepsilon^\mathrm{empty}:=\{k \in \varepsilon\mathbb{Z}^d \colon Q_{(3+R)\varepsilon}(k)\cap X = \emptyset\}\,,
\end{split}
\end{align*}
and $\mathcal{I}_\varepsilon^\mathrm{bad} :=\{ k \in \varepsilon\mathbb{Z}^n \colon Q_{(3+R)\varepsilon}(k) \cap \mathrm{supp}\, v \neq \emptyset\}  \setminus (\mathcal{I}_\varepsilon^\mathrm{full} \cup  \mathcal{I}_\varepsilon^\mathrm{empty})$. We now set
\begin{align*}
\nu_\varepsilon : = \sum_{k \in \mathcal{I}_\varepsilon^\mathrm{full}} \varepsilon^n \#(\mathcal{L}\cap Q)\delta_k\,.
\end{align*}
Our goal is to show that
\begin{align}\label{equiv:convergencenumu}
\mu_\varepsilon \overset{*}{\rightharpoonup} \mu \quad \text{in the sense of measures} \quad \iff \quad \nu_\varepsilon \overset{*}{\rightharpoonup} \mu \quad \text{in the sense of measures.}
\end{align}
First of all we claim that
\begin{align}\label{ineq:boundbadpoints}
\#\mathcal{I}_\varepsilon^{\mathrm{bad}} \leq C \varepsilon^{1-n}\,.
\end{align}
In fact, due to its definition and \eqref{cond:cubes}, we have that for all $k \in  \mathcal{I}_\varepsilon^\mathrm{bad}$ there exists $x \in X \cap Q_{(3+R)\varepsilon}(k) \subset A$ such that $\#(\mathcal{NN}(x) \setminus X) \neq \emptyset$. Therefore, since for $k \in \varepsilon\mathbb{Z}^n$ fixed  $Q_{(3+R)\varepsilon}(k)\cap Q_{(3+R)\varepsilon}(j)\neq \emptyset$ for only finitely many $j$ independent of $\varepsilon$, we have
\begin{align*}
\varepsilon^{n-1}\#\mathcal{I}_\varepsilon^\mathrm{bad} \leq \sum_{k \in \mathcal{I}_\varepsilon^\mathrm{bad}} \sum_{x \in Q_{(3+R)\varepsilon}(k)\cap X}\varepsilon^{n-1} \#(\mathcal{NN}(x) \setminus X)\leq CF_\varepsilon(X_\varepsilon,A)\,.
\end{align*}
Using  \eqref{ineq:uniform energy} yields \eqref{ineq:boundbadpoints}.
Let now $k \in \mathcal{I}_\varepsilon^\mathrm{full}$. Then
\begin{align*}
\int_{Q_\varepsilon(k)} v(x) \,\mathrm{d}\nu_\varepsilon(x) = v(k) \varepsilon^n \#(\mathcal{L}\cap Q) &=\sum_{x \in Q_{\varepsilon}(k) \cap \varepsilon\mathcal{L}} \varepsilon^n v(k) \\&= \sum_{x \in Q_{\varepsilon}(k) \cap X} \varepsilon^n v(x) +\sum_{x \in Q_{\varepsilon}(k) \cap \varepsilon\mathcal{L}} \varepsilon^n (v(k)- v(x))\\&=\int_{Q_\varepsilon(k)}v(x)\,\mathrm{d}\mu_\varepsilon(x) +\sum_{x \in Q_{\varepsilon}(k) \cap \varepsilon\mathcal{L}} \varepsilon^n (v(k)- v(x))\,.
\end{align*}
Thus,
\begin{align*}
\left|\int_{Q_\varepsilon(k)} v(x) \,\mathrm{d}(\nu_\varepsilon -\mu_\varepsilon)(x)\right|\leq \sum_{x \in Q_{\varepsilon}(k) \cap \varepsilon\mathcal{L}} \varepsilon^n |v(k)- v(x)| \leq \omega(\varepsilon \sqrt{n}) |Q_\varepsilon(k)|
\end{align*}
and, recalling that $\mathrm{supp}\, v \subset B_{R_0}$ and \eqref{ineq:uniformcontinuity}, we have for $\varepsilon>0$ small enough
\begin{align}\label{ineq:differencemeasuresfull}
\sum_{k \in \mathcal{I}_\varepsilon^\mathrm{full}}\left|\int_{Q_\varepsilon(k)} v(x) \,\mathrm{d}(\nu_\varepsilon -\mu_\varepsilon)(x)\right|\leq  \omega(\varepsilon\sqrt{n}) |B_{2R_0}|\,.
\end{align}
 Noting that  both  $|\mu_\varepsilon|(Q_\varepsilon(k))$ and $|\nu_\varepsilon|(Q_\varepsilon(k))$ are bounded above by  
 $\varepsilon^n\#(\mathcal{L}\cap Q)\leq C\varepsilon^n$ for all $k \in \varepsilon\mathbb{Z}^n$, using \eqref{ineq:boundbadpoints}, we observe
\begin{align}\label{ineq:differencemeasuresbad}
\sum_{k \in \mathcal{I}_\varepsilon^\mathrm{bad}}\left|\int_{Q_\varepsilon(k)} v(x) \,\mathrm{d}(\nu_\varepsilon -\mu_\varepsilon)(x)\right|\leq 2\|v\|_\infty\#\mathcal{I}_\varepsilon^\mathrm{bad} \varepsilon^n\#(\mathcal{L}\cap Q)\leq C\varepsilon\|v\|_\infty\,.
\end{align}
 Therefore, noting that $\mu_\varepsilon\lfloor_{Q_\varepsilon(k)}=\nu_\varepsilon\lfloor_{Q_\varepsilon(k)}$ for $k \in \mathcal{I}_\varepsilon^{\mathrm{empty}}$, using \eqref{ineq:differencemeasuresfull} and \eqref{ineq:differencemeasuresbad}, we obtain
\begin{align*}
\begin{split}
\left|\int_{\mathbb{R}^n}v(x)\,\mathrm{d}(\nu_\varepsilon-\mu_\varepsilon)(x)\right|&\leq \omega(\varepsilon\sqrt{n})|B_{2R}| + C\varepsilon\|v\|_\infty\,.
\end{split}
\end{align*}
This shows \eqref{equiv:convergencenumu}.
\end{step}\\
\begin{step}{2}(Equivalence of convergence) We now prove that 
\begin{align*}
\chi_{V_\varepsilon}\to\chi_V \quad \text{in }L^1_{\mathrm{loc}}(A) \quad \iff \quad \nu_\varepsilon \overset{*}{\rightharpoonup} \mu \quad \text{in the sense of measures.}
\end{align*}
First of all, recalling $\rho$ defined in \eqref{def:rho}, we note
\begin{align}\label{eq:nuepscubes}
\begin{split}
\int_{\mathbb{R}^n} v\,\mathrm{d}\nu_\varepsilon &=\sum_{k\in \mathcal{I}_\varepsilon^\mathrm{full}} \varepsilon^n\#(\mathcal{L}\cap Q) v(k) = \frac{\#(\mathcal{L}\cap Q)}{|Q|} \sum_{k \in \mathcal{I}_\varepsilon^\mathrm{full}} |Q_\varepsilon(k)|  v(k) \\&= \rho\sum_{k\in \mathcal{I}_\varepsilon^\mathrm{full}}  \int_{Q_\varepsilon(k)} v(y)\,\mathrm{d}y + \rho\sum_{k\in \mathcal{I}_\varepsilon^\mathrm{full}}\int_{Q_\varepsilon(k)}(v(k)-v(y))\,\mathrm{d}y\,.
\end{split}
\end{align}
Now, due to \eqref{ineq:uniformcontinuity}, we have for $\varepsilon>0$ small enough
\begin{align}
\left|\sum_{k\in \mathcal{I}_\varepsilon^\mathrm{full}}\int_{Q_\varepsilon(k)}(v(k)-v(y))\,\mathrm{d}y\right|\leq\sum_{k\in \mathcal{I}_\varepsilon^\mathrm{full}}\int_{Q_\varepsilon(k)}|v(k)-v(y)|\,\mathrm{d}y\leq C\omega(\varepsilon\sqrt{n}) |B_{2R_0}|\,.
\end{align}
Note that, by Lemma \ref{lemma:propertiesadmissiblesets}(i), we have
\begin{align}
\rho\sum_{k\in \mathcal{I}_\varepsilon^\mathrm{full}}  \int_{Q_\varepsilon(k)} v(y)\,\mathrm{d}y  = \rho\sum_{k\in \mathcal{I}_\varepsilon^\mathrm{full}}  \int_{Q_\varepsilon(k)\cap V_\varepsilon} v(y)\,\mathrm{d}y 
\end{align}
and also $V_\varepsilon \cap Q_\varepsilon(k) =\emptyset$ for $k \in \mathcal{I}_\varepsilon^\mathrm{empty}$. Note that by \eqref{ineq:boundbadpoints} we have
\begin{align}\label{ineq:Iepsbad}
\sum_{k\in \mathcal{I}_\varepsilon^\mathrm{bad}}\int_{Q_\varepsilon(k)}\left|v(y)\right|\,\mathrm{d}y \leq  \varepsilon^n \#\mathcal{I}_\varepsilon^\mathrm{bad}\|v\|_\infty \leq C\varepsilon\|v\|_\infty\,.
\end{align}
Due to \eqref{eq:nuepscubes}--\eqref{ineq:Iepsbad}, we obtain that
\begin{align*}
\rho\chi_{V_\varepsilon}\overset{*}{\rightharpoonup} \mu \quad  \text{in the sense of measures}\quad \iff \quad \nu_\varepsilon \overset{*}{\rightharpoonup} \mu \quad \text{in the sense of measures.}
\end{align*}
Now clearly (ii) implies (i), since the $L^1_{\mathrm{loc}}(A)$ convergence of the characteristic functions implies their weak* convergence as measures. As for the implication (i) to (ii) we proceed as follows. Let $V\subset A$ be a set of finite perimeter and assume that $\mu_\varepsilon \overset{*}{\rightharpoonup} \mu$ and $\mu = \rho\mathcal{L}^n\llcorner V$. By Step~1 this is equivalent to $\nu_\varepsilon \overset{*}{\rightharpoonup} \mu$,  hence to $\rho\chi_{V_\varepsilon} \overset{*}{\rightharpoonup}\rho\chi_V$ in the sense of measures (which is to say that $\chi_{V_\varepsilon} \overset{*}{\rightharpoonup}\chi_V$)  Take now an arbitrary subsequence (not relabeled) of $\{X_\varepsilon\}_\varepsilon$. We show that there exists a further subsequence (again not relabeled) such that $\chi_{V_\varepsilon} \to \chi_V$ with respect to the strong $L^1_{\mathrm{loc}}(A)$-topology. Since $L^1_{\mathrm{loc}}(A)$-topology satisfies the Urysohn property this implies the claim. By the compactness statement in \GGG Proposition \ref{prop:Compactnessgeneral} \EEE we have that there exists a set of finite perimeter $V' \subset \mathbb{R}^n$ and a further subsequence $\{X_{\varepsilon_k}\}_k \subset  \{X_\varepsilon\}_\varepsilon$ such that $\chi_{V_\varepsilon} \to \chi_{V'}$ with respect to the strong $L^1_{\mathrm{loc}}(A)$-topology. Since this implies their weak* convergence as measures and we already know that the whole sequence converges to $\chi_V$ we deduce  $V=V'$ which implies the claim and concludes the proof of the lemma.
\end{step}
\end{proof}

\GGG
\begin{corollary}(Equivalence of $\Gamma$-convergence)\label{cor:equivalenceGamma} Let $\mathcal{L}$ be an admissible periodic set of points and let $F_\varepsilon$ be defined in \eqref{def:Feps} with $\mathcal{L}$ in place of $\Sigma$. Let $A \subset \mathbb{R}^n$ be a bounded open set such that $\partial A \in \mathrm{Lip}$. Then the following statements are equivalent:
\begin{itemize}
\item[(i)] $F_\varepsilon \colon \mathcal{M}_+(A) \to [0,+\infty]$ $\Gamma$-converges with respect to the weak* convergence of measures to the functional $F_{\mathrm{hom}} \colon \mathcal{M}_+(A) \to [0,+\infty]$ defined as
\begin{align*}
F_\mathrm{hom}(\mu,A):=\begin{cases} \int_{\partial^* V\cap A} \varphi_{\mathrm{hom}}(\nu)\,\mathrm{d}\mathcal{H}^{n-1} &\text{if } \mu = \rho \mathcal{L}^n\llcorner V\,;\\
+\infty &\text{otherwise.}
\end{cases}
\end{align*}
\item[(ii)] $F_\varepsilon \colon L^1_{\mathrm{loc}}(A) \to [0,+\infty]$ $\Gamma$-converges with respect to strong $L^1_{\mathrm{loc}}(A)$-topology to the functional $F_{\mathrm{hom}} \colon L^1_{\mathrm{loc}}(A) \to [0,+\infty]$ defined as
\begin{align*}
F_\mathrm{hom}(u,A):=\begin{cases} \int_{\partial^* V\cap A} \varphi_{\mathrm{hom}}(\nu)\,\mathrm{d}\mathcal{H}^{n-1} &\text{if } u=\chi_V \text{ and } \chi_V \in BV(A)\,;\\
+\infty &\text{otherwise.}
\end{cases}
\end{align*}
\end{itemize}
\end{corollary}
\begin{proof} The statement follows directly from Lemma \ref{lemma:equivalenceoftopgeneral}.
\end{proof}
\begin{remark}\label{rem:corollaryFhat} The analogous statements are true for $\hat F_\varepsilon$ as well.
\end{remark}

\EEE

\begin{theorem} \label{theorem:intrepperiodic}($\Gamma$-convergence for periodic admissible lattices) Let $\mathcal{L}$ be an admissible periodic set of points and let $F_\varepsilon$ be defined by \eqref{def:Feps} with $\mathcal{L}$ in place of $\Sigma$. Let $A \subset \mathbb{R}^n$ be bounded open set with $\partial A \in \mathrm{Lip}$ or $A=\mathbb{R}^n$. The functionals $F_\varepsilon$ $\Gamma$-converge with respect to weak* convergence of measures to the functional $F_{\mathrm{hom}} \colon \mathcal{M}_+(A) \to [0,+\infty]$ defined by  
\begin{align*}
F_\mathrm{hom}(\mu,A):=\begin{cases} \int_{\partial^* V\cap A} \varphi_{\mathrm{hom}}(\nu)\,\mathrm{d}\mathcal{H}^{n-1} &\text{if } \mu = \rho \mathcal{L}^n\llcorner V\,;\\
+\infty &\text{otherwise.}
\end{cases}
\end{align*}
The function $\varphi_{\mathrm{hom}} \colon \mathbb{R}^n \to [0,+\infty]$ is given by
\begin{align*}
\varphi_{\mathrm{hom}}(\nu) =\lim_{T\to +\infty} \frac{1}{T^{n-1}} \inf\left\{F(X,Q_T^\nu) \colon X\subset \mathcal{L}, \chi_X(i) = u_\nu(i) \text{ for } i \in \mathcal{L} \setminus Q^\nu_{T-l_T}\right\}\,,
\end{align*}
where $l_T \to +\infty$ and $l_T/T \to 0$ as $T\to+\infty$.
\end{theorem} 
\begin{proof}  \GGG 
\begin{step}{1}(Probabilistic setup)
We exploit the integral representation result  (Theorem \ref{theorem:AliCicRuf}) to obtain the specific form of the $\Gamma$-limit. We fix $(\Omega, \mathcal{F},\mathbb{P}) = (\{0\},\{\{0\},\emptyset\},\delta_0)$ to be a probability space and a trivial additive and ergodic group action (see Definition 5.1 in \cite{AliCicRuf}) $\tau_z \colon \{0\}\to \{0\}, z \in \mathbb{Z}^3$ given by $\tau_z(0)=0$. With respect to this group action $\mathcal{L}(0) = \mathcal{L}$ is an admissible stationary and ergodic stochastic lattice according to Definition \ref{def:statergod}. In fact, since $\mathcal{L}$ is periodic according to Definition \ref{def:periodicity},  for all $z=(z_1,z_2,z_3) \in \mathbb{Z}^3$ we have
\begin{align*}
\mathcal{L}(0)=\mathcal{L}\,, \quad \mathcal{L}(\tau_z(0)) =\mathcal{L}(0)= \mathcal{L}(0) +  \sum_{k=1}^n z_k e_k  \,.
\end{align*}
Therefore, all conditions of Theorem \ref{theorem:AliCicRuf} are satisfied. 
 This shows that for $\hat F_\varepsilon \colon L^1_{\mathrm{loc}}(A) \to [0,+\infty]$ we have
 \begin{align*}
 \Gamma\text{-}\lim_{\varepsilon\to0} \hat F_\varepsilon(\chi_V,A) = F_{\mathrm{hom}}(\chi_V,A) 
 \end{align*}
for all  $A \subset \mathbb{R}^n, \partial A \in \mathrm{Lip}, \chi_V \in BV(A,\{0,1\})$. Note that, by Corollary \ref{cor:equivalenceGamma}, this is equivalent to saying that for  $\hat F_\varepsilon \colon \mathcal{M}_+(A) \to [0,+\infty]$ we have
\begin{align*}
 \Gamma\text{-}\lim_{\varepsilon\to0} \hat F_\varepsilon(\mu,A) = F_{\mathrm{hom}}(\mu,A)\,. 
\end{align*} 
This concludes Step 1.
\end{step} \\
 \begin{step}{2}($\Gamma$-convergence of $F_\varepsilon$) \GGG
   We use the $\Gamma$-convergence of $\hat{F}_\varepsilon$ obtained in Step~1 in order to prove the $\Gamma$-convergence of $F_\varepsilon$. Let us first prove the result for $A \subset \mathbb{R}^n$ open and bounded with $\partial A \in \mathrm{Lip}$. We note that
\begin{align*}
\hat F_{\varepsilon}(X,A) \leq  F_{\varepsilon}(X,A)\,
\end{align*}
and therefore for all $\mu_\varepsilon \overset{*}{\rightharpoonup} \mu $, we have
\begin{align*}
\liminf_{\varepsilon \to 0}F_{\varepsilon}(\mu_\varepsilon,A) \geq \liminf_{\varepsilon \to 0}\hat F_{\varepsilon}(\mu_\varepsilon,A) \geq F_{\mathrm{hom}}(\mu,A)\,.
\end{align*}
Therefore
$
\Gamma\text{-}\liminf_{\varepsilon\to 0} F_\varepsilon(\mu,A) \geq F_{\mathrm{hom}}(\mu,A)\,.
$
Next, we show the Limsup-inequality. Let  $V\subset A$ be such that 
\begin{align}\label{ineq:Fhomfinite}
F_{\mathrm{hom}}(\rho\chi_V,A) <+\infty\,
\end{align}
since otherwise there is nothing to prove. Using Lemma \ref{lemma:propertiesadmissiblesets}(iv), we obtain that $\mathrm{Per}(V,A) <+\infty$, i.e., $V$ is a set of finite perimeter. In particular, by \cite[Theorem 1.3]{GarciaBravo-Rajala:2021} (noting that $\partial A \in \mathrm{Lip}$ implies that $A$ is an extension domain), there exists an extension $V_\mathrm{e}\subset \mathbb{R}^n$ such that 
\begin{itemize}
\item[(i)] $V_{\mathrm{e}}\cap A=V\cap A$ up to a set of zero $\mathcal{L}^n$-measure;\medskip
\item[(ii)] 
$
\mathrm{Per}(V_{\mathrm{e}},\mathbb{R}^n) \leq C\mathrm{Per}(V,A)\,;
$\medskip
\item[(iii)]
$
\mathcal{H}^{d-1}(\partial^*V_{\mathrm{e}}\cap \partial A) =0\,.
$
\end{itemize}
Now (ii) together with \eqref{ineq:Fhomfinite}  implies, again by Lemma \ref{lemma:propertiesadmissiblesets}(iv), that 
\begin{align*}
F_{\mathrm{hom}}(\rho\chi_{V_{\mathrm{e}}},\mathbb{R}^n) = \int_{\partial^* V_{\mathrm{e}}} \varphi_{\mathrm{hom}}(\nu)\,\mathrm{d}\mathcal{H}^{n-1} <+\infty\,.
\end{align*}
Fix $\delta >0$ and let $\mu_\varepsilon^\delta \to \rho \chi_{V_{\mathrm{e}}}$ weakly* in $\mathcal{M}_+((A)_\delta)$ be such that
\begin{align*}
\limsup_{\varepsilon\to 0} \hat{F}_\varepsilon(\mu_\varepsilon^\delta,(A)_\delta)  \leq F_{\mathrm{hom}}(\rho\chi_{V_{\mathrm{e}}},(A)_\delta)\,.
\end{align*}
Note that, due to Lemma \ref{lemma:propertiesadmissiblesets}(ii), for every $\varepsilon < \delta/(2R)$ we have that
\begin{align*}
  F_{\varepsilon}(X,A)\leq \hat F_{\varepsilon}(X,(A)_\delta) \,.
\end{align*}
Therefore, recalling Definition \ref{def:Gamma-convergence}, we obtain,
\begin{align*}
\Gamma\text{-}\limsup_{\varepsilon \to 0} F_{\varepsilon}(\rho \chi_V,A)\leq \limsup_{\varepsilon \to 0} F_{\varepsilon}(\mu_\varepsilon^\delta,A)\leq \limsup_{\varepsilon \to 0} \hat F_{\varepsilon}(\mu_\varepsilon^\delta,(A)_\delta)\leq F_{\mathrm{hom}}(\rho \chi_{V_{\mathrm{e}}},(A)_\delta)\,.
\end{align*}
Sending $\delta \to 0$ we obtain
\begin{align*}
\Gamma\text{-}\limsup_{\varepsilon \to 0} F_{\varepsilon}(\rho \chi_{V},A) \leq F_{\mathrm{hom}}(\rho \chi_{V_{\mathrm{e}}},\overline{A})= F_{\mathrm{hom}}(\rho \chi_{V_{\mathrm{e}}},A)=F_{\mathrm{hom}}(\rho \chi_V,A)\,,
\end{align*}
where the last equality follows by properties (i) and (iii) of $V_{\mathrm{e}}$. This shows the desired integral representation for all $A \subset \mathbb{R}^n$ such that $\partial A \in \mathrm{Lip}$.
\end{step}\\
\begin{step}{3}(Integral representation on unbounded sets) 
 It remains to prove the integral representation of the $\Gamma$-limit for $\mathbb{R}^n$. The Liminf-inequality follows by monotonicity since for all $R>0$ and $X\subset \varepsilon
 \mathcal{L}$ we have
 \begin{align*}
 F_\varepsilon(X, B_R) \leq F_\varepsilon(X)
 \end{align*}
 and therefore, given $\mu_\varepsilon \to \rho \chi_V$ weakly* in $\mathcal{M}_+(\mathbb{R}^n)$, we have
 \begin{align*}
 \liminf_{\varepsilon \to 0 } F_\varepsilon(\mu_\varepsilon) \geq  \liminf_{\varepsilon \to 0 } F_\varepsilon(\mu_\varepsilon, B_R) \geq F_{\mathrm{hom}}(\rho \chi_V, B_R)\,.
 \end{align*}
The claim follows by taking the supremum over $R >0$. We now turn our attention to the Limsup inequality. We can assume without loss of generality that $V \subset \mathbb{R}^n$ is a set of finite perimeter and 
\begin{align}\label{ineq:Fhomlimsupunbounded}
C\mathrm{Per}(V,\mathbb{R}^n)\leq\int_{\partial^* V}\varphi(\nu)\,\mathrm{d}\mathcal{H}^{n-1}= F_{\mathrm{hom}}(\rho \chi_V,\mathbb{R}^n) <+\infty
\end{align} 
since otherwise there is nothing to prove. By the isoperimetric inequality there exists $C>0$ such that
\begin{align*}
\min\{|V|, |\mathbb{R}^n\setminus V| \}^{\frac{n-1}{n}} \leq C\mathrm{Per}(V,\mathbb{R}^n)\,.
\end{align*}
Without loss of generality we assume that $|V| <+\infty$. By the Fleming-Rishel formula we can find $\{R_k\}_k$ such that $R_k \to +\infty$ and
\begin{align}\label{ineq:smallness}
{\rm (i)}\quad \mathcal{H}^{n-1}(V\cap \partial B_{R_k})\leq \frac{1}{k}\quad \text{and}\quad {\rm (ii)}\quad  |V\cap B_{R_k}^c|\leq \frac{1}{k}\,.
\end{align}
We define $V_k = V\cap B_{R_k}$.
Then, thanks to \eqref{ineq:smallness}(ii), $\chi_{V_k} \to \chi_V $ in $L^1(\mathbb{R}^n)$ and thus also $ \rho \chi_{V_k} \to \rho\chi_V $ weakly* in $\mathcal{M}_+(\mathbb{R}^n)$. Furthermore, \eqref{ineq:smallness}{\rm (i)} implies that 
\begin{align*}
\lim_{k\to+\infty} F_{\mathrm{hom}}(\rho\chi_{V_k}) =F_{\mathrm{hom}}(\rho\chi_V)\,.
\end{align*}
It therefore suffices to construct the recovery sequence for $\rho\chi_{V_k}$. Let $S_k = R_k+2R$ and let $\mu_\varepsilon \to \rho_{V_k}$ in $L^1(B_S)$ such that
\begin{align} \label{ineq:recoveryBS}
\limsup_{\varepsilon \to 0} F_\varepsilon(\mu_\varepsilon,B_{S_k})\leq  F_{\mathrm{hom}}(\rho \chi_V,B_{S_k}) \,.
\end{align}
We modify $\mu_\varepsilon$ such that 
\begin{align}\label{ineq:comparison}
F_\varepsilon(\mu_\varepsilon) \leq F_\varepsilon(\mu_\varepsilon,B_{S_k}) +o(1)\,,
\end{align}
and this, by \eqref{ineq:recoveryBS}, proves the statement. By Lemma \ref{lemma:equivalenceoftopgeneral}, there exists $\eta_\varepsilon \to 0$ such that $|V_\varepsilon\cap A_{R_k,S}| \leq \eta_\varepsilon$ with $V_\varepsilon$ defined in \eqref{def:piecewiseconstantinterpolant}. Now, let us take $k \in \{ \lceil \frac{R_k}{\varepsilon} +R\rceil,\ldots, \lfloor \frac{S}{\varepsilon}-3R\rfloor\}=:\mathcal{K}_\varepsilon$. Noting that $\#\mathcal{K}_\varepsilon \approx \varepsilon^{-1}$, Lemma \ref{lemma:propertiesadmissiblesets}(i) implies that there exists $k_\varepsilon \in \mathcal{K}_\varepsilon$ such that
\begin{align*}
\omega_n \left(\frac{r}{2}\right)^n  \varepsilon^{n-1} \#(X_\varepsilon \cap A_{k_\varepsilon \varepsilon,(k_\varepsilon+3R)\varepsilon}) &\leq \omega_n \left(\frac{r}{2}\right)^n \varepsilon^n \sum_{k \in \mathcal{K}_\varepsilon}  \#(X_\varepsilon\cap A_{k \varepsilon,(k+3R)\varepsilon}) \\&\leq C |V_\varepsilon\cap A_{R_k,S}| \leq C\eta_\varepsilon\,.
\end{align*}
Here, we have used that for $k \in \mathcal{K}_\varepsilon$ fixed
\begin{align*}
A_{k\varepsilon,(k+3R)\varepsilon} \cap A_{j\varepsilon,(j+3R)\varepsilon}  \neq \emptyset 
\end{align*}
for finitely many indices $j$ independent of $\varepsilon$ (clearly for $j > k + 3R$ or $j < k-3R$ the intersection is empty).
We can therefore define $\hat X_\varepsilon = \mathcal{L} \cap B_{k_\varepsilon}$ and employ Lemma \ref{lemma:propertiesadmissiblesets}(ii) to obtain
\begin{align*}
F_\varepsilon(\hat{X}_\varepsilon) = F_\varepsilon(\hat{X}_\varepsilon,B_{k_\varepsilon}) &\leq F_\varepsilon(X_\varepsilon, B_{k_\varepsilon})+ \overline{c}\max_{x \in \mathcal{L}} \#\mathcal{NN}(x)\varepsilon^{n-1} \#(X_\varepsilon\cap A_{k_\varepsilon \varepsilon,(k_\varepsilon+3R)\varepsilon})\\&\leq F_\varepsilon(X_\varepsilon,B_S) +  C\eta_\varepsilon\,,
\end{align*}
which proves \eqref{ineq:comparison} and with it the claim.
\end{step}
 \EEE
\end{proof}

\EEE

In the following we show that for minimizing sequences we can improve Proposition \ref{prop:Compactnessgeneral} to obtain strong $L^1(\mathbb{R}^n)$ compactness. This implies that the rescaled empirical measures of such sequences converge to a  suitably normalized Wulff shape for the limiting perimeter energy.

\begin{lemma}\label{lemma:nucleation}(Nucleation Lemma) For every $(p_0,v_0) \in (0,+\infty)^2$ there exists $m=m(p_0,v_0)>0$ such that if $V \subset \mathbb{R}^n$, $\mathrm{Per}(V) \leq p_0$, $|V|\geq v_0$, then there exists $x\in \mathbb{R}^n$ such that
\begin{align*}
|V\cap B_1(x)| \geq m\,.
\end{align*}

\end{lemma}
\begin{proof} Assume that $|V\cap B_1(x)| <\delta \leq \omega_n/2$ for all $x \in \mathbb{R}^n$. Then, noting that $\chi_{B_1(x)}(y)=\chi_{B_1(y)}(x)$ for all $x,y \in \mathbb{R}^n$, we obtain
\begin{align}\label{eq:onenucleaction}
\omega_n |V| = \int_{x \in \mathbb{R}^n}\int_{y \in \mathbb{R}^n} \chi_V(x)\chi_{B_{1}(x)}(y)\,\mathrm{d}y\,\mathrm{d}x = \int_{x \in \mathbb{R}^n}\int_{y \in \mathbb{R}^n} \chi_{V\cap B_1(y)}(x)\,\mathrm{d}y\,\mathrm{d}x \,.
\end{align}
Due to the assumption we have $|V\cap B_1(x)| <\delta $ and therefore due to the relative isoperimetric inequality and Fubini's Theorem, we obtain
\begin{align}\label{eq:twonucleaction}
\begin{split}
\int_{x \in \mathbb{R}^n}\int_{y \in \mathbb{R}^n} \chi_{V\cap B_1(y)}(x)\,\mathrm{d}y\,\mathrm{d}x= \int_{y \in \mathbb{R}^n}|V\cap B_1(y)|\,\mathrm{d}y &\leq  \delta^{1/n}\int_{y \in \mathbb{R}^n} |V\cap B_1(x)|^{1-\frac{1}{n}}\,\mathrm{d}y  \\&\leq \frac{\delta^{1/n}}{C_B}\int_{y \in \mathbb{R}^n}  \mathcal{H}^{n-1}(\partial^* V\cap B_1(y))\,\mathrm{d}y\,,
\end{split}
\end{align}
where $C_B >0$ denotes the relative isoperimetric constant of $B_1$. Now, again due to Fubini's Theorem, we obtain
\begin{align*}
\int_{y \in \mathbb{R}^n}  \mathcal{H}^{n-1}(\partial^* V\cap B_1(y))\,\mathrm{d}y &= \int_{y \in \mathbb{R}^n}  \int_{x\in \mathbb{R}^n} \chi_{B_1(y)}(x) \, \mathrm{d}\mathcal{H}^{n-1}\lfloor_{\partial^* V}(x) \,\mathrm{d}y \\&= \int_{x \in \mathbb{R}^n}\int_{y \in \mathbb{R}^n} \chi_{B_1(x)}(y)\,\mathrm{d}y\, \mathrm{d}\mathcal{H}^{n-1}\lfloor_{\partial^* V}(x) = \omega_n \mathcal{H}^{n-1}(\partial^* V)
\end{align*}
This together with \eqref{eq:onenucleaction} and \eqref{eq:twonucleaction} leads to
\begin{align*}
 |V| \leq \frac{\delta^{1/n}}{C_B} \mathrm{Per}( V)\,.
\end{align*}
Then, we conclude that $\delta \geq \left(C_B|V|/\mathrm{Per}(V)  \right)^n$ and thus for any $m < \min \{\omega_n/2, \left( C_B v/p_0\right)^n\}$ the thesis holds.
\end{proof}
%

\begin{lemma}(Concentration Lemma)\label{lemma:concentration} \GGG Let $\{X_\varepsilon\}_\varepsilon$ and $\{n_\varepsilon\}_\varepsilon \subset \mathbb{N}$  be  such that for all $\varepsilon>0$ $X_\varepsilon \subset \varepsilon\mathcal{L}$, $\# X_\varepsilon =n_\varepsilon$, $n_\varepsilon \varepsilon^n \to \rho v$ as $\varepsilon \to 0$, and 
\begin{align*}
F_{\varepsilon}(X_\varepsilon)= \min_{\# X =n_\varepsilon} F_{\varepsilon}(X)\,.
\end{align*}
Then, there exists $\{\tau_\varepsilon\}_\varepsilon \subset \mathbb{R}^n$ such that $\tau_\varepsilon \in \varepsilon \, \mathrm{span}_{\mathbb{Z}}\{e_1,\ldots, e_n\}$, and $\mu_\varepsilon(\cdot-\tau_\varepsilon)  \rightharpoonup^* \rho \chi_{W}$, where $W$ is the Wulff shape of $\varphi_{\mathrm{hom}}$ defined in Theorem \ref{theorem:intrepperiodic} and such that $|W|= v$.
\end{lemma}
\begin{proof}  Let $\{X_\varepsilon\}_\varepsilon$ be as in the assumptions \GGG of \EEE the Lemma. \\
\begin{step}{1}\GGG(Energy bound) \EEE We first show that 
\begin{align}\label{ineq:perimeterbound}
\sup_{\varepsilon>0} F_{\varepsilon}(X_\varepsilon) <+\infty\,.
\end{align}
To this aim, for each $\varepsilon>0$ we construct a competitor $Y_\varepsilon$ such that $F_{\varepsilon}(Y_\varepsilon) \leq C$ for some constant $C>0$ independent of $\varepsilon$. To this end let $r_\varepsilon >0$ be the maximal $s>0$ such that
\begin{align*}
\# (B_s \cap \varepsilon \mathcal{L}) \leq n_\varepsilon\,.
\end{align*}
Due to Lemma \ref{lemma:propertiesadmissiblesets}{\rm (i)}, we have that $r_\varepsilon \leq r_0 <+\infty$ and
\begin{align}\label{ineq:layerbound}
\omega_n \left(\frac{r}{2}\right)^n \varepsilon^n\#\left(A_{r_\varepsilon, r_\varepsilon +2\varepsilon R}\cap \varepsilon\mathcal{L}\right) \leq |A_{r_\varepsilon-\varepsilon r/2,r_\varepsilon+3\varepsilon R}| \leq C\varepsilon  r_\varepsilon^{n-1}\leq C\varepsilon\,.
\end{align}
By the maximality in the choice of $r_\varepsilon$ we have $\#(B_{r_\varepsilon + \varepsilon r} \cap \varepsilon\mathcal{L}) > n_\varepsilon$ and thus 
\begin{align*}
n_\varepsilon -C\varepsilon^{1-n} \leq \#(B_{r_\varepsilon} \cap \varepsilon\mathcal{L}) \leq n_\varepsilon\,.
\end{align*}
We set $Y_\varepsilon = (B_{r_\varepsilon} \cap \varepsilon\mathcal{L}) \cup Y^0_\varepsilon$, where $Y_\varepsilon^0 \subset \varepsilon\mathcal{L}\setminus B_{r_\varepsilon} $ such that $\# Y_\varepsilon = n_\varepsilon$. Due to Lemma \ref{lemma:propertiesadmissiblesets} (ii), (iii), the minimality of $X_\varepsilon$, and \eqref{ineq:layerbound}, we have
\begin{align*}
\begin{split}
F_{\varepsilon}(X_\varepsilon)&\leq F_{\varepsilon}(Y_\varepsilon)\\&\leq \sum_{x \in \mathcal{L}_\varepsilon\cap B_{r_\varepsilon}} \varepsilon^{n-1}\#(\mathcal{NN}_\varepsilon(x)\setminus X) + C\varepsilon^{n-1}\#Y_\varepsilon^0 \leq C\#\left(A_{r_\varepsilon, r_\varepsilon +2\varepsilon R}\cap \mathcal{L}\right)\leq C\,.
\end{split}
\end{align*}
This shows \eqref{ineq:perimeterbound} and concludes Step 1. \end{step}\\
\begin{step}{2}(Nucleation) We show that if \eqref{ineq:perimeterbound} holds, then there exist $v_1>0$ and $\{\tau_\varepsilon^1\}_\varepsilon \subset \mathbb{R}^n$ such that $\tau_\varepsilon^1 \in \varepsilon \,\mathrm{span}_\mathbb{Z}\{e_1,\ldots,e_n\}$ and, up to subsequences,
\begin{align}\label{eq:convergence}
\mu_\varepsilon(\cdot-\tau_\varepsilon^1) \overset{*}{\rightharpoonup} \rho \chi_{V_1} \text{ in the sense of measures,}
\end{align}
where $ |V_1| := v_1 \in (m,v]$, with $m = m(p_0,v_0) >0$ given by Lemma \ref{lemma:nucleation}. In fact, due to \eqref{ineq:perimeterbound} and Lemma \ref{lemma:propertiesadmissiblesets}(iv) there exists $p_0 >0$ such that 
$
\mathrm{Per}(V_\varepsilon)\leq p_0$ for all $\varepsilon>0
$.
Additionally, due to Lemma \ref{lemma:propertiesadmissiblesets}(i) and the fact that $n_\varepsilon \varepsilon^n \to v>0$, we have that $|V_\varepsilon| \geq v_1 >0$ for all $\varepsilon >0$ small enough. Thus, by Lemma \ref{lemma:nucleation} there exists $m=m(p_0,v_1)$ and $\{\tau_\varepsilon^1\}_\varepsilon \subset \mathbb{R}^d$ such that
\begin{align}\label{ineq:massball}
|V_\varepsilon \cap B_1(\tau_\varepsilon^1)| \geq m\,.
\end{align}
Actually, by lowering $m$ a bit we can assume without loss of generality that $\tau_\varepsilon^1 \in \varepsilon\, \mathrm{span}_\mathbb{Z}\{e_1,\ldots,e_n\}$. By \eqref{ineq:massball}, \eqref{ineq:perimeterbound}, Proposition \ref{prop:Compactnessgeneral}, and Lemma \ref{lemma:equivalenceoftopgeneral}, we get \eqref{eq:convergence} up to passing to a subsequence.
\end{step}\\ 
\begin{step}{3}(Splitting of the energy) We show that for any $0<\delta <|V_1|$ there exists $S_\delta=S_\delta(V_1) >0$ big enough and $S_\varepsilon\in (S_\delta, S_\delta +1)$  such that, for $\varepsilon>0$ small enough, there holds
\begin{align}\label{ineq:cardstrip}
\#(X_\varepsilon \cap B_{S_\delta}) \geq  \omega_n^{-1}R^{-n}   (|V_1|-\delta)\varepsilon^{-n}  \quad \text{and} \quad \#(X_\varepsilon \cap A_{S_\varepsilon,S_\varepsilon+3\varepsilon R}) \leq  \delta\varepsilon^{1-n} \,.
\end{align}
First of all,   we find $S_\delta$ such that for $\varepsilon>0$ small enough
\begin{align*}
|V_1 \cap B_{S_\delta-R\varepsilon}|\geq  |V_1 \cap B_{S_\delta-1}| \geq |V_1| -\delta/2 \text{ and } |V_1 \cap A_{S_\delta, S_\delta +1}| \leq |V_1 \setminus B_{S_\delta -R\varepsilon}| \leq \delta/2 \,.
\end{align*}
 Note that, due to Step~1 and Step~2, $\chi_{V_\varepsilon(\cdot-\tau_\varepsilon^1)} \to \chi_{V_1}$ in $L^1(B_{S_\delta+1})$. Due to Lemma \ref{lemma:propertiesadmissiblesets}(i), we have for $\varepsilon>0$ small enough
\begin{align}\label{ineq:cardstrip2}
&\omega_n \varepsilon^n R^n\#(X_\varepsilon \cap B_{S_\delta})\geq |V_\varepsilon \cap B_{S_\delta-R\varepsilon}|\geq  |V_1|-\delta\,; \\& \label{ineq:cardstrip22}
\omega_n \varepsilon^n \left(\frac{r}{2}\right)^n\#(X_\varepsilon \cap A_{S_\delta+R\varepsilon,S_\delta+1-R\varepsilon})\leq |V_\varepsilon \cap A_{S_\delta,S_\delta+1}|\leq  \delta\,.
\end{align}
 Next, we find $S_\varepsilon$ by averaging: We choose $k_\varepsilon \in \left\{\lceil \tfrac{S}{\varepsilon}+R\rceil,\ldots, \lfloor \tfrac{S+1}{\varepsilon}-4R\rfloor \right\} =:\mathcal{K}_\varepsilon$ (note that $\mathcal{K}_\varepsilon \neq \emptyset$ and actually $\#\mathcal{K}_\varepsilon \approx \tfrac{S}{\varepsilon} $ for $\varepsilon>0$ small enough) such that for some $C>1$
\begin{align}\label{ineq:Cdelta}
\begin{split}
\varepsilon^{n-1}\#(X_\varepsilon \cap A_{k_\varepsilon \varepsilon, (k_\varepsilon +3R)\varepsilon}) &\leq C\sum_{k \in \mathcal{K}_\varepsilon} \varepsilon^n \#(X_\varepsilon \cap A_{k \varepsilon, (k +3R)\varepsilon})\\&\leq C\varepsilon^n \#(X_\varepsilon \cap A_{S-R\varepsilon,S+1-R\varepsilon}) \leq C\delta\,.
\end{split}
\end{align}
In the latter estimate, we have used \eqref{ineq:cardstrip22} and that, for $k \in \mathcal{K}_\varepsilon$ fixed, we have
\begin{align*}
A_{k\varepsilon,(k+3R)\varepsilon} \cap A_{j\varepsilon,(j+3R)\varepsilon}  \neq \emptyset 
\end{align*}
for at most a finite number of indices $j$ independent of $\varepsilon$ (clearly for $j > k + 3R$ or $j < k-3R$ the intersection is empty). Up to replacing $\delta$ by $C^{-1}\delta$, \eqref{ineq:cardstrip2} and \eqref{ineq:Cdelta} give \eqref{ineq:cardstrip} and this concludes Step~3.
\end{step}\\
In the next steps we use the following notation.
For any $\lambda >0$ let $X_\varepsilon^\lambda \subset (W_\lambda)_{1}$ be such that, setting $V_\varepsilon^\lambda= V_\varepsilon(X_\varepsilon^{\lambda})$, we have $\chi_{V_\varepsilon^\lambda} \to \chi_{W_\lambda}$ in $L^1(\mathbb{R}^n)$, and
\begin{align}\label{ineq:optimizer}
\lim_{\varepsilon \to 0} F_\varepsilon(X_\varepsilon^\lambda)=\limsup_{\varepsilon \to 0} F_\varepsilon(X_\varepsilon^\lambda) = \limsup_{\varepsilon \to 0} F_\varepsilon(X_\varepsilon^\lambda, (W_\lambda)_{1}) = F_{\mathrm{hom}}(\GGG\rho\EEE\chi_{W_\lambda})\,.
\end{align}
\begin{step}{4}(Identification of $V_1$) We claim that $V_1= W_{v_1}$ up to null sets, where $W_{v_1}$ is the Wulff Shape of $F_{\mathrm{hom}}$ such that $|W_{v_1}|=v_1$. Assume by contradiction that this were not the case. By the anisotropic isoperimetric inequality we have that \GGG
\begin{align*}
0< \frac{1}{2}\left( F_{\mathrm{hom}}(\rho\chi_{V_1})- F_{\mathrm{hom}}(\rho\chi_{W_{v_1}})\right)=: \eta\,.
\end{align*} \EEE
Up to translating, we may assume that $\tau_\varepsilon^1=0$. Then, there exists $S_0>0$  big enough such that for all $S \geq S_0$ we have $W_{v_1} \subset\subset B_S$ and
\begin{align*}
F_{\mathrm{hom}}(\rho\chi_{V_1},B_S) \geq F_{\mathrm{hom}}(\rho\chi_{V_1}) -\eta\,.
\end{align*}
By Theorem \ref{theorem:intrepperiodic}, we have that
\begin{align}\label{ineq:liminfBallXeps}
\begin{split}
\liminf_{\varepsilon \to 0} F_\varepsilon(X_\varepsilon, B_S) - F_{\mathrm{hom}}(\rho\chi_{W_{v_1}},B_S) &\geq F_{\mathrm{hom}}(\rho\chi_{V_1},B_S)- F_{\mathrm{hom}}(\rho\chi_{W_{v_1}},B_S) \\&\geq F_{\mathrm{hom}}(\rho\chi_{V_1})-\eta- F_{\mathrm{hom}}(\rho\chi_{W_{v_1}})\geq \eta >0\,. 
\end{split}
\end{align}
We now construct a competitor $\hat{X}_\varepsilon$  for $\varepsilon>0$ small enough such that $\#X_\varepsilon=n_\varepsilon$ and
\begin{align}\label{ineq:contradictionStep4}
F_\varepsilon(\hat X_\varepsilon) < F_\varepsilon(X_\varepsilon)\,.
\end{align}
This contradicts the assumptions of the Lemma on the minimality of $X_\varepsilon$ and therefore $V_1= W_{v_1}$ up to null sets. 
Let us take $\delta :=\delta_\eta = (2C \max_{x \in \mathcal{L}} \#\mathcal{NN}(x))^{-1}\eta $ with $C$ as in \eqref{ass:boundcnn}  and let $S_\delta$, $S_\varepsilon \in (S_\delta,S_\delta+1)$ be as in Step~3 with $S_\delta \geq \mathrm{diam}\, W_{v_1} +1$. We define 
\begin{align}\label{def:Xhat}
\hat{X}_\varepsilon := X_\varepsilon^{v_1} \cup (X_{\varepsilon} \setminus B_{S_\varepsilon}) \cup Z_\varepsilon\,,
\end{align}
where $\#\hat{X}_\varepsilon  = \# X_\varepsilon$. In what follows we assume for simplicity that $Z_\varepsilon =\emptyset$. Otherwise, one choose $Z_\varepsilon$ such that its contributions to both energy and volume are negligible as $\varepsilon \to 0$.\EEE First of all note
\begin{align}\label{eq:split}
F_\varepsilon(\hat{X}_\varepsilon) = F_\varepsilon(\hat{X}_\varepsilon, B_{S_\varepsilon}) + F_\varepsilon(\hat{X}_\varepsilon, A_{S_\varepsilon+3\varepsilon R}) +  F_\varepsilon(\hat{X}_\varepsilon, B_{S_\varepsilon+3\varepsilon R}^c)\,.
\end{align}
Then, by noting first that $X_\varepsilon = \hat X_\varepsilon$ on  $(B_{S_\varepsilon+3\varepsilon R}^c)_{\varepsilon R}$, we have
\begin{align}\label{eq:outside}
 F_\varepsilon(X_\varepsilon, B_{S_\varepsilon+3\varepsilon R}^c) =  F_\varepsilon(\hat X_\varepsilon, B_{S_\varepsilon+3\varepsilon R}^c)\,.
\end{align}
Furthermore, 
\begin{align}\label{eq:inside}
F_\varepsilon(\hat{X}_\varepsilon, B_{S_\varepsilon}) = F_\varepsilon(X_\varepsilon^{v_1}, B_{S_\varepsilon}) = F_\varepsilon(X_\varepsilon^{v_1})\,. 
\end{align}
Lastly, due to Lemma \ref{lemma:propertiesadmissiblesets}(iii) and the choice of $\delta$, and \eqref{ineq:cardstrip}, we have
\begin{align}\label{ineq:annulus}
F_\varepsilon(\hat{X}_\varepsilon, A_{S_\varepsilon, S_\varepsilon+3\varepsilon R}) \leq C\max_{x \in \mathcal{L}} \#\mathcal{NN}(x) \,  \varepsilon^{n-1} \#(X_\varepsilon \cap A_{S_\varepsilon,S_\varepsilon+3\varepsilon R}) \leq \eta/2\,.
\end{align}
Comparing this to the energy of $X_\varepsilon$, also noting that $S_\delta \geq \mathrm{diam}\, W_{v_1} +1$ and, using \eqref{eq:split}--\eqref{ineq:annulus}, we obtain
\begin{align*} 
\begin{split}
F_\varepsilon(X_\varepsilon) &= F_\varepsilon(X_\varepsilon, B_{S_\varepsilon}) + F_\varepsilon(X_\varepsilon, A_{S_\varepsilon+3\varepsilon R}) +  F_\varepsilon(X_\varepsilon, B_{S_\varepsilon+3\varepsilon R}^c) \\&\geq F_\varepsilon(X_\varepsilon, B_{S_\delta})  + F_\varepsilon(X_\varepsilon, B_{S_\varepsilon+3\varepsilon R}^c)  \geq F_\varepsilon(\hat{X}_\varepsilon) -\eta/2 +F_\varepsilon(X_\varepsilon, B_{S_\delta})  - F_\varepsilon( X_\varepsilon^{v_1}) 
\end{split}
\end{align*}
Therefore, using \eqref{ineq:optimizer} and \eqref{ineq:liminfBallXeps}, we obtain
\begin{align*}
\liminf_{\varepsilon \to 0} ( F_\varepsilon(X_\varepsilon)  -  F_\varepsilon(\hat X_\varepsilon)) \geq \eta/2>0\,.
\end{align*}
This yields \eqref{ineq:contradictionStep4} for $\varepsilon>0$ small enough.
\end{step}\\
\begin{step}{5}($v_1 = v$) Assume by contradiction that $v_1 < v$. We repeat Step~2 and Step~3 for $X_\varepsilon \setminus (W_{v_1}(\tau_\varepsilon^1))_1$ to find $\tau_\varepsilon^2$ such that 
\begin{align}
\mu_\varepsilon(\cdot-\tau_\varepsilon^2) \overset{*}{\rightharpoonup} \rho\chi_{V_2} \text{ in the sense of measures,}
\end{align}
where $|V_2|=v_2 >0$. Note that we can assume $|\tau_\varepsilon^1 - \tau_\varepsilon^2| \to +\infty$ since this would otherwise contradict $\mu_\varepsilon(\cdot-\tau_\varepsilon^1) \overset{*}{\rightharpoonup} \rho \chi_V$. By Step~4, we observe that $V_2=W_{v_2}$. Note that $v \mapsto F_{\mathrm{hom}}(W_v)$ is strictly concave in $v$. In fact there holds
\begin{align}\label{eq:homogeneity}
F_{\mathrm{hom}}(W_v) = v^{\frac{n-1}{n}}
F_{\mathrm{hom}}(W_1)\,.
\end{align}
Therefore, given $v_1,v_2 >0$, for $v=\lambda v_1 + (1-\lambda)v_2$ for $\lambda \in (0,1)$, using \eqref{eq:homogeneity} and the strict concavity of the function $v\mapsto v^{\frac{n-1}{n}}$, we have
\begin{align*}
F_{\mathrm{hom}}(W_v)&= F_{\mathrm{hom}}(W_{ \lambda v_1 +(1-\lambda) v_2}) = (\lambda v_1 +(1-\lambda)v_2)^{\frac{n-1}{n}}
F_{\mathrm{hom}}(W_1) \\&> \left(\lambda v_1^{\frac{n-1}{n}} +(1-\lambda) v_2^{\frac{n-1}{n}}\right)F_{\mathrm{hom}}(W_1) = \lambda F_{\mathrm{hom}}(W_{v_1}) +(1-\lambda)F_{\mathrm{hom}}(W_{v_2})\,.
\end{align*}
Since $F_{\mathrm{hom}}(W_0)=0$, this in  particular implies that $v\mapsto F_{\mathrm{hom}}(W_v)$  is strictly subbaditive, i.e.,
\begin{align*}
F_{\mathrm{hom}}(W_{v_1+v_2}) < F_{\mathrm{hom}}(W_{v_1})+ F_{\mathrm{hom}}(W_{v_2})
\end{align*}
for all $v_1,v_2 >0$. Set
\begin{align*}
\eta:=F_{\mathrm{hom}}(W_{v_1})+ F_{\mathrm{hom}}(W_{v_2})-F_{\mathrm{hom}}(W_{v_1+v_2})>0\,.
\end{align*}
Due to Step~3 applied with $\delta = \delta_\eta:=(3C \max_{x \in \mathcal{L}} \#\mathcal{NN}(x) )^{-1}\eta $ and Step~4 we have that 
\begin{align*}
\mu_\varepsilon(\cdot-\tau_\varepsilon^1) \overset{*}{\rightharpoonup} \rho \chi_{V_1} \quad \text{and} \quad \mu_\varepsilon(\cdot-\tau^2_\varepsilon) \overset{*}{\rightharpoonup} \rho \chi_{V_2} 
\end{align*}
and 
\begin{align*}
\#(X_\varepsilon \cap A_{S_\varepsilon,S_\varepsilon+3\varepsilon R}(\tau_\varepsilon^1)) \leq  \delta_\eta\varepsilon^{1-n}  \quad \text{and} \quad \#(X_\varepsilon \cap A_{\tilde{S}_\varepsilon,\tilde{S}_\varepsilon+3\varepsilon R}(\tau_\varepsilon^2)) \leq  \delta_\eta\varepsilon^{1-n} \,,
\end{align*}
where $\tilde{S}_\varepsilon$ is associated with $\tau_\varepsilon^2$ by following the same procedure to $X_\varepsilon \setminus B_{S_\varepsilon}(\tau_\varepsilon^1)$ as in Step~3.
We can therefore define
\begin{align*}
\tilde{X}_\varepsilon = X^{v_1+v_2}_\varepsilon \cup (X_\varepsilon \setminus (B_{S_\varepsilon}(\tau_\varepsilon^1)\cup B_{\tilde{S}_\varepsilon}(\tau_\varepsilon^2)))\,,
\end{align*}
and without loss of generality (see Step~4) we can directly assume $\#\tilde{X}_\varepsilon = \# X_\varepsilon$. 
Now the argument follows very much in the spirit of Step~4. We first observe that
\begin{align}\label{eq:split2}
\begin{split}
F_\varepsilon(\tilde{X}_\varepsilon) &= F_\varepsilon(\tilde{X}_\varepsilon, B_{S_\varepsilon}(\tau_\varepsilon^1)) + F_\varepsilon(\tilde{X}_\varepsilon, B_{\tilde{S}_\varepsilon}(\tau^2_\varepsilon)) + F_\varepsilon(\tilde{X}_\varepsilon, A_{S_\varepsilon+3\varepsilon R}(\tau_\varepsilon^1)) \\ &\quad+ F_\varepsilon(\tilde{X}_\varepsilon, A_{S_\varepsilon+3\varepsilon R}(\tau_\varepsilon^2)) +  F_\varepsilon(\tilde{X}_\varepsilon, (B_{S_\varepsilon+3\varepsilon R}(\tau_\varepsilon^1) \cup  B_{\tilde{S}_\varepsilon+3\varepsilon R}(\tau_\varepsilon^2))^c)\,.
\end{split}
\end{align}
Then, by noting first that $X_\varepsilon = \tilde X_\varepsilon$ on  $((B_{S_\varepsilon+3\varepsilon R}(\tau_\varepsilon^1) \cup  B_{S^2_\varepsilon+3\varepsilon R}(\tau_\varepsilon^2))^c)_{\varepsilon R}$, we have
\begin{align}\label{eq:outside2}
 F_\varepsilon(X_\varepsilon, ((B_{S_\varepsilon+3\varepsilon R}(\tau_\varepsilon^1) \cup  B_{S^2_\varepsilon+3\varepsilon R}(\tau_\varepsilon^2))^c)_{\varepsilon R}) =  F_\varepsilon(\tilde X_\varepsilon, ((B_{S_\varepsilon+3\varepsilon R}(\tau_\varepsilon^1) \cup  B_{S^2_\varepsilon+3\varepsilon R}(\tau_\varepsilon^2))^c)_{\varepsilon R})\,.
\end{align}
Furthermore, 
\begin{align}\label{eq:inside2}
&F_\varepsilon(\tilde{X}_\varepsilon, B_{S_\varepsilon}(\tau_\varepsilon^1))= F_\varepsilon(X_\varepsilon^{v_1}) \quad \text{and}\quad F_\varepsilon(\tilde{X}_\varepsilon, B_{\tilde{S}_\varepsilon}(\tau_\varepsilon^2))  = F_\varepsilon(X_\varepsilon^{v_2}) \,. 
\end{align}
Lastly, due to Lemma \ref{lemma:propertiesadmissiblesets}(iii) and the choice of $\delta$, we have
\begin{align}\label{ineq:annulus2}
F_\varepsilon(\tilde{X}_\varepsilon, A_{S_\varepsilon, S_\varepsilon+3\varepsilon R}(\tau_\varepsilon^1)) \leq C\max_{x \in \mathcal{L}} \#\mathcal{NN}(x) \, \varepsilon^{n-1} \#(X_\varepsilon \cap A_{S_\varepsilon,S_\varepsilon+3\varepsilon R}(\tau_\varepsilon^1)) \leq \eta/3
\end{align}
and
\begin{align*}
F_\varepsilon(\tilde{X}_\varepsilon, A_{\tilde{S}_\varepsilon, \tilde{S}_\varepsilon+3\varepsilon R}(\tau_\varepsilon^2)) \leq C\max_{x \in \mathcal{L}} \#\mathcal{NN}(x) \, \varepsilon^{n-1} \#(X_\varepsilon \cap A_{\tilde{S}_\varepsilon,\tilde{S}_\varepsilon+3\varepsilon R}(\tau_\varepsilon^2)) \leq \eta/3\,.
\end{align*}
Comparing this to the energy of $X_\varepsilon$ we obtain, using \eqref{eq:split2}--\eqref{ineq:annulus2}, as in Step~5,
\begin{align*}
\liminf_{\varepsilon\to 0} (F_\varepsilon(X_\varepsilon)-F_\varepsilon(\tilde{X}_\varepsilon))\geq \eta/3>0\,.
\end{align*}
This is a contradiction and therefore $v_1=v$.
\end{step}
Setting $\tau_\varepsilon :=\tau_\varepsilon^1$ this concludes the proof.
\end{proof}

\begin{remark} We want to observe that Lemma \ref{lemma:concentration} can be extended to the setting of \cite{AliCicRuf} in which the functional $F_\varepsilon$ also accounts for long-range interactions.  In order to adapt the proof to the general case, the annulus $A_{r_\varepsilon,r_\varepsilon+3R \varepsilon}$ must be replaced by $A_{r_\varepsilon,r_\varepsilon+s_\varepsilon}$ where  $s_\varepsilon=k_\varepsilon \varepsilon$ (here $k_\varepsilon$ is such that $k_\varepsilon \to +\infty$ and $k_\varepsilon\varepsilon \to 0$). This choice ensures that  $\hat X_\varepsilon \cap B_{r_\varepsilon}$ and $\hat X_\varepsilon \setminus  B_{r_\varepsilon+s_\varepsilon}$ (resp. $ X_\varepsilon \cap B_{r_\varepsilon}$ and $ X_\varepsilon \setminus  B_{r_\varepsilon+s_\varepsilon}$) are sufficiently distant such that the energy contribution that accounts for the interactions crossing the annulus are negligible as $\varepsilon \to 0$.
\end{remark}

\section*{Acknowledgements} 
This work was supported by the DFG project FR 4083/1-1 and by the Deutsche Forschungsgemeinschaft (DFG, German Research Foundation) under Germany's Excellence Strategy EXC 2044 -390685587, Mathematics M\"unster: Dynamics--Geometry--Structure. The work of M. Cicalese was supported by the DFG Collaborative Research Center TRR 109, ''Discretization in Geometry and Dynamics''. The first and second author gratefully acknowledge the GNAMPA visiting professor program 2022 and the CIRM research in pairs program, Trento 2020, during which parts of the project have been carried out.

\end{document}